\documentclass{imsart}
\usepackage[pdfencoding=auto]{hyperref}

\usepackage{braket,amsfonts}
\usepackage{array}
\usepackage{pgfplots}
\pgfplotsset{compat=newest}
\usepackage[caption=false]{subfig}
\usepackage{mathtools}
\usepackage{eqparbox}
\usepackage{enumitem}
\usepackage{etoolbox}

\AtEndEnvironment{problem}{\null\hfill\proofbox}

\usepackage{algorithm}

\usepackage{graphicx,epstopdf}
\usepackage{float}

\usepackage{amsopn}

\usepackage{verbatim}

\usepackage{ulem}

\usepackage{lineno}

\usepackage{amsthm}

\usepackage{dsfont} 

\reversemarginpar
\usepackage[colorinlistoftodos,prependcaption,textsize=tiny]{todonotes}

\startlocaldefs
\newtheorem{theorem}{Theorem}
\newtheorem{lemma}{Lemma}
\newtheorem{corollary}{Corollary}
\newtheorem{assumptions}{Assumption}
\newtheorem{proposition}{Proposition}

\newtheorem{remark}{Remark}
\newtheorem{expl}{Example}
\newtheorem{mainresult}{Main Result}

\newtheorem{definition}{Definition}

\AtEndEnvironment{expl}{\null\hfill$\Diamond$}

\newcommand{\lvertiii}{{\vert\kern-0.25ex\vert\kern-0.25ex\vert}}    
\newcommand{\rvertiii}{{\vert\kern-0.25ex\vert\kern-0.25ex\vert}}

\newcommand{\mcl}{\mathcal}
\newcommand{\mbf}{\mathbf}
\newcommand{\mbb}{\mathbb}

\newcommand{\dd}{\text{d}}

\newcommand{\Law}{{\rm Law}}

\usepackage{mhequ}

\newcommand{\be}{\begin{equs}}
\newcommand{\ee}{\end{equs}}
\newcommand{\bpm}{\begin{pmatrix}}
\newcommand{\epm}{\end{pmatrix}}
\DeclareMathOperator{\E}{\mathbb E}

\renewcommand{\P}{\mathcal P}
\newcommand{\Q}{\mathcal{Q}}
\newcommand{\K}{\mathcal{K}}
\renewcommand{\H}{\mathcal{H}}
\newcommand{\X}{\mathcal{X}}

\newcommand{\bb}[1]{\mathbb #1}

\DeclareMathOperator{\N}{N}

\newcommand{\mc}[1]{\mathcal{#1}}
\newcommand{\weakto}{\rightharpoonup}
\newcommand{\TT}{\mathbb{T}}

\newcommand{\Beta}{\text{Beta}}
\newcommand{\eps}{\epsilon}
\newcommand{\veps}{\varepsilon}

\newcommand\iidsim{\stackrel{\tiny \mathclap{\mbox{ iid}}}{\sim}}
\newcommand{\tP}{P^1}
\newcommand{\td}{\tilde{d}}
\newcommand{\hd}{\tilde{d}}

\usepackage{soul}

\endlocaldefs

\begin{document}
\allowdisplaybreaks

\begin{frontmatter}
 
  \title{Spectral gaps and error estimates for
    infinite-dimensional Metropolis-Hastings
     with non-Gaussian priors}

\runtitle{Spectral gaps and error estimates for MH}

\author{\fnms{Bamdad} \snm{Hosseini}\ead[label=e1]{bamdadh@uw.edu}}
\address{University of Washington, Seattle, WA}
\and
\author{\fnms{James E.} \snm{Johndrow}\ead[label=e2]{johndrow@wharton.upenn.edu}}
\address{University of Pennsylvania, Philadelphia, PA}
\affiliation{University of Washington and University of Pennsylvania}

\runauthor{Hosseini and Johndrow}




\begin{abstract}
We study a class of Metropolis-Hastings algorithms for
  target measures that are absolutely continuous with respect to 
  a large class of non-Gaussian prior measures on Banach spaces. The algorithm is shown to have a
  spectral gap in a Wasserstein-like semimetric weighted by a Lyapunov function. A number of
  error bounds are given for computationally tractable approximations of the algorithm
  including bounds on the closeness of Ces\'{a}ro averages and other pathwise quantities
  via perturbation theory. Several applications illustrate the breadth of problems to which the
  results apply such as various likelihood approximations 
  and perturbations of prior measures.
\end{abstract}



\end{frontmatter}


\section{Introduction}\label{sec:introduction}
The goal of this article is to study convergence rates and stability to perturbations
of a class of Metropolis-Hastings (MH) 
algorithms 
for sampling target measures that are absolutely continuous 
with respect to an underlying non-Gaussian measure. Targets in this class naturally arise
as  posterior measures in Bayesian inverse problems with non-Gaussian priors. We show that under 
general conditions, the algorithms of interest to us have a dimension-independent
spectral gap with respect to
a transport semimetric on the space of probability measures.
Furthermore, we present a general perturbation result 
stating that the invariant measure of the algorithm depends continuously on
perturbations of the proposal kernel and acceptance ratio. We also give bounds on the closeness
of Ces\'{a}ro averages and other pathwise quantities from the perturbed transition kernel.

  Let $\mcl{H}$ be a separable Banach  space with norm $\| \cdot \|$ and
$P(\H)$ denote the space of Radon  probability measures on $\mcl{H}$,
assigning measure 1 to the whole space.
Consider $\mu, \nu \in P(\mcl{H})$
 satisfying
\begin{equation} \label{nu-definition}
  \frac{\dd \nu}{\dd \mu}(u) = \frac{1}{Z} \exp( - \Psi(u)), \qquad u \in \mcl{H},
\end{equation}
where $\Psi: \mcl{H} \mapsto \mbb{R}$ is a measurable function and
$Z = \mu (\exp(-\Psi))$ is a normalizing 
constant, and for any measure $\mu$ and function $\varphi$, 
$\mu( \varphi) := \int_{\mc H} \varphi(u) \mu(du)$. 
In Bayesian inference, the measure $\nu$ is precisely the posterior measure, which is absolutely continuous with respect to the prior measure $\mu$. 
In applications such as Bayesian inverse problems and uncertainty quantification, 
our goal is often to estimate integrals of the form $\nu(\varphi)$
for a function of interest $\varphi : \H \to \X$ where $\X$ is  a separable Hilbert space with norm $\| \cdot \|_\X$.
Since this integral is often intractable we
approximate it using 
$n^{-1} \sum_{k=1}^n \varphi(U_k)$,
where the sequence $\{U_k\}_{k=1}^n$ are distributed according to $\nu$ as $n \to \infty$. 

We are primarily interested in the setting where we cannot sample from $\nu$ directly but we can 
sample from $\mu$. Then an algorithm is needed that can approximately sample $\nu$. A common approach constructs a Markov transition operator $\P$ with invariant measure $\nu$, then collects paths
$U_k \sim \P^{k-1} \delta_{u_0}$ starting from a fixed initial condition $U_0 = u_0$. Because it is not possible to simulate paths numerically on an infinite-dimensional state space, in practice finite-dimensional approximations to the exact algorithm are used. A 
well-known example of such an algorithm is the preconditioned Crank-Nicholson (pCN) algorithm \cite{stuart-mcmc}; a MH algorithm for $\mu$ that is absolutely continuous with respect to a Gaussian measure. In this article, we consider a  generalization of the pCN  algorithm, called the 
RCAR algorithm, that 
generalizes this assumption to non-Gaussian prior measures.

The remainder of this Section is organized as follows: We recall the RCAR algorithm in Subsection~\ref{sec:rcar-algorithm} and 
give an overview of our main results in Subsection~\ref{sec:overv-main-results}. Relevant literature to our work is discussed in 
Subsection~\ref{sec:relevant-literature} followed by a concrete running example in Subsection~\ref{sec:an-illustr-example} 
which is used throughout the article to demonstrate our theoretical results and conditions in a practical setting. An outline of 
the article is given in Subsection~\ref{sec:outline-article}.

\subsection{The RCAR algorithm}\label{sec:rcar-algorithm}
The MH algorithm we study utilizes a proposal akin to a random coefficient
autoregressive proposal (RCAR), defined as follows.
\begin{definition}[RCAR-MH kernel]\label{def:kernel-P}

  Given a function $\Psi: \mcl H \mapsto \mbb R$, a transition kernel $\mcl K(u, \cdot)$,
    and an innovation measure $\lambda \in P(\mcl H)$,
  the RCAR-MH transition kernel $\P$ is defined as
\begin{equation}
  \label{lazy-chain-kernel}
  \mcl{P}(u, \dd v) :=  \mcl Q(u,  \dd v )\alpha(u, v) +
  \delta_u \int_{\mcl{H}} (1- \alpha(u,w)) \mcl{Q}(u,\dd w) ,  \quad u \in \H,
\end{equation}
with proposal transition kernel
\begin{equation}
  \label{Q-definition}
  \mcl Q(u, \cdot) := \mcl K(u, \cdot) \ast \lambda,
\end{equation}
and acceptance ratio function
\begin{equation}
  \label{ARSD-acceptance-ratio}
  \alpha(u,v) := 1 \wedge \exp(\Psi(u)-\Psi(v)).
\end{equation}

\end{definition}

  The RCAR-MH family of  kernels defined above are commonly encountered in
  the design of MH algorithms. The form \eqref{lazy-chain-kernel} is often
  referred to as the {\it lazy chain} representation of  $\P$. The first term
    accounts for the proposal of a new point $v \sim \Q(u, \cdot)$ that is
  then accepted with probability $\alpha(u,v)$ and the chain moves from $u$ to $v$.
  The second term accounts for the
  event where the proposed point $v$ is rejected and the chain remains at $u$.

  We summarize the RCAR algorithm in  
Algorithm~\ref{generic-RCAR} for reference.
  Many common MH algorithms such as the Random Walk (RW)  algorithm \cite{casella}
  and pCN \cite{stuart-mcmc} fall within
  the RCAR-MH family. In both RW and pCN the measure $\lambda$
  is taken to be an appropriate Gaussian measure. The kernel $\K = \delta_u$
  for RW, while $\K = \delta_{\beta u}$ for a constant $\beta \in (0,1)$ in the case of
  pCN.

\begin{algorithm}
  \caption{Generic RCAR-MH}
  \label{generic-RCAR}
\begin{enumerate}
\item Set $j = 0$ and choose $U_0 \in \mcl{H}$. 
\item At iteration $j$ propose $W_{j+1} = \zeta_{j+1} +  \xi_{j+1}$
  where $\zeta_{j+1} \sim \K(U_j, \cdot)$ and $\xi_{j+1} \sim \lambda$.
\item Set $U_{j+1} = W_{j+1}$ with probability $\alpha( U_j, W_{j+1})$.
\item Otherwise set $U_{j+1} = U_j$.
\item set $j \leftarrow j +1$ and return to step 2.
\end{enumerate}
\end{algorithm}

Consider the measure $\nu$  defined in \eqref{nu-definition} with $\mu \in P(\mcl{H})$.
It was shown in \cite[Thm.~2.1]{hosseini-RCAR} that under mild conditions on $\Psi$,
the measure $\nu$ is an invariant measure of $\P$ provided that $\Q$ is reversible
with respect to $\mu$, i.e.,
\begin{equation}\label{Q-prior-reversible}
  \int_A \mcl Q(u, B) \mu(du) = \int_B \mcl Q(u, A) \mu(\dd u),
 \end{equation}
 for Borel sets $A,B \in \mc H$. 

The article \cite{hosseini-RCAR} presents multiple numerical experiments demonstrating the ability of 
RCAR to sample the target measures $\nu$ that arise as posterior measures in
Bayesian inverse problems with non-Gaussian priors. 
The RCAR algorithm is widely applicable since appropriate proposal kernels $\mcl Q$
can be identified for many commonly used
probability distributions such as Gaussian, Laplace and Gamma (together with their extensions to infinite-dimensional measures). However, as yet no analysis of RCAR convergence rates and existence/uniqueness of
invariant measures of $\P$ has been performed.
  In general, ensuring that $\mcl Q$ is $\mu$-reversible
depends on the choice of  $\lambda$ and $\K$ in relation to $\mu$ and
 is often the most 
 difficult aspect of designing new MH algorithms, especially when $\H$ is infinite dimensional
 \cite{stuart-mcmc, hosseini-RCAR}.
However, $\mu$-reversibility of $\Q$ only ensures that $\nu$ is an invariant measure of $\P$.
The primary goal of this article is to analyze the convergence properties of
$\P$, showing the existence and uniqueness of an invariant measure to which
convergence occurs at an exponential rate. We further justify the use of perturbed/finite-dimensional versions of the algorithm by providing general perturbation bounds.

\subsection{Overview of main results}\label{sec:overv-main-results} 

We now give a brief survey of our main results with simplified technical assumptions.
Details of these results are presented in Sections~\ref{sec:spectral-gaps} and 
\ref{sec:Approximations}.
Let $\mcl P$ be an RCAR-MH kernel
as in Definition~\ref{def:kernel-P} with 
a transition kernel $\K$ and innovation measure $\lambda$. Our first result concerns the existence of a spectral gap for 
$\mcl{P}$ in certain semimetrics implying exponential convergence to
a unique invariant measure; throughout we will use the term "spectral gap" in topologies other than $L^2$, consistent with  \cite{hairer2011asymptotic,hairer2014spectral}.

For $q \ge 1$ and $\eta, \omega, \theta >0$ define the semimetric
\begin{equation}
  \label{eq:2}
  \hd_q( u,v) :=  \left[ \left(1 \wedge \frac{(1 + \eta \| u\| + \eta \| v\|)^q \| u - v\| }{\omega} \right) (2 + \theta \| u\|^q + \theta \|v \|^q) \right]^{1/2},
\end{equation}
for points $u,v \in \H$. We refer to $\hd_q$ as a semimetric since
it does not satisfy the triangle inequality but satisfies other metric axioms.
This  semimetric further induces a  transport semimetric
\begin{equation}
  \label{eq:3}
  \hd_q( \nu_1, \nu_2) :=
  \inf_{\pi \in \Upsilon(\nu_1, \nu_2)} \int_{\mcl{H} \times \mcl{H}} \hd_q(u,v) \pi( \dd u, \dd v),
  \quad \forall \nu_1,\nu_2 \in P(\H),
\end{equation}
where $\Upsilon(\nu_1, \nu_2)$ denotes the space of all couplings between probability
measures $\nu_1,\nu_2$. 
We let $\tP(\H; \hd_q) \subset P(\H)$ denote the subspace of probability
measures on $\H$ for which $\hd_q( \cdot, 0)$ is integrable. 
Then our first main result states that the RCAR-MH kernel
has a unique invariant measure to which exponential convergergence occurs
in  $\hd_q$.

\begin{mainresult}\label{mainresult:spectral-gap}
 Suppose
 $\lambda$ has bounded moments of degree $p \ge 1$ and the Lipschitz constant of $\Psi$ does
 not grow faster than $\|  \cdot  \|^q$ for some integer $q \le p$. 
 Then under regularity conditions on $\K$ and for an appropriate choice of
 the constants $\eta, \omega, \theta$ it holds that: 

\begin{enumerate}[label=(\alph*)]
\item 
There exist  constants $(\gamma, n) \in (0,1) \times \mbb N$ so that 
\begin{equation}\label{eq:ContractionSummary}
  \hd_q(\mcl{P}^n\nu_1, \mcl{P}^n\nu_2) \le \gamma \hd_q(\nu_1, \nu_2),
  \quad \forall \nu_1,\nu_2 \in \tP(\mcl{H}; \hd_q).
\end{equation}

\item $\P$  has a unique invariant measure $\nu \in \tP(\H; \hd_q)$.

\item If $\Q$ is $\mu$-reversible then $\nu$ coincides with the target measure
   \eqref{nu-definition}.
  
\end{enumerate}
\end{mainresult}

Detailed statement and proof of this result is presented in Section~\ref{sec:spectral-gaps} 
where we give a detailed statement of the underlying assumptions on $\Psi$ and $\K$ required to 
prove the three statements as well as the detailed versions of these results. The  proofs 
are further postponed to Appendix~\ref{sec:proofs-not-included}.

It was shown in \cite{hosseini-RCAR} that the RCAR algorithm satisfies
detailed balance whenever \eqref{Q-prior-reversible} holds and so has unique invariant measure
$\nu$ given by \eqref{nu-definition}. 
In Section~\ref{sec:main-results-convergence} we present an alternative
proof of the fact that $\P$ has a unique invariant measure by showing  that  $\mcl{P}$ is Feller, implying 
 that $\P$ has a unique invariant measure under more general conditions than
\eqref{Q-prior-reversible}. However, without \eqref{Q-prior-reversible}
one cannot guarantee
that the invariant measure is the  target $\nu$ in \eqref{nu-definition}; see also Remark~\ref{remark-feller-reversibility}.

Our second main result concerns the perturbation properties  of RCAR-MH kernels.
In many applications, such as when $\mcl{H}$ is a function space, the RCAR algorithm  
cannot be implemented exactly since it is not possible
to simulate $\K$ and  $\lambda$, and one resorts to numerical approximations
by discretization or direct approximations of $\Psi$, $\K$ or  $\lambda$.
To this end, we provide bounds on the approximation error 
resulting from using perturbations $\P_\veps$ of an RCAR-MH kernel $\P$. 
We characterize closeness of the invariant measure(s) of $\P_{\veps}$
to $\nu$, as well as the similarity of the dynamics of $U_k \sim \P^k \delta_{u_0}$ to $U_k^\veps \sim \P^k_\veps \delta_{u_0}$
for $u_0 \in \mc H$. We emphasize that while the main result below is stated for RCAR-MH kernels, the results we prove in
Section \ref{sec:Approximations} 
are indeed more general and are applicable to any $\P$ that satisfies  the conditions of the  weak Harris' theorem (see Proposition~\ref{thm:weak-harris}) below. Before proceeding further let us recall the  Lipschitz seminorm with respect to the semimetric $\hd_q$ on (Bochner) measurable functions $\varphi : \H \to \X$ for a separable Hilbert space $\X$:
\begin{equation} \label{eq:LipSeminorm}
\lvertiii \varphi \rvertiii_{\hd_q} := \sup_{u \ne v} \frac{\|\varphi(u) - \varphi(v)\|_{\X}}{\hd_q(u,v)}.
\end{equation}

 \begin{mainresult}\label{mainresult:approximation}
   Suppose that the conditions of Main Result~\ref{mainresult:spectral-gap} hold and
   let $\P_\veps$ be a Markov transition kernel on $\mc H$. Suppose that $\| \cdot \|^q$ is a common Lyapunov function (see 
   Definition~\ref{lyapunov-function}) for  $\P$ and $\P_\veps$
   for sufficiently small $\veps$
   so that 
   \begin{equation*}
       \P \|  u \|^q \le \kappa \| u \|^q + K, \qquad \P_\veps \| u \|^q \le \kappa \| u\|^q + K,
   \end{equation*}
   for constants $(\kappa, K) \in (0,1) \times (0, + \infty)$, and that there exists a bounded function $\psi: \mbb{R}_+ \mapsto \mbb{R}_+$ for which
\begin{equation*}
  \hd_q(\P_\veps \delta_u, \P \delta_u) \le  \psi(\veps) (1 +  \| u\|^{q/2}).
\end{equation*}
\begin{enumerate}[label=(\alph*)]
\item Then 
there exists a constant $C_1 >0 $ independent of $\veps >0$ so that
$$
\hd_q( \nu, \nu_\veps) \le C_1 \psi(\veps) \big[ 1+ \nu_\veps \big( \| \cdot \|^{q/2} \big) \big],
$$
where $\nu$ is the unique invariant measure of $\P$ and $\nu_\veps$ is any invariant measure of  $\P_\veps$.

\item  Let $\mc X$ be a separable Hilbert space with norm $\| \cdot \|_\X$ and $\varphi: \H \mapsto \mc X$ be $\nu$-Bochner measurable and satisfy $\lvertiii \varphi \lvertiii_{\hd_q} < +\infty$.
  Then there exist constants
  $C_j \ge 0, j=2,\ldots,4$, independent of $n \ge 2$ and $\veps > 0$, such that  
\begin{equation*}
  \E \left\| \frac1n \sum_{k=0}^{n-1} \varphi(U_k^\veps) - \nu(\varphi)  
\right\|_{\X} \le  \frac{\lvertiii \varphi \rvertiii_{\hd_q}}{1-\gamma} 
\left( C_2 \psi(\veps) 
+ C_3 \frac{\psi(\veps)}{n} + C_4\frac{1}{\sqrt{n}}  \right),
\end{equation*}
where $U^\veps_k \sim \mcl P^{k-1}_\veps \delta_{u_0}$ for any initial state $u_0 \in \H$.
\end{enumerate}
\end{mainresult}

We present detailed versions of the above statements together with our underlying assumptions on $\P, \P_\veps$ 
in Section~\ref{sec:Approximations}. Detailed proofs of those results are postponed to Appendix~\ref{sec:proof-main-approx-results}.

\subsection{Relevant literature}\label{sec:relevant-literature} 
The convergence rate results we give here rely on the existence of Lyapunov functions
of $\P$ and $\P_\veps$ to control stochastic stability. The use of Lyapunov functions
has been important at least since \cite{khasminskii1980stochastic}, and their application
to  convergence analysis of Markov chains is developed in great detail in the influential 
text of Meyn and Tweedie \cite{meyn1993markov}; see also the more recent text of Douc et al. \cite{douc2018markov}. In the Markov chain Monte Carlo (MCMC) literature, convergence is often studied by showing
a form of Harris' classic theorem \cite{harris1956existence}, which states that a Markov chain is uniquely
ergodic if there exists a set satisfying an analogue of Doeblin's condition, perhaps holding only for the $n$-step  
kernel $\P^n$ for some $n < +\infty$, that is visited infinitely often. 
One typically proves Harris' result by showing a minorization condition for $\P^n$ on sublevel sets of the Lyapunov
function \cite{rosenthal1995minorization}; an elementary proof can be found in
\cite{hairer2011yet}. 
Example  applications of such ``drift and minorization''
arguments to MH algorithms can be found in 
\cite{jarner2000geometric, mengersen1996rates,roberts1996geometric}.

Proofs of Harris' theorem utilizing a Lyapunov condition typically guarantee exponential convergence
toward the unique invariant measure  in a total variation (TV) metric weighted by the Lyapunov 
function \cite{hairer2011yet}. When the state space is high or infinite dimensional, such TV  
metrics are a poor choice  because probability measures on infinite-dimensional spaces have a tendency to become mutually singular
after small perturbations \cite{bogachev-malliavin}. Due to this phenomenon
it is typically not possible to couple two copies of a Markov chain such that they move to exactly the same point with positive probability, even over multiple steps. However, for measures on Banach spaces one can typically show a topological irreducibility condition, i.e., that the two copies draw together over time in an appropriate (semi)metric, at least when initialized inside of sublevel sets of a Lyapunov function.    

We study convergence 
of the  RCAR algorithm on infinite-dimensional Banach spaces using the ``weak Harris'' theorem
of Hairer et al. \cite{hairer2011asymptotic}. This can be viewed as an extension of
the ordinary Harris theorem to transport semimetrics. These semimetrics are designed to induce a topology on bounded sets such that the topological irreducibility condition holds. An application of the weak Harris' theorem to the pCN algorithm can be found in \cite{hairer2014spectral}, wherein it is proved that pCN 
has a dimension-independent spectral gap. As the pCN algorithm is a special case of RCAR with a 
Gaussian innovation $\lambda$ and a deterministic kernel $\K(u, \cdot) =
\delta_{\beta u}$,
our results for dimension-independent spectral gap of RCAR can be viewed as a
generalization of \cite{hairer2014spectral}.

Our approximation theory on the other hand is inherently
different from \cite{hairer2014spectral}. Rather than showing analogous spectral
gap  results for 
discretizations of the algorithm and then showing that the invariant measures are close as in
\cite{dashti2011uncertainty, mattingly2012diffusion}, 
we instead utilize perturbation bounds as in \cite{johndrow2017error} to
bound the distance between the invariant measures by the $n$-step approximation error
between the exact kernel $\P$ and the approximation $\P_\veps$.
 Perturbation theory for MCMC is also studied in \cite{mitrophanov2005sensitivity, pillai2014ergodicity, rudolf2018perturbation}, but these results are not well-suited to our infinite-dimensional state space setting, since they require the triangle inequality which is typically not satisfied by the transport semimetrics that we work with. The perturbation bounds we obtain have the advantage of controlling all of the quantities of interest in terms of the spectral gap of the kernel $\P$ and a pointwise bound on the approximation error of the approximate kernel $\P_\veps$. 
 We further apply these results to cases where the innovation $\lambda$ or the kernel $\K$
 cannot be exactly simulated, using arguments
similar in spirit to those used to prove convergence rates for discretization of MH proposal kernels in \cite{johndrow2017error} but technically much more involved. 
Rather than changing norms to $L_2(\nu)$ to obtain a central limit theorem, we proceed in the tradition of \cite{GlynnMeyn1996,KontoyiannisMeyn2003,KontoyiannisMeyn2012,MattinglyStuartTretyakov2015} and give variation bounds using the
Poisson equation. This gives approximation error bounds to $\nu$, as well as approximation error bounds for pathwise quantities for both $\P$ and $\P_\veps$ for elements of the function space $\{ \varphi: \lvertiii \varphi \rvertiii_{\hd_q} < +\infty\}$, without requiring any direct analysis of $\P_\veps$ or $\nu_\veps$. This technique is related to classical Martingale and potential methods \cite{MR0415773}.

We highlight that 
our purpose here is to show error bounds that are independent of 
dimension and allow us to obtain rates for the error in quantities of interest,
 not to produce quantitative estimates of the number of steps necessary to achieve a particular accuracy.
All of the bounds we give for 
approximate
versions of the algorithm depend only on the spectral gap of the exact kernel $\mcl P$, and
the pointwise accuracy of the 
approximate
kernel $\mcl P_\veps$ as well as the constants in its Foster-Lyapunov condition.
While the former is independent of dimension,
the latter two quantities relating to $\P_\veps$ typically improve
as $\veps \to 0$
and $\P_\veps$ draws closer to $\P$. This behavior contrasts with the typical performance of ``drift and minorization'' bounds in weighted TV norms for 
finite-dimensional problems
 where the spectral gap tends to vanish as dimension increases (see e.g. \cite{rajaratnam2015mcmc}). In a sense, the dimension-independence of our results can be attributed to choosing a semimetric that is better adapted to high or infinite-dimensional spaces than weighted TV.

It is worth noting that one can typically obtain sharper numerical estimates of mixing or relaxation times
using geometric inequalities, such as log-Sobolev, Cheeger, and Poincar\'{e} inequalities. A thorough 
review of these techniques and their application to MH algorithms is given in \cite{diaconis1998we}. More recent work applying geometric inequalities to obtain sharp bounds for mixing times
of MH on bounded subsets of $\bb R^d$ can be found in
\cite{diaconis2009micro, diaconis2011geometric}.
 Geometric inequalities are combined with Lyapunov arguments to obtain sharper estimates of relaxation times for MH on
$\bb R$ in \cite{johndrow2018fast}.
At the time of this writing, we are not aware of analogous results
for infinite-dimensional MH.

\subsection{An illustrative example in nonlinear regression}\label{sec:an-illustr-example}

We now outline the details of a running example that is used throughout the
article  to
give context to the main ideas and assumptions in our analysis.
The motivation for this example is the semi-supervised regression (SSR)
problem \cite{bertozzi2018uncertainty, dunlop2020large}; the task of inferring a function on a graph
from indirect and limited observations of its values on a subset of the nodes.
In the large graph limit, as the number of vertices
tend to infinity, the SSR problem converges to a nonlinear regression
problem where a nonlinear transformation of a latent function is observed at a few points 
and the goal is to recover the latent function.

  \begin{expl}[Nonlinear regression]\label{SSL-example}
    Let $\TT$ be the unit circle 
  and  $\mcl H = H^1(\TT)$ the
    Sobolev  space of weakly differentiable functions on the unit circle with square integrable first derivatives.
    Suppose
     $u^\dagger \in H^1(\TT)$ is the 
        {\it ground truth}  function from which  the following  data is measured
      \begin{equation*}
        y \in \mbb R^m, \qquad y_j = \tanh(u^\dagger(x_j)) + \epsilon_j.
      \end{equation*}
     Here  $\{x_j \}_{j=1}^m$ are fixed points in $\TT$ and the
      $\epsilon_j \iidsim N(0, \sigma^2)$ with variance $\sigma >0$. Since $H^1(\TT)$ is embedded in $C(\TT)$ by the Sobolev embedding theorem \cite{adams}, then 
      the pointwise evaluation of $u^\dagger$ is well-defined.

      Now consider the inverse problem of inferring
      the function $u^\dagger$ from an instance of the data $y$.  
      To solve the problem we write Bayes' rule \cite{stuart-acta-numerica}
      in the following form 
      \begin{equation}\label{ex1-Bayesian-posterior}
        \frac{\dd \nu}{\dd \mu} (u) = \frac{1}{Z(y)} \exp( - \Psi(u;y) ), \quad
        Z(y)  =\int_{L^2(\mbb T)} \exp( - \Psi(u;y)) \mu( \dd u),
      \end{equation}
      where $\Psi$ is the likelihood potential, $\mu$ is the prior probability measure
      and $\nu$ is the posterior measure. Since the $\epsilon_j$ are
      Gaussian we ascertain that the likelihood potential $\Psi(u;y)$ is given by
      
      \begin{equation*}
        \Psi(u;y) = \frac{1}{2\sigma^2} \sum_{j=1}^m  |\tanh(u(x_j))  - y_j|^2.  
      \end{equation*}
        As for the prior measure $\mu$ we take 
      \begin{equation}\label{product-prior-expl-1}
        \mu = \Law \left\{   \sum_{j=1}^\infty a_j  \eta_j
         \phi_j  \right\},
      \end{equation}
      where  
       $\{\eta_j\}_{j=1}^\infty \iidsim \text{Gamm}(1/5,1)$
      random variables with Lebesgue density
      $f(t) =  \frac{1}{\Gamma(1/5)} t^{1/5-1} \exp(-t) \mbf{1}_{(0, \infty)}(t)$ for $t \in \mathbb R$.
      The $\phi_j$ are the DB12 wavelet basis \cite{daubechies1992ten} normalized in $L^2(\TT)$ with scaling function $\phi_0$  and
      \begin{equation*}
          \phi_{2^{k} + m}(t)  = 2^{k/2} \phi(2^k t - m_k), \qquad k = 1, 2, \dots, \qquad m_k = 0, 1, 2, \dots, 2^{k}-1,
      \end{equation*}
      with $\phi$  denoting the DB12 mother wavelet. Finally, the coefficients 
      $\{a_j\}_{j=1}^\infty$ are  chosen as 
      \begin{equation*}
           a_1 = 1 \qquad \text{and} \qquad a_{2^k + m_k} = 2^{-2k}.
      \end{equation*}
      Our choice of the $a_j$ and the laws of $\eta_j$ together with the regularity of the DB12 wavelets ensure that $\mu$ has full support
      on the subspace of $H^1(\TT)$ consisting of functions with positive wavelet coefficients.

      In order to sample the resulting posterior we employ  \cite[Alg.~4]{hosseini-RCAR},
      an instance of the RCAR algorithm for the prior $\mu$. 
            For a fixed $\beta \in (0,1)$ we  take
      \begin{equation*}
      \begin{aligned}
        & \mcl K(u,\cdot) = \Law \left\{   \sum_{j=1}^\infty 
          \tau_j \langle u, \phi_j \rangle_{L^2(\TT)} \phi_j  \right\}, \quad 
        \{\tau_j\}_{j=1}^\infty \iidsim \Beta(\beta/5,(1- \beta)/5),
      \end{aligned}
      \end{equation*}
      with $\langle \cdot, \cdot \rangle_{L^2(\TT)}$ denoting the $L^2(\TT)$-inner product.
        We then define the  innovation $\lambda$  as
      \begin{equation}\label{innovation-deconvolution}
      \begin{aligned}
        & \lambda = \Law \left\{  \sum_{j=1}^\infty a_j
         \xi_{j} \phi_j\right\}, \quad 
         \xi_j \iidsim \text{Gamma}((1- \beta)/5, 1).
      \end{aligned}
     \end{equation}
    It then follows from \cite[Thm.~3.4]{hosseini-RCAR} that the
    resulting proposal  kernel $\Q$ as in \eqref{Q-definition} is
    $\mu$-reversible, implying that the RCAR-MH kernel $\P$ 
      is  $\nu$-reversible in this example.

      Figure~\ref{fig:nonlinear-regression-intro}  depicts an example application of the RCAR-MH algorithm described above for 
      recovering a function $u^\dagger$ with sparse and positive wavelet coefficients; the details of this experiment 
      are summarized in Subsection~\ref{sec:nonlinear-regression}. Here we truncate the infinite sum  in \eqref{product-prior-expl-1} up to $N$ terms.
      Figure~\ref{fig:nonlinear-regression-intro}(a) shows the ground truth function $u^\dagger$ together with the
      measurements $y$ and the resulting
      posterior 
      mean obtained from RCAR-MH samples with $N=128$ wavelet modes.
      Figure~\ref{fig:nonlinear-regression-intro}(b)
      shows the average MH acceptance ratio as a function of the step size parameter $\beta$
      for various  choices of $N$ (the dimension of the inference parameter). The independence of the acceptance ratio from the dimension $N$
      is a telltale sign of the dimension-independent convergence properties of RCAR-MH.  Similar behavior was also observed in the 
      numerical experiments considered in \cite{hosseini-RCAR}.

    \begin{figure}
        \centering
        \subfloat[]{\includegraphics[height=.35\textwidth]{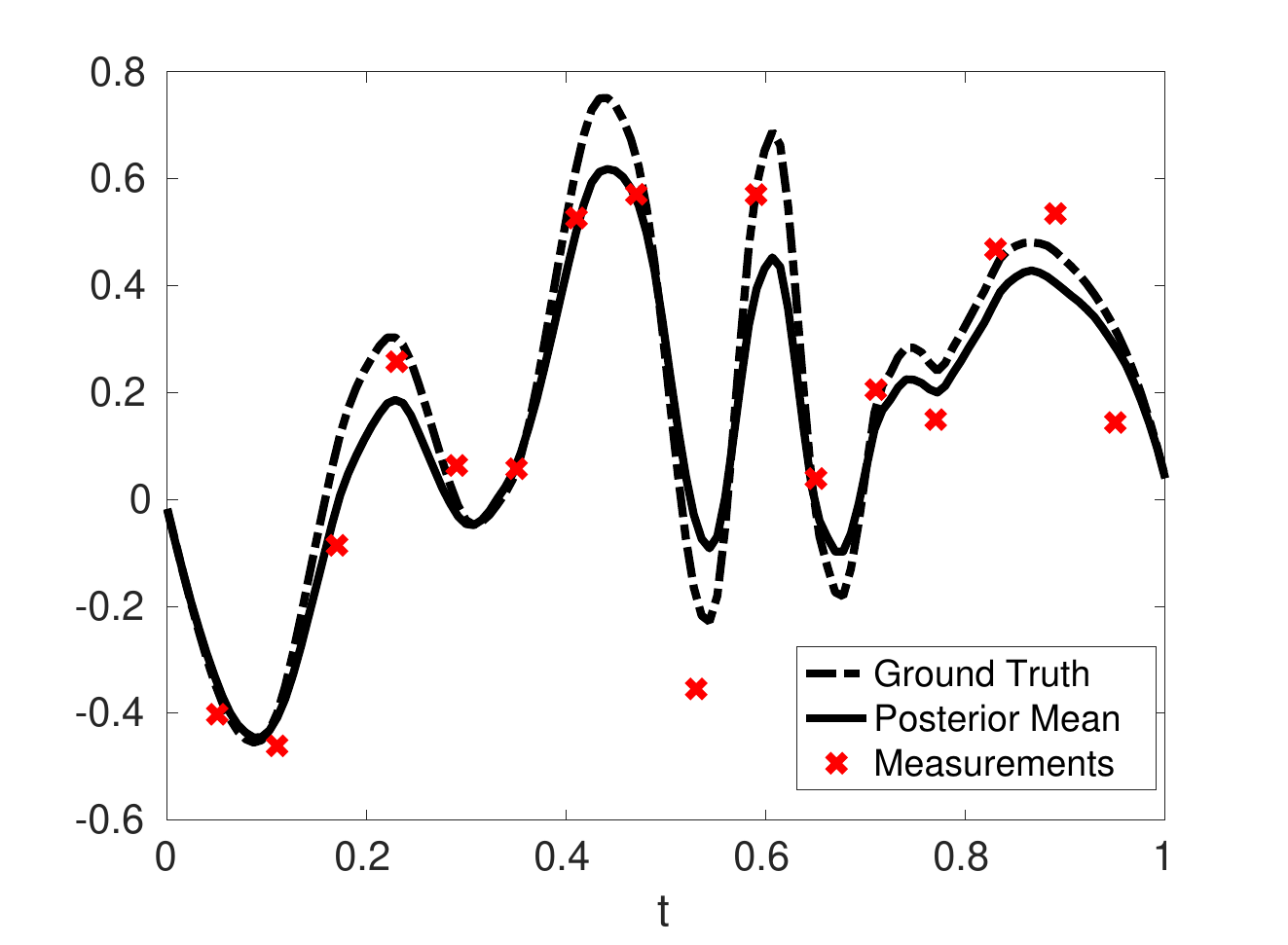}}
        \subfloat[]{\includegraphics[height=.35\textwidth]{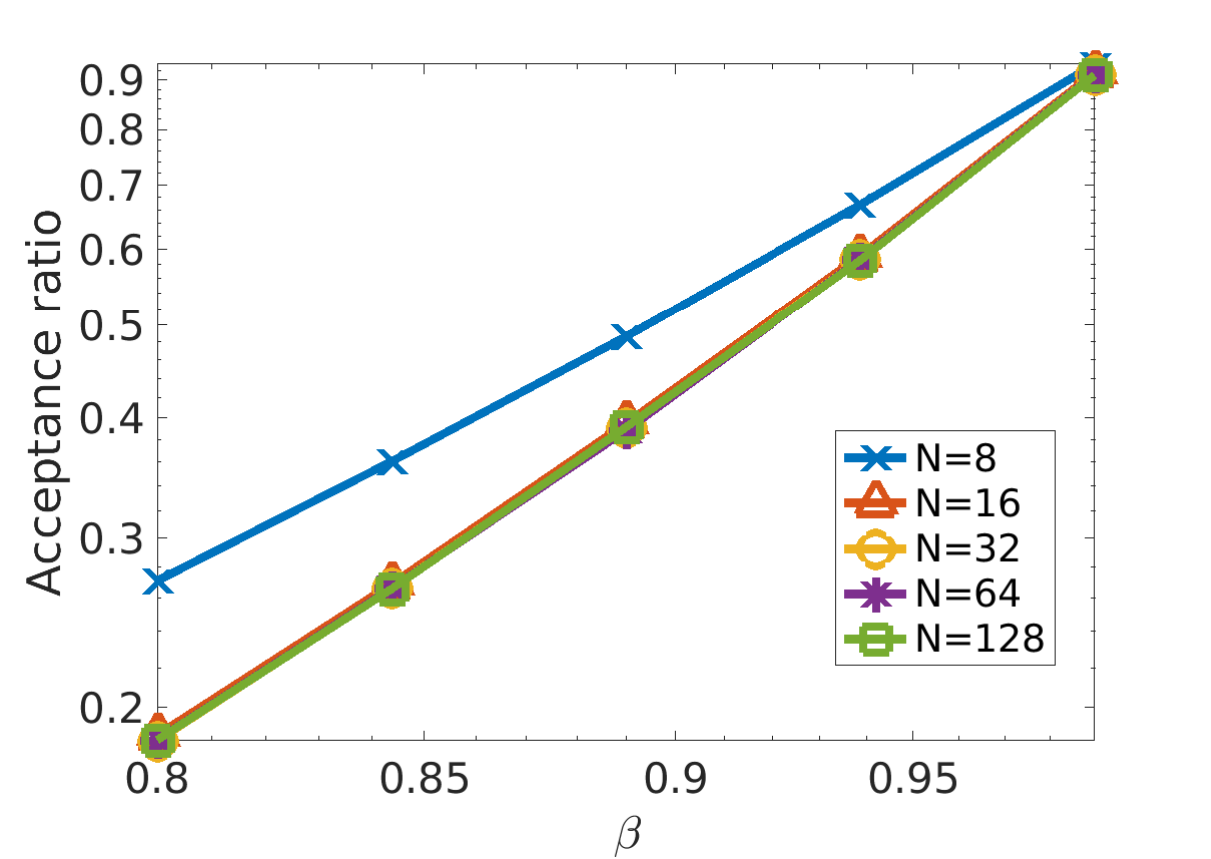}}
        \caption{Representative numerical results for Example~\ref{SSL-example}. (a) The ground truth function $u^\dagger$ 
        together with the measurements $y$ and the RCAR-MH estimate of the posterior mean. (b) The acceptance ratio of 
        RCARC-MH as a function of the parameter $\beta$  for different  number of wavelet modes $N$.}
        \label{fig:nonlinear-regression-intro}
    \end{figure}

  \end{expl}



\subsection{Outline of the article}\label{sec:outline-article} 
We dedicate Section~\ref{sec:prel-notat} to  preliminary results and definitions that are used throughout the article 
and fix our notation. In Section~\ref{sec:spectral-gaps} we give results pertaining to the convergence properties of RCAR-MH kernels
which together constitute the detailed version of Main Result~\ref{mainresult:spectral-gap}. Analogously Section~\ref{sec:Approximations} contains our 
main perturbation results for Markov kernels that satisfy the conditions of weak Harris' theorem constituting the detailed 
statement of Main Result ~\ref{mainresult:approximation}. We consider several applications of our
results in Section~\ref{sec:applications}, and dedicate Section~\ref{sec:conclusion} to conclusion and offer some thoughts on future directions. The Appendix
contains the technical proofs of key results that are not included in the main text.

\section{Preliminaries}\label{sec:prel-notat}
We gather here some preliminary results on ergodic theorems and the weak Harris' theorem as
well as some notation and terminology that is used throughout the article. Results on ergodicity, most notably the weak Harris' theorem 
are reviewed in Subsection~\ref{sec:results-ergodicity} while further notation is outlined in Subsection~\ref{sec:notation}.

\subsection{Results on ergodicity} \label{sec:results-ergodicity} 
We study convergence in the context of the {\it weak Harris} theorem of \cite{hairer2011asymptotic},
which is an extension of the classical Harris' theorem to Wasserstein-type notions of distance defined 
in terms of lower semi-continuous semimetrics referred to as ``distance-like'' in
\cite{hairer2011asymptotic}.
\begin{definition}\label{def:distance-like-function}
  A  function $d: \mcl{H} \times \mcl{H} \mapsto \mbb{R}$ is {\it distance-like} 
if it is positive,  symmetric, lower semi-continuous, and $d(u,v) = 0$ iff $u =v$.
\end{definition}
Non-negative functions that satisfy all of the metric axioms save the triangle inequality
are often referred to as {\it semimetrics}, and we also adopt this terminology. Thus, 
a distance-like function is a lower semi-continuous semimetric. Given a distance-like function $d$, we can extend it to a Wasserstein or
transport-like positive function on $P(\mcl{H})$
via 
\begin{equation}\label{def:wasserstein-1-metric}
d(\mu_1, \mu_2) := \inf_{\pi \in \Upsilon(\mu_1, \mu_2)} \int_{\mcl{H} \times \mcl{H}} d(u,v) \pi( \dd u, \dd v),
\end{equation}
where we recall $\Upsilon(\mu_1, \mu_2)$ is the space of all couplings of $\mu_1$ and $\mu_2$, i.e., 
the space of measures $\pi \in P(\H \times \H)$ whose marginals on the first and second variables coincide with $\mu_1$ and $\mu_2$ 
respectively.
We also introduce the subspace  $P^1(\H; d) \subset P(\H)$ as
\begin{equation}
  \label{P-1-space-def}
  P^1(\H, d) := \left\{ \mu \in P(\H)  \; | \; \int_\H d(u, 0) \mu(\dd u) < + \infty \right\}, 
\end{equation}
following the standard notation for Wasserstein topologies \cite{villani-OT}.

Given a distance-like function $d$ we also introduce the space $\text{Lip}(d)$
consisting of functions that are Lipschitz continuous with respect to $d$. More precisely, for a separable Hilbert space $\X$ with norm $\| \cdot \|_\X$, define
\begin{equation}
  \text{Lip}(d) := \{ \varphi : \mcl{H} \mapsto \mc X : \quad \rvertiii \varphi \lvertiii_d  < + \infty \},
\end{equation}
where
\begin{equation} \label{eq:LipSeminormGeneral}
  \lvertiii \varphi \rvertiii_d := \sup_{u \neq v} \frac{ \| \varphi(u) - \varphi(v) \|_{\mc X} }{d(u,v)}.
\end{equation}
Observe that the definition in \eqref{eq:LipSeminorm} is just a specific example of \eqref{eq:LipSeminormGeneral} with the choice of $d = \tilde d_q$ defined in \eqref{eq:2}. We now define some properties of $\P$ that together give the weak Harris theorem. 
\begin{definition}\label{lyapunov-function}
  A function $V: \mcl{H} \mapsto \mbb{R}$ is a {\it Lyapunov} function for 
a Markov transition kernel $\mcl{P}$ if there exist $(\kappa, K) \in (0,1) \times (0, +\infty)$ so that 
$$
(\mcl{P}V)(u) \le \kappa V(u) + K, \qquad \forall\, u \in \mcl{H}.
$$ 
\end{definition}
Lyapunov functions
are a standard way to control tail behavior of $\P$.
We further require that when initiated from two $d$-nearby points, we can couple two copies of the 
Markov chain evolving according to $\P$ such that they draw together in one step.
\begin{definition}\label{d-contracting-metric}
A distance-like function $d : \H \times \H \to [0,1]$ is contracting for a Markov operator $\mcl{P}$ if 
there exists $\gamma_1 \in (0,1)$  so that 
\begin{equation*}
d(\P \delta_u, \P \delta_v) \le \gamma_1 d(u,v), \qquad \text{whenever} \qquad d(u,v) < 1.
\end{equation*}
\end{definition}
The assumption that $d$ is capped at 1 is entirely innocuous; for details see \cite[Remark 4.7]{hairer2011asymptotic}.
Finally, we will need a type of topological irreducibility on sublevel sets of $V$ reminiscent of Doeblin's condition \cite[Sec.~16.2.1]{meyn1993markov}
in the classical TV theory of convergence.
\begin{definition}\label{W-minorization}
  For every $R > 0$ the sublevel sets $S(R) := \{ u \: | \: V(u) < R \}$ of $V$ are 
$d$-small for a distance-like function $d : \H \times \H \to [0,1]$ if there exists $\gamma_2(R) \in (0,1)$ and $n \in \bb N$ so that,
$$
\sup_{u,v \in S(R)} d(\mcl{P}^n\delta_u, \mcl{P}^n\delta_v)  \le \gamma_2 (R).
$$
\end{definition}
In some cases it is possible to show that Definitions~\ref{d-contracting-metric} and \ref{W-minorization} hold
with $n=1$, but it is imperative for  our technical results to use Definition~\ref{W-minorization} with a sufficiently large $n$.
Given a distance-like function $d$, we define a new weighted distance-like function
\begin{equation}\label{tilde-d-definition}
\tilde{d}(u,v) := \left[ d(u,v) ( 2 + \theta V(u) + \theta V(v) \right]^{\frac{1}{2}},
\end{equation}
for a parameter $\theta > 0$. Observe that the $\hd_q$ semimetric introduced in \eqref{eq:2} is a particular example of \eqref{tilde-d-definition}
with $V(u) = \| u \|^q$ and $d(u,v) = 1 \wedge \omega^{-1}(1 + \eta \| u \| + \eta \| v\|)^q \| u - v\| $.
Once again we use the same notation to denote the induced semimetric
$\td$ on ${P}(\mcl{H})$ and in turn the subspace $P^1(\H; \td)$ as in \eqref{P-1-space-def}.
The following is a discrete-time version of \cite[Thm~4.8]{hairer2011asymptotic} that is more natural
for our setting. Although \cite{johndrow2017error} shows this condition in the case where $n=1$ in Definition \ref{W-minorization}, the extension to general $n$ is a minor modification of their argument; and indeed, of the continuous-time result in \cite{hairer2011asymptotic}.

\begin{proposition}[{\cite[Thm. 4.8]{hairer2011asymptotic} and \cite[Thm.~3.9]{johndrow2017error}}]\label{thm:weak-harris}
 Suppose $d$ is contracting for $\P$, and $\P$ has a continuous Lyapunov function $V$
with $d$-small level sets. Then there exist constants
$(\gamma, \theta, n) \in (0,1) \times (0, + \infty) \times  \bb N$ such that 
$$
\td( \mcl{P}^n\delta_u, \mcl{P}^n\delta_v) \le \gamma \td(u, v), \qquad 
\forall u,v \in \mcl{H}.
$$
We refer to the constant $1-\gamma$ as the {\it $\td$-spectral gap} of $\mcl{P}$.
\end{proposition}

Next we recall the definition of a Feller Markov operator. Let $C(\H)$ and $C_b(\H)$
denote the spaces of continuous
functions  and continous and bounded functions on $\H$ respectively.
\begin{definition}\label{def-feller-kernel}
 A Markov operator $\P$ is Feller if $\P\varphi \in C(\H)$  for every $\varphi \in C_b(\H)$. That is, $\P$ is Feller if and only if $u \mapsto \P(u,\cdot)$ is continuous in the topology of weak convergence.
\end{definition}
By \cite[Cor.~4.11]{hairer2011asymptotic}
if $\P$ is  Feller and the distance-like function $d$ satisfies some mild conditions, Proposition \ref{thm:weak-harris} also implies the existence of a unique invariant measure for $\P$.

\subsection{Notation}\label{sec:notation} 
We  gather some   notation for future reference.  Throughout
Sections~\ref{sec:spectral-gaps} $\H$ is a separable Banach space with
norm $\| \cdot \|$. In the applications in
Section~\ref{sec:applications}, we take $\H$ to
be a separable Hilbert space with norm $\| \cdot \|$ and inner product $\langle \cdot, \cdot \rangle$. In either case
$\H^\ast$ denotes the dual of $\H$. Throughout,
$B_R(v) \subset \H$ denotes the closed ball in the topology of $\| \cdot \|$ of radius $R>0$ centered at $v$. 
We say that a
measure $\mu\in P(\H)$ has bounded moments of degree $p$ whenever $\| \cdot \|^p \in L^1(\mu)$.
At various points we also consider a second separable Hilbert space  $\X$  with norm $\| \cdot \|_\X$
and inner product $\langle \cdot, \cdot \rangle_\X$.
Furthermore, we often use the notation $\mu(\varphi)$ to denote
the integral $\int_{\H} \varphi(u) \mu(\dd u)$ which is understood in the Bochner sense whenever $\varphi: \H \to \X$. We also use the
standard notation $\mbb E \xi$ to denote expectation of a random variable $\xi$ and $\mbb P(A)$ to denote the probability of an 
event $A$. The measure with respect to which this probability is computed will be clear from context.

\section{Convergence theory  for RCAR}\label{sec:spectral-gaps}
In this section we gather our main theoretical results pertaining to the
   convergence properties of RCAR-MH kernels as in Definition~\ref{def:kernel-P}, constituting 
   a detailed version of Main Result~\ref{mainresult:spectral-gap}.
 We gather our main assumptions
 on the potential $\Psi$ and the kernel $\K$ 
 in Subsection~\ref{sec:assumpt-convergence-theory}, 
 using Example~\ref{SSL-example} throughout to give context to our assumptions.
 We then outline Theorems~\ref{thm:lyapunov-function-for-RCAR}, \ref{thm:spectral-gap-locally-Lip} and
 \ref{thm:existence-uniqueness-of-invariant-measure} which, in turn,
 identify a Lyapunov function for $\P$, show that $\P$ is contracting for an appropriate semimetric,
 and that $\P$ has a unique invariant measure. We postpone the proofs of these theorems to
 Appendix~\ref{sec:proofs-not-included} and only summarize the important details and implications of
 our results. 

 \subsection{Assumptions on the potential \texorpdfstring{$\Psi$}{Psi} and the kernel \texorpdfstring{$\K$}{K} }
 \label{sec:assumpt-convergence-theory}

 Following  Definition~\ref{def:kernel-P} the RCAR-MH kernel $\P$ requires three main ingredients:
 the potential function $\Psi$, which is used to define the acceptance ratio function $\alpha$;
 the kernel $\K$; and, the innovation measure $\lambda$. 
 In order to prove Main Result~\ref{mainresult:spectral-gap} we need the function $\Psi$
 and the kernel $\K$ 
 to satisfy certain regularity and growth assumptions. As we will discuss below and also in Subsection~\ref{sec:deconvolution-circle}, some of these conditions are rather strict and technical. 
 Our only requirement for the innovation measure $\lambda$ is  that it
 has moments of degree $p \ge 1$.

\begin{assumptions}\label{assumptions-on-Phi}
  The function $\Psi: \mcl{H} \mapsto \mbb{R}$ satisfies one or more of the following conditions:
  \begin{enumerate}[label=(\alph*)]
\item (locally bounded from above) For every $R>0$ there exists a constant $M_1(R) \ge 0$ so that 
$$
\Psi(u) \le M_1, \qquad \forall u \in B_R(0).
$$  
\item (locally bounded from below) There exist constants $ q, M_2, M_3 \ge 0$
   so that
$$
\Psi(u) \ge M_3 - M_2 \log( 1 + \| u\|^q), \qquad \forall u \in \mcl{H}.
$$

\item (increasing in the tail) For every $\tilde{\beta}, \tilde{b} \in (0,1)$ 
  there exist strictly positive  constants
   $R_0(\tilde{\beta}, \tilde b), M_4(\tilde{\beta}, \tilde b) >0$ so that
  $\forall u \in B_{R_0}(0)^c$
  $$
  \inf_{v \in B_{\tilde b (1 - \tilde \beta) \| u\|}(\tilde \beta u)} \exp( \Psi(u)  - \Psi(v) ) \ge  M_4.
  $$
\item (locally Lipschitz) There exist constants $L >0$ and 
$q \ge 0$ so that 
$$
|\Psi(u) - \Psi(v)| \le L (1 \vee \| u\|^q \vee \| v\|^q) \| u - v\|,  \qquad \forall u,v \in \mcl{H}.
$$ 
\end{enumerate}
\end{assumptions}

These assumptions are morally equivalent to the assumptions on the potential in
\cite{hairer2014spectral}. Assumption~\ref{assumptions-on-Phi}(a) and (b)
ensure that whenever $\mu$ has bounded moments of degree $p \ge 1$ then
the measure $\nu$ as in \eqref{nu-definition} is well defined for any $q \le p$ 
 \cite[Thm.~4.3]{hosseini-sparse}.
In the context of Bayesian inverse 
problems the function $\Psi$ is a negative log-likelihood function and often these assumptions
are easily satisfied.
For example, for additive noise models \cite[Sec.~4.1]{hosseini-sparse}
the constants $M_3$ and $M_2$ can be taken as zero and Assumption~\ref{assumptions-on-Phi}(a) can
be verified so long as the forward map is bounded.
Assumption~\ref{assumptions-on-Phi}(d) is a regularity condition controlling
the rate at which the  Lipschitz constant of $\Psi$ can grow. This condition
is also commonly encountered in the
literature on Bayesian inverse problems and can be
verified in many applications \cite{hosseini-sparse}.

Assumption~\ref{assumptions-on-Phi}(c), however, is not common and  amounts to $\Psi$ being an increasing function in the tails
which, as discussed in Subsection~\ref{sec:deconvolution-circle}, may not hold
for many benchmark problems in statistics and inverse problems. Although a
simple workaround can be devised using our perturbation theory in Section~\ref{sec:Approximations}.
Hence, replacing Assumption~\ref{assumptions-on-Phi}(c) with a weaker assumption could be an interesting
generalization of the current work that is highly relevant to applications.
A slightly different version of this assumption also appears in 
\cite{hairer2014spectral}, where the radius of the ball over which the infimum is computed is
left as a general function $r(\|u\|)$. In the special case of the pCN algorithm
 one can simply choose the constant function $r(\|u\|) = c$.
However, our analysis reveals that for RCAR algorithms with non-Gaussian priors one really needs
the radius to grow with $\|u\|$. Thus we explicitly state the assumption in this way.
Simply put, the reason is that RCAR proposals concentrate less strongly around the
point $\tilde \beta \|u\|$ than the Gaussian proposals in pCN, making it more difficult to
prove contractive properties of  $\P$.

Next we collect a set of assumptions on the kernel $\K$.

\begin{assumptions}\label{assumptions-on-K-beta}
  Consider the Markov transition kernel $\K$ and let $\zeta_u \sim \K(u, \cdot)$.
  Then  one or more of the following conditions hold:
  \begin{enumerate}[label=(\alph*)]
  \item (almost sure contraction)  $\| \zeta_u \| < \| u\|$ a.s. $\forall u \in \H$.

  \item (local concentration) There exist  constants $0 <b_0 < \beta_0 < 1$ and $\epsilon_0 > 0$
    so that
    $$\mbb P [ \| \zeta_u - \beta_0 u \| \le b_0(1 - \beta_0) \| u\| ] \ge \epsilon_0, \qquad \forall
    u \in \mcl H.$$
  \item (contracting couplings) For all $u,v \in \mcl H$ there exists
    a coupling $\varpi_{u,v} \in \Upsilon(\K \delta_u, \K \delta_v)$ so that
    \begin{equation*}
      \| \zeta_u - \zeta_v \| < \| u - v\| \quad {\rm a.s.} \quad  {\rm for}\quad (\zeta_u, \zeta_v) \sim \varpi_{u,v},
    \end{equation*}
    and  there exists a uniform constant $\beta_c \in (0,1)$ so that
    \begin{equation*}
      \sup_{u,v \in \mcl H,\: u \neq v}  \frac{\int_{\mcl H \times \mcl H} \| \zeta_u - \zeta_v \|
        \varpi_{u,v}( \dd \zeta_u,
       \dd \zeta_v)}{\| u - v\|} \le \beta_c. 
   \end{equation*}
   \end{enumerate}
\end{assumptions}
Unlike the assumptions on $\Psi$, our assumptions on $\K$ do not have an
analogue in
\cite{hairer2014spectral} since our class of algorithms is considerably more general than pCN.
Assumption~\ref{assumptions-on-K-beta}(a) requires $\K(u, \cdot)$ to behave like a random linear operator
  (matrix)
  that shrinks and possibly rotates the vector $u$. Assumption~\ref{assumptions-on-K-beta}(b) ensures that
  $\K(u, \cdot)$  dedicates 
  positive probability mass to a neighborhood of a point $\beta_0 u$ for some $\beta_0 < 1$; ensuring
  that $\zeta_u$ can get sufficiently close to $\beta_0 u$. 
  Assumption~\ref{assumptions-on-K-beta}(c)
is perhaps the most consequential due to the fact that our technical arguments rely on couplings
between measures 
$\K \delta_u$ and $\K \delta_v$
such that they contract in one step. This contractive property is important for handling several technical difficulties that arise in 
proving the $d$-contraction and $d$-smallness conditions in the weak Harris' theorem (Proposition~\ref{thm:weak-harris}).
In the pCN algorithm, $\mc K_\beta(u,\cdot) = \delta_{\beta u}$ is just a delta
measure at $\beta u$ for $\beta < 1$, so the existence of this coupling is trivial, unlike the RCAR algorithm in general. Let us return to Example~\ref{SSL-example} and verify that Assumptions~\ref{assumptions-on-Phi} and \ref{assumptions-on-K-beta} hold for the 
RCAR algorithm presented in that example.

\addtocounter{expl}{-1}
\begin{expl}[Continued]
  As  $\tanh$  is smooth, globally Lipschitz
  and  bounded, and $\Psi$ is quadratic,
    we immediately have that
    Assumption~\ref{assumptions-on-Phi}(a, b) and (d) hold;
    in fact, $\Psi$ is globally Lipschitz.
It remains to check
condition (c). For a fixed $u$ let $v \in B_{\tilde b(1- \tilde \beta) \| u\|_{H^1(\TT)}} (\tilde \beta u)$
for $\tilde \beta, \tilde b \in (0,1)$,     where we used $\| \cdot \|_{H^1(\TT)}$ to denote the $H^1(\TT)$ Sobolev norm.
Using the fact that $\tanh( \cdot ) \in (-1, 1)$ we can write 
\begin{align*}
  2\sigma^2  & \left(  \Psi(u) - \Psi(v) \right) \\
  &= \sum_{j=1}^m | \tanh( h u(x_j))|^2 - |\tanh(h v(x_j))|^2 \\
   & \qquad  - 2y_j \left[\tanh(h u(x_j))  - \tanh(h v(x_j))  \right] \ge - 4 \sum_{j=1}^m |y_j|-m.
\end{align*}
Thus, Assumption~\ref{assumptions-on-Phi}(c) holds with $M_4 = \exp( - 4 \| y\|_1-m)$. 

Now recall the kernel $\K$ given by
  \begin{equation*}
            \K(u, \cdot) = \Law \left\{ \sum_{j=1}^\infty 
          \tau_{j} \langle u, \phi_j \rangle_{L^2(\TT)} \phi_j  , \quad \tau_j\iidsim \Beta(\beta/5,
        (1- \beta)/5)\right\},
    \end{equation*}
    for  $\beta \in (0,1)$.
      Since $\Beta(\beta/5,
        (1- \beta)/5)$ is supported on $(0,1)$ and
      has bounded moments of all degrees we can directly verify, using Markov's inequality, that
      $\K$ satisfies Assumption~\ref{assumptions-on-K-beta}(a).
      
      Next we check Assumption~\ref{assumptions-on-K-beta}(b). Let $\beta_0, b_0 \in (0,1)$.
      Then by Markov's inequality once more,
      \begin{equation*}
        \mbb P [ \| \zeta_u - \beta_0 u \|_{H^1(\TT)}^2 > b^2_0(1 - \beta_0)^2 \| u\|_{H^1(\TT)}^2 ]
        \le \frac{\mbb E \| \zeta_u - \beta_0 u \|_{H^1(\TT)}^2 }{b^2_0 (1- \beta_0)^2 \| u\|^2_{H^1(\TT)}}.
      \end{equation*}
      Following \cite[Sec.~2]{dashti-besov}  we characterize the $H^1(\TT)$ norm  via the DB12 wavelets 
        \begin{equation}\label{H1-wavelet-norm}
            \| u \|_{H^1(\TT)}^2 = \sum_{j=1}^\infty j^2 \langle u, \phi_j \rangle_{L^2(\TT)}^2.
        \end{equation}
      By this expression and Fubini we have
      \begin{align*}
        \mbb E  \| \zeta_u - \beta_0 u\|_{H^1(\TT)}^2
        = \mbb E \sum_{j=1}^\infty j^2 (\tau_j - \beta_0)^2 \langle u, \phi_j\rangle_{L^2(\TT)}^2
        = \| u\|_{H^1(\TT)}^2 \mbb E ( \tau_1 - \beta_0)^2.
      \end{align*}
      To this end, 
            \begin{equation*}
        \mbb P [ \| \zeta_u - \beta_0 u \|_{H^1(\TT)}^2 > b^2_0(1 - \beta_0)^2 \| u\|_{H^1(\TT)}^2 ]
        \le \frac{\mbb E ( \tau_1 - \beta_0)^2 }{b^2_0 (1- \beta_0)^2}.
      \end{equation*}
      Since $\tau_1$ is a Beta random variable we can always choose  constants $b_0,\beta_0 \in (0,1)$
      so that $\mbb E ( \zeta_1 - \beta_0)^2 <b^2_0 (1- \beta_0)^2$.
      For example,
      choosing $\beta_0$ to be arbitrarily small we can simply choose $b_0^2 > \mbb E \zeta_1^2$. 
      This ensures that 
      $$
      \mbb P [ \| \zeta_u - \beta_0 u \|_{H^1(\mbb T)} \le b_0(1 - \beta_0) \| u\|_{H^1(\TT)} ]
      \ge \epsilon_0 >0.
      $$
      
   It remains to verify Assumption~\ref{assumptions-on-K-beta}(c).
      We will construct
      the coupling $\varpi_{u,v}$ explicitly. Take $u,v \in H^1(\TT)$ and draw
      an i.i.d sequence $\{ \tau_j \}_{j=1}^\infty$ where $\tau_j \sim \Beta(\beta/5,
      (1- \beta)/5)$. Then define $\zeta_u$ and $\zeta_v$ via
      \begin{equation*}
        \zeta_u = \sum_{j=1}^\infty \tau_j \langle u, \phi_j \rangle_{L^2(\TT)} \phi_j, \qquad
        \zeta_v = \sum_{j=1}^\infty \tau_j \langle v, \phi_j \rangle_{L^2(\TT)} \phi_j.
      \end{equation*}
      That is, the two chains use the same draw of the $\tau_j$'s.
      Then a straightforward calculation using \eqref{H1-wavelet-norm},  Jensen's inequality, and Fubini
      gives
      \begin{equation*}
        \E \| \zeta_u - \zeta_v \|_{H^1(\TT)}
        \le \left( \E \| \zeta_u - \zeta_v \|_{H^1(\TT)}^2 \right)^{1/2}
        = \left( \E \tau_1^2 \right)^{1/2} \| u -v \|_{H^1(\TT)}.
      \end{equation*}
      Thus, Assumption~\ref{assumptions-on-K-beta}(c) holds with
      $\beta_c = \sqrt{\E \tau_1^2} = \sqrt{\frac{ \beta (\beta +5)}
        { 6}} $ where we used well-known  expressions for
       the
      second raw moment of  $\tau_1$ \cite{kotz-univariate-v2}. 
\end{expl}

\subsection{Statement of main results: Convergence of  RCAR-MH kernels}
\label{sec:main-results-convergence}
  
In this section we present our main theoretical results pertaining to the convergence properties
of RCAR-MH kernels $\P$ that together constitute Main Result~\ref{mainresult:spectral-gap}. 
Theorem~\ref{thm:lyapunov-function-for-RCAR} gives families of Lyapunov functions for $\P$.
This Lyapunov function is then used to define the family of semimetrics $\td_q$ as in \eqref{eq:2} with
respect to which a uniform spectral gap exists according to Theorem~\ref{thm:spectral-gap-locally-Lip}.
Finally, Theorem~\ref{thm:existence-uniqueness-of-invariant-measure} shows that $\P$ is Feller and hence has a unique invariant measure.

\begin{theorem} \label{thm:lyapunov-function-for-RCAR} 
  Let   $\P$  be an RCAR-MH kernel as in Definition~\ref{def:kernel-P} and
  suppose  Assumptions~\ref{assumptions-on-Phi}(c)
  and \ref{assumptions-on-K-beta}
  (a,b) are satisfied by the function $\Psi$ and the kernel $\K$ respectively, and
  that the innovation measure $\lambda \in P(\H)$ has bounded moments of integer degree $p \ge 1$.
  Then $V(u) = \|u\|^p$ is a Lyapunov function for $\P$.
\end{theorem}

\begin{remark}\label{rem:polynomial-Lyapunov-func}
  One can easily check that the above theorem further implies that any polynomial of
  the form $V(u)= \sum_{j=0}^p a_j \|u\|^j$ with coefficients $a_j \ge 0$ is also a Lyapunov function of $\P$.   
\end{remark}

We present the proof of Theorem~\ref{thm:lyapunov-function-for-RCAR} in Appendix~\ref{sec:proof-prop-Lyapunov-fun} using a direct argument akin to the proof of Lyapunov functions in \cite{hairer2014spectral}. 
This result states that the choice of the Lyapunov function  is tied to
 the tail decay of $\lambda$ so long as $\Psi$ is increasing in the tails
following Assumption~\ref{assumptions-on-Phi}(c).

The choice of the Lyapunov function $V$ is crucial since Proposition~\ref{thm:weak-harris}
gives the existence of a spectral gap in the $\td$ semimetrics which in turn depend on the
choice of Lyapunov functions; recall \eqref{tilde-d-definition}.
To this end, we introduce a family of semimetrics with respect to which $\P$ has a
uniform spectral gap.

For 
$\omega, \eta > 0$  define
\begin{equation}\label{def:distance-like-d-q}
d_q(u,v) :=  1 \wedge  \frac{ (1 + \eta\| u\| + \eta\| v\|)^q \| u - v \|}{\omega},
\end{equation}
and in turn the $V$-weighted semimetric
\begin{equation}
  \label{def:tilde-d-theta-q}
  \td_q(u,v) := \left[ d_q(u,v) (2 + \theta V(u) + \theta V(v)) \right]^{\frac{1}{2}},
\end{equation}
for any Lyapunov function $V$ of $\P$ 
following \eqref{tilde-d-definition}. Note, 
 $d_q(u,v)$ behaves similarly to
$ 1 \wedge \| u -v \| $, except that nearby points in $1 \wedge \| u -v \|$ become further away from
each other in $d_q$ as they get further away from the origin. Recall that $d_q$ and $\tilde d_q$ are just specific choices of the distance-like function $d$ in Definition \ref{def:distance-like-function} and its $V$-weighted analogue in \eqref{tilde-d-definition} that appear in the statement of the weak Harris theorem. In fact,  $\hd_q$ is the same metric introduced 
in \eqref{eq:2} with the choice  $V(u) = \| u \|^q$.
 With the $\td_q$ semimetric identified we can present our next result,
showing that RCAR-MH kernels have uniform  $\td_q$ spectral gaps.

\begin{theorem}\label{thm:spectral-gap-locally-Lip}
  Let   $\P$  be an RCAR-MH kernel as in Definition~\ref{def:kernel-P}
  and suppose 
  $\Psi$ satisfies Assumption~\ref{assumptions-on-Phi} 
  with  $q \ge 0$,
  $\K$ 
  satisfies Assumption~\ref{assumptions-on-K-beta}, and the innovation $\lambda \in P(\mcl H)$
  has bounded moments of integer degree $p \ge 1 \vee  \lceil q \rceil$.
  Let $\td_q$  be as in \eqref{def:tilde-d-theta-q} with a Lyapunov function $V$ of the form
\begin{equation}\label{def:polynomial-Lyapunov-func}
  V(u) = \sum_{j=0}^p a_j \| u\|^j,
\end{equation}
with  coefficients $a_p >0$ and  $a_j \ge 0$ for $0\le j <p$.
  Then there exist constants $(\theta, \omega, \eta, n, \gamma) \in (0, +\infty)^3 \times  \mbb{N}
  \times (0,1)$ such that
\begin{equation} \label{eq:SpectralGapGlobalLip}
\td_q(\mcl{P}^n \delta_u , \mcl{P}^n \delta_v ) \le \gamma \td_q(u,v).
\end{equation}
\end{theorem}

This theorem
is a detailed statement of Main Result~\ref{mainresult:spectral-gap}(a) once we
note that the bound \eqref{eq:SpectralGapGlobalLip} can
be readily extended from point masses $\delta_u, \delta_v$ to general measures
$\nu_1, \nu_2 \in P^1(\H; \td_q)$ by the following remark.

\begin{remark}\label{remark:point-mass-to-measure-generalization}
Observe that $\td_q: P^1(\H; \td_q) \times P^1(\H; \td_q) \mapsto \mbb{R}_+$ is
convex in both of its arguments and so,
\begin{equation*}
   \td_q(\P^n \nu_1, \P^n \nu_2)
   \le \int_{\H \times \H} \td_q(\P^n \delta_u, \P^n \delta_v)
   \pi(\dd u, \dd v),
 \end{equation*}
 for any coupling $\pi \in \Upsilon(\nu_1, \nu_2)$. In fact, this  is true  with
 $\td_q$ replaced with other transport distance-like functions  \cite[pg. 246]{hairer2011asymptotic}.
\end{remark}

 Complete proof of Theorem~\ref{thm:spectral-gap-locally-Lip}
 is given in Appendix~\ref{sec:proof-P-is-d-contracting}.
Our method of proof relies on the weak Harris' theorem (Proposition~\ref{thm:weak-harris}) which in
turn requires us to show that $d_q$ is contracting for $\P$ and that the level sets of $V$
are $d_q$-small; these are shown in Propositions~\ref{P-is-d-contracting-locally-Lipschitz}
and \ref{d-q-small-sets-locally-Lipschitz} respectively. To prove that $d_q$ is contracting for $\P$
we need to choose the parameters $\eta, \omega/\eta^q>0$ to be sufficiently small depending on
the constants appearing in Assumptions~\ref{assumptions-on-Phi} and \ref{assumptions-on-K-beta}
as well as the tail decay of $\lambda$. The integer $n $ then emerges in the proof
of  $d_q$-smallness of the $V$ level sets and depends on the choice of $V$ as
well as the parameters $\omega, \eta$ and the tail decay of $\lambda$.

For our third and final set of main theoretical results  we 
show that $\P$ has a unique invariant measure.
By \cite[Cor.~4.11]{hairer2011asymptotic}, the weak Harris theorem guarantees the existence of
a unique invariant measure if $\P$ is Feller  (recall Definition~\ref{def-feller-kernel})
and  there exists a complete metric $\bar d$ on $\mcl H$ such that $\bar d \le \sqrt{d_q}$.
Observe that for $q \ge 0$,
\begin{align*}
d_q(u,v) &= 1 \wedge \omega^{-1}  (1+\eta \|u\| + \eta \|v\|)^q\|u-v\| \\
\sqrt{d_q(u,v)} &\ge \sqrt{1 \wedge \omega^{-1} \|u-v\|},
\end{align*}
and the right side is a complete metric since $\|u-v\|$ is a complete metric.
Thus to prove existence and uniqueness of invariant measures of $\P$ we simply need
to verify that it is indeed a Feller kernel, this result constitutes
the claim in Main Result~\ref{mainresult:spectral-gap}(b) which we summarize in  Corollary~\ref{cor:P-has-unique-invariant-measure} below. 

\begin{theorem}\label{thm:existence-uniqueness-of-invariant-measure}
  Let   $\P$  be an RCAR-MH kernel as in Definition~\ref{def:kernel-P}
  and suppose Assumptions~\ref{assumptions-on-Phi}(a,b,d) and \ref{assumptions-on-K-beta}(a, c)
  are satisfied by 
  $\Psi$ and 
  $\K$, and  $\lambda \in P(\mcl H)$ has bounded moments of degree $q$. Then $\P$ is Feller. 
\end{theorem}

We prove this result in Appendix~\ref{sec:proof-existence-uniqueness-of-invariant-measure} using
direct arguments relying on the dominated convergence theorem and  the 
assumptions on $\K$ and $\Psi$.

\begin{corollary}\label{cor:P-has-unique-invariant-measure}
Suppose the conditions of Theorems~\ref{thm:lyapunov-function-for-RCAR}, \ref{thm:spectral-gap-locally-Lip}  and \ref{thm:existence-uniqueness-of-invariant-measure}  are satisfied. Then $\P$ has a unique invariant measure. 
\end{corollary}

\begin{remark}\label{remark-feller-reversibility}
It is important to note that existence of a unique invariant measure of $\P$  does not guarantee
that the invariant measure coincides with the target measure $\nu$ defined in \eqref{nu-definition}.
To ensure that $\nu$ is indeed the invariant measure
we still require $\P$ to be $\nu$-reversible, which in turn
holds whenever $\mcl Q$ is $\mu$-reversible  \cite{hosseini-RCAR}. Unfortunately the latter
is difficult to establish for general choices of $\mu$ and requires explicit balancing
of $\K$ and $\lambda$ to achieve $\mu$-reversibility of $\mcl Q$.
\end{remark}

\addtocounter{expl}{-1}
\begin{expl}[Continued]
    The measure $\lambda$ defined in \eqref{innovation-deconvolution}
  has bounded moments of all degrees \cite[Thm.~3.1]{hosseini-RCAR} and so we may take
  $V(u) = \| u\|_{H^1(\TT)}^p$ for any $p \ge 2$ as the Lyapunov
  function. Furthermore, we already verified that $\Psi$ satisfies Assumption~\ref{assumptions-on-Phi}(b) with 
  $q= 0$ and so we may choose the semimetrics,
  \begin{equation}\label{def-d0-td0}
      \begin{aligned}
      d_0(u,v) & = 1 \wedge \frac{\| u - v \|_{H^1(\TT)}}{\omega}, \\
      \td_0(u,v) &=  \left[  d_0(u,v) \big( 2 + \theta \| u\|^2_{H^1(\TT)} + \theta \| v \|^2_{H^1(\TT)} \big) \right]^{\frac12}.
    \end{aligned}
  \end{equation}
  Application of Theorems~\ref{thm:lyapunov-function-for-RCAR} through \ref{thm:existence-uniqueness-of-invariant-measure} 
  together with the weak Harris' theorem (Proposition~\ref{thm:weak-harris}) then yield that $\P$ has a 
  uniform $\td_0$ spectral gap for appropriate choices of the constants $\omega, \theta$. Moreover, $\P$ has a 
  unique invariant measure and by \cite[Thm.~3.4]{hosseini-RCAR} that invariant measure is precisely the Bayesian posterior $\nu$
  defined in \eqref{ex1-Bayesian-posterior}
 \end{expl}
 
\section{Perturbation theory for MH kernels} \label{sec:Approximations}

In this section we present our perturbation analysis of MH kernels
that amounts to a detailed statement of Main Result~\ref{mainresult:approximation}.
Our perturbation theory is much more general than the present application to RCAR-MH kernels, hence 
we present it for more general kernels $\P_0$ and
corresponding perturbations $\P_\veps$. 
As the notation suggests, the approximation parameter $\veps$ 
controls the quality of the approximating kernel -- akin to discretization resolution -- so
that $\P_\veps \to \P_0$ in an appropriate sense as $\veps \to 0$.

Our results are of practical
interest for two key reasons: First, simulation can only be done in finite dimensions and therefore 
approximations of the acceptance ratio $\alpha(u,v)$ are unavoidable in practice when $\H$ is a
function space, as is the case in Example~\ref{SSL-example}. Second, in some cases
 the innovation measure $\lambda$ or the kernel $\K$ may be intractable or costly to simulate.
 
The difficulty in obtaining approximation results using
semimetrics such as $\td_q$ in \eqref{def:tilde-d-theta-q}
is the fact that these semimetric do not satisfy the triangle inequality. However,
a ``weak'' triangle inequality is still satisfied; see Lemma~\ref{tilde-d-q-generalized-triangle-inequality}
below. Fortunately, the weak triangle inequality
allows us to bound the approximation error of $\P_\veps$ 
in a similar manner as if $\td_q$ were a  metric. Throughout this section, we will often use the generic notation $d$ and $\tilde d$ for a distance-like function and its $V$-weighted version, since these perturbation results hold in general for Markov kernels $\P_0$ that satisfy the weak Harris theorem. When we verify assumptions or apply results to RCAR-MH, we will make the specific choice of $d_q$ and $\tilde d_q$.

In Subsection~\ref{sec:pert-in-semim-d} we identify general assumptions on the
kernels $\P_0$ and $\P_\veps$ followed by our main
perturbation theorems in 
Subsection~\ref{sec:stat-main-results-perturbation-theory}, with the proofs postponed
to Appendix~\ref{sec:proof-main-approx-results}.

\subsection{Assumptions on the kernels \texorpdfstring{$\P_0$}{P0} and \texorpdfstring{$\P_\veps$}{Peps}}
\label{sec:pert-in-semim-d}

Let us first collect our assumptions on the MH kernel $\P_0$ and the distance-like function $d$. 
\begin{assumptions}\label{assumptions-on-P-and-td}
  Let $\P_0$ be a Markov transition kernel on $P(\mcl H)$. 
  Then one or more of the following  hold:
  \begin{enumerate}[label=(\alph*)]
    \item $d: \H \times \H \mapsto [0,1]$ is a distance-like function on $\H$. 
      \item $\P_0$ is contracting for $d$.
  \item $\P_0$ has a continuous Lyapunov function $V$.
  \item For $\theta > 0$ define $\td$ as in \eqref{tilde-d-definition}
    using $d$ and $V$. Then $\td$ satisfies a weak triangle inequality
    \begin{equation}\label{generic-weak-triangle-inequality}
    \tilde d(u,v) \le G \left[ \tilde d(u,w) + \tilde d(w, v)\right], \qquad \forall u,v,w \in \H.
  \end{equation}
where $G > 0$ is a uniform constant.
  \item  $\P_0$ has a unique invariant measure $\nu_0\in P^1(\H; \td)$.
  \item  There exists an integer $n \ge 1$ and a constant
  $\gamma \in (0,1)$ so that
  \begin{equation*}
    \tilde d( \P_0^n \delta_u, \P_0^n \delta_v )\le \gamma \tilde d(u,v), \qquad \forall u,v \in \H.
  \end{equation*}
  \end{enumerate}
\end{assumptions}

  Observe that conditions (b,c,f) are automatically satisfied if $\P_0$ satisfies the weak Harris' theorem
  (Proposition~\ref{thm:weak-harris}) and so by proving that $\P_0$ has a spectral gap one automatically
  verifies these assumptions. 
  Moreover, the $\td_q$ semimetrics
  defined in \eqref{def:tilde-d-theta-q} satisfy condition (d)
  by Lemma~\ref{tilde-d-q-generalized-triangle-inequality} below, and so the above assumptions on $\P_0$
  are naturally satisfied in the setting of Subsection~\ref{sec:main-results-convergence} and for RCAR-MH
  kernels.
\begin{lemma}\label{tilde-d-q-generalized-triangle-inequality}
  Define $\td_q$ as in \eqref{def:tilde-d-theta-q} and let
  $p,q\ge 0$ be  integers and  $V(u) = \sum_{j=0}^p a_j\| u\|^p$ with
  $a_p > 0$ and $a_j \ge 0$ for $j =0, \dots, p-1$.  
   Then
   there exists  $G(\theta, p, q, \omega, \eta, a_j) >0$ so that 
   \begin{equation}\label{weak-triangle-inequality}
\td_q(u,v) \le G \left( \tilde d_q(u,w) +  \tilde d_q(w,v) \right),
\qquad \forall u,v,w \in \mcl{H}.
\end{equation}
\end{lemma}
See  Appendix~\ref{sec:proof-tilde-d-q-generalized-triangle-inequality} for the proof.
Next we  collect  assumptions on the family of approximate kernels $\P_\veps$  following 
\cite{johndrow2017error}.

\begin{assumptions}\label{Assumption-on-P-eps}
  Let $\P_0$ be a Markov transition kernel on $P(\H)$ with Lyapunov function $V$, and
  let $d$ be a distance-like function on $P(\H)$ and define $\td$ using $d$ and $V$ as in \eqref{tilde-d-definition}.
   Then there exists a constant
  $\veps_0 >0$ so that the family of
   transition kernels $\P_\veps$ satisfy:
\begin{enumerate}[label=(\alph*)]
  \item For every $\veps \in (0, \veps_0)$ there exist constants
    $(\kappa_\veps, K_\veps) \in (0,1) \times (0, +\infty)$ so that
$$
\mcl{P}_\veps V(u) \le \kappa_\veps V(u) + K_\veps \qquad \forall u \in \H.
$$
\item There exists a bounded function $\psi(\veps): \mbb R_+ \mapsto \mbb R_+$ so that 
\begin{equation} \label{eq:OneStepError}
\tilde{d} (\P_0 \delta_u, \P_\veps \delta_u ) \le 
\psi(\veps) \left( 1 + \sqrt{V(u)} \right), \qquad \forall u \in \H.
\end{equation}
\end{enumerate}

\end{assumptions}

Note that our assumptions on  $\P_\veps$ are far less stringent in comparison to $\P_0$. Simply put
condition (a) requires $\P_0$ and $\P_\veps$ to have the same Lyapunov function while
(b) requires control on the one step error between $\P_0$ and $\P_\veps$.  In fact, $\P_\veps$
is not required to have a unique invariant measure or satisfy any contractive properties
directly. This flexibility allows access to larger classes of approximate kernels
and makes our perturbation theory conveninet to apply since there are fewer conditions to check. 
In comparison, more direct methods such as the perturbation analysis of \cite{hairer2014spectral}
prove the convergence properties of $\P_0$ and $\P_\veps$ separately and then show that the invariant
measures are close to each other.

\addtocounter{expl}{-1}
\begin{expl}[Continued]
    We now consider an approximation of the kernel $\P$ for the nonlinear regression problem 
    and verify the above assumptions.     
    Suppose the     likelihood potential $\Psi$ is replaced with the numerical approximation $\Psi \circ \Pi_N$ 
    where $\Pi_N$ denotes the $L^2$ projection on the span of the first $N$ wavelets 
     $\phi_0, \dots, \phi_N$. We let $\veps = 1/N$ and define the perturbed kernel $\P_\veps$
    as the identical RCAR-MH kernel to $\P$ with the accept/reject probability 
    \begin{equation*}
        \alpha_\veps(u, v) := 1 \wedge \exp( \Psi(\Pi_N u) - \Psi(\Pi_N v) ).
    \end{equation*}
    A straightforward application of Theorem~\ref{thm:lyapunov-function-for-RCAR} reveals that both $\P$ and $\P_\veps$ share the same 
    Lyapunov functions of the form  $V(u) = \sum_{j=1}^p a_j\| u\|_{H^1(\TT)}^p$ with positive coefficients $a_j$ and constants $\kappa_\veps = \kappa, K_\veps = K$, thus
     Assumption~\ref{Assumption-on-P-eps}(a) is verified easily.
    
    Assumption (b) can be verified via a coupling argument. Fix $u \in H^1(\TT)$ and  let 
    $\pi_u \in \Upsilon( \P \delta_u, \P_\veps \delta_u)$
    be an optimal coupling that achieves $d_0( \P \delta_u, \P_\veps \delta_u)$. Then by Cauchy-Schwartz we have that 
    \begin{equation}\label{coupling-argument-example-1}
        \begin{aligned}
        \td_0 &(\P \delta_u,  \P_\veps \delta_u)^2 \\
        & \le \int_{H^1(\TT) \times H^1(\TT)} d_0(v,w) ( 2 + \theta V(v) + \theta V(w) ) \pi_u(\dd v, \dd w), \\
        & \le d_0(\P \delta_u, \P_\veps \delta_u) \int_{H^1(\TT) \times H^1(\TT)}  ( 2 + \theta V(v) + \theta V(w) ) \pi_u(\dd v, \dd w), \\
        & \le d_0(\P \delta_u, \P_\veps \delta_u)  \left[ 2 + 2\theta ( \kappa V(u) + K) \right].
        \end{aligned}        
    \end{equation}
    It is thus sufficient to bound $d_0(\P \delta_u, \P_\veps \delta_u)$. Now consider a coupling of $\P \delta_u$ and $\P_\veps \delta_u$ where 
    both chains propose a new point $u^\ast = \zeta_u + \xi  $ where $\zeta_u \sim \K(u, \cdot)$ and $\xi \sim \lambda$. We then generate 
    a uniform random variable $\varsigma$ and  one chain accepts $u^\ast$ when $\varsigma < \alpha(u, u^\ast)$ 
    while the other chain accepts if $\varsigma < \alpha_\veps(u, u^\ast)$. Since this coupling is not necessarily optimal 
    we have 
    \begin{equation}\label{ex1-num-approx-intermid-display}
    \begin{aligned}
         d_0(\P \delta_u, \P_\veps \delta_u) 
         & \le  \E \left[ d_0(u^\ast, u^\ast)  \mathbb P(\text{both chains accept}) \right] \\ 
         & \quad +  \E \left[ d_0(u^\ast, u)   \mathbb P(\text{only one chain accepts}) \right] \\ 
         & \quad +  \E \left[ d_0(u, u)   \mathbb P(\text{both chains reject}) \right], \\ 
         & \le \E \left[ \mathbb P(\text{only one chain accepts}) \right].
    \end{aligned}
    \end{equation}
    Moreover, since $(1 \wedge \exp)$ is Lipschitz and we already showed that $\Psi$ is globally Lipschitz we have 
    \begin{equation*}
    \begin{aligned}
        \mathbb P(\text{only one chain accepts}) & \le | \Psi(u^\ast) - \Psi(\Pi_N u^\ast)| \le L \| u^\ast - \Pi_N u^\ast\|_{H^1(\TT)}, \\
         &\le L \| I - \Pi_N \|_{H^1(\TT)} \| u^\ast \|_{H^1(\TT)}.
    \end{aligned}
    \end{equation*}
    Now since $\| \zeta_u\|_{H^1(\TT)} \le \| u \|_{H^1(\TT)}$ by Assumption~\ref{assumptions-on-K-beta} and triangle inequality 
    we have $\| u^\ast \|_{H^1(\TT)} \le \| u \|_{H^1(\TT)} + \| \xi \|_{H^1(\TT)}$. Substituting back into \eqref{ex1-num-approx-intermid-display} 
    gives 
    \begin{equation*}
    \begin{aligned}
        d_0(\P \delta_u, \P_\veps \delta_u)
        & \le L \| I - \Pi_N \|_{H^1(\TT)}  \E \left[ \| u\|_{H^1(\TT)} + \| \xi \|_{H^1(\TT)} \right], \\
        & \le C_1 \| I - \Pi_N \|_{H^1(\TT)} (1 + \| u\|_{H^1(\TT)}),
    \end{aligned}
    \end{equation*}
    for some constant $C_1 > 0$. To this end we have shown that 
    \begin{equation*}
        \td_0(\P \delta_u, \P_\veps \delta_u)^2 \le C_2 \| I - \Pi_N\|_{H^1(\TT)} ( 1 +  \tilde{V}(u)),
    \end{equation*}
    where $C_2 >0$ and $\tilde V(u) =  \left(  1 + \| u\|_{H^1(\TT)} + \| u \|_{H^1(\TT)}  V(u) \right)$.
    Noting that $\tilde V$ is also a polynomial of $\| u\|_{H^1(\TT)}$ with positive coefficients then verifies 
    Assumption~\ref{Assumption-on-P-eps}(b) with Lyapunov function $\tilde V$ and $\psi(\veps) = C_3 \| I - \Pi_N \|_{H^1(\TT)}^{1/2}$ 
    for a  constant $C_3> 0$.
\end{expl}

\subsection{Statement of main results: Perturbation theory for MH kernels
on Banach spaces}\label{sec:stat-main-results-perturbation-theory}

We are now ready to present our main theoretical results pertaining to perturbations of
Markov kernels on Banach spaces. 
Our first result is a generalization of \cite[Thm.~3.13]{johndrow2017error}  to Banach spaces
that allows us to bound the distance between the 
the invariant measure(s) of $\P_0$ and $\P_\veps$. A similar argument is used to bound
the distance between invariant measures of a Markov transition kernel and its perturbation
in \cite{hairer2011asymptotic}. 

\begin{theorem}\label{thm:approximation-result-in-td}
  Suppose Assumptions~\ref{assumptions-on-P-and-td} and
\ref{Assumption-on-P-eps} are satisfied. 
Then there exists
a function $\vartheta(n): \mbb{N} \mapsto [0,\infty)$
so that for each $k > 0$ such that $\gamma^{\lfloor k/n \rfloor} < G^{-1}$  we have
\begin{equation}\label{errorbound:thm:approximation-result-in-td}
   \tilde d(\nu_0,\nu_\veps) \le
   \frac{G \psi(\veps) \vartheta(k)}{1-G \gamma^{\lfloor k/n \rfloor}}
   \left(1+\nu_0\left( \sqrt{V}\right) \wedge \nu_\veps\left( \sqrt{V}\right)\right)
   \qquad \forall u,v \in \H,
 \end{equation}
 where $\nu_0,\nu_\veps \in P^1(\H; \td)$ are 
 the invariant measure of $\P_0$ and any invariant measure of $\mcl{P}_{\veps}$, respectively.
\end{theorem}
We  prove this theorem in Appendix~\ref{sec:proof-perturbation-tilde-d}.
Note that the invariant measure $\nu_\veps$ need not be unique, though the Lyapunov condition for
$\P_\veps$ in Assumption \ref{Assumption-on-P-eps}(a) guarantees there exists at least one.
The function $\vartheta(n)$ appears in our bound due to the fact that $\td$ satisfies the weak
triangle inequality \eqref{weak-triangle-inequality} with a constant $G >0$ that is possibly bigger than 1.
To overcome this difficulty we choose $k >n$ sufficiently large so that
 $G \gamma^{\lfloor k/n \rfloor} < 1$, where
  $\gamma$ is the $n$ step spectral gap of $\mcl P_0$ in
  Assumption~\ref{assumptions-on-P-and-td}(b). We then take
  $\vartheta(k) =  C \sum_{j=1}^k G^j (C^\ast \gamma)^{\lfloor j/n  \rfloor} < \infty$ with constants $C^\ast, C >0$
 depending on the Lipschitz constant of $\P_0$ and the growth rate of $V$. Noting that $G \vartheta(k)/(1 - G \gamma^{\lfloor k/n \rfloor}) < +\infty$
 is a constant independent of $\veps$  and by taking $\hd_q$ as
 our semimetric we obtain the detailed version of Main Result \ref{mainresult:approximation}(a).

We can further extend the error bound \eqref{errorbound:thm:approximation-result-in-td} to
a practical error bound between the expectation of $\td$-Lipschitz functions
under $\nu_0$ and $\nu_\veps$. 

\begin{corollary}
  Suppose the conditions of Theorem~\ref{thm:approximation-result-in-td} are satisfied. 
  Let $\mathcal{X}$ be a separable Hilbert space with norm $\| \cdot \|_{\X}$ and consider 
  a $\td$-Lipschitz function $f : \H \to \mathcal{X}$ satisfying 
      $\| f(u) - f(v) \|_{\mathcal{X}} \le \td(u,v) \: \forall u,v \in \H$.
 Then we have,
\begin{equation}\label{errorbound:td-lipschitz-functions}
  \| \nu_0(f) - \nu_\veps(f) \|_{\mathcal{X}} \le \frac{G \vartheta(k) }{1 - G \gamma^{\lfloor k/n \rfloor}}
  \psi(\veps) \left(1 +  \nu_0 \left( \sqrt{V}  \right) \wedge\nu_\veps \left( \sqrt{V}  \right) \right).
\end{equation}
\end{corollary}

\begin{proof} Since the argument is short we present it here. Let $\pi$ be an optimal coupling between $\nu_0$ and $\nu_\veps$
which exists following \cite[Thm.~4.1]{villani-OT}. 
Then 
  \begin{align*}
 \int_{\mcl{H} \times \mcl{H}} \tilde{d}(u,v) \pi(\dd u,\dd v) & \ge  
\int_{\mcl{H} \times \mcl{H}}  \|f(u) - f(v)  \|_\X \pi(\dd u,\dd v),  \\
& \ge \left\| \int_{\mcl{H}}  f(u) \dd \nu  -  \int_{\mcl{H}} f(v)   \nu_\veps(\dd v) \right\|_\X.
  \end{align*}
The last step follows from Jensen's inequality.
\end{proof}

Note that the assumption that $f$  is $\td$-Lipschitz is
not very restrictive given that  such an $f$ is continuous but $\| f (u) \|_{\mathcal{X}}$ 
 can grow as fast as the Lyapunov function $V(u)$. 

  We continue to derive practical error bounds pertaining to Markov kernels and their
  perturbations, turning our attention to pathwise properties of realizations of the  Markov chains.  
More precisely we 
bound the error of finite-time Ces\'{a}ro averages from $\P_\veps$
and expectations under $\nu_0$ for real valued $\td$-Lipschitz functions. Our bounds are desirable as
they are a major improvement over standard arguments using the weak triangle inequality. This 
is a consequence of the fact that the $\td$-Lipschitz seminorm $\lvertiii \cdot \rvertiii_{\td}$
obeys the triangle inequality even when $\td$ does not, indeed for functions $f,g: \H \mapsto \X$,
\begin{equation} \label{eq:LipschitzTriangle}
\begin{aligned}
  \lvertiii f+g \rvertiii_{\td}& = \sup_{u \ne v} \frac{\|f(u) + g(u) - f(v) - g(v)\|_{\X}}{\td(u,v)}\\
  &\le  \sup_{u \ne v} \frac{\|f(u) - f(v)\|_{\X} + \|g(u) - g(v)\|_{\X}}{\td(u,v)}
  = \lvertiii f \rvertiii_{\td} + \lvertiii g \rvertiii_{\td}.
\end{aligned}
\end{equation}
Our next main result  bounds the mean error of pathwise estimates from
$\P_\veps$ with respect to expected values under the exact target $\nu_0$.

  \begin{theorem} \label{thm:VariationBound}
    Suppose  Assumptions~\ref{assumptions-on-P-and-td} and
    ~\ref{Assumption-on-P-eps} are satisfied and let $\X$ be a separable Hilbert space. Then for $U_k^\veps 
\sim  \P_\veps^{k-1} \delta_{u_0}$ and for any function $\varphi: \H \mapsto \mc X$ with 
$\lvertiii \varphi \rvertiii_{\td} < +\infty$ we have:
\begin{enumerate}[label=(\alph*)]
\item there exist positive constants
  $C_1, C_2, C_3 >0$ that are  
independent of  $n$ but depend on $\theta, \kappa_\veps, K_\veps, \gamma$ and $V(u_0)$ such that
\be 
  \E \left\| \frac1n \sum_{k=0}^{n-1} \varphi(U_k^\veps) - \nu_0(\varphi)  
\right\|_{\X} \le  \frac{\lvertiii \varphi \rvertiii_{\td}}{1-\gamma} 
\left( C_2 \psi(\veps) 
+ C_3 \frac{\psi(\veps)}{n} + C_4\frac{1}{\sqrt{n}}  \right).
\ee

\item there exist constants
  $C_4, C_5, C_6>0$
  independent of $n$ but depending on $\theta, \kappa_\veps$ and $K_\veps$ such  that
  \begin{equation*}
  \left\|  \bb E \frac1n \sum_{k=0}^{n-1} \varphi(U_k^\veps) - \nu_0(\varphi) \right\|_{\mc X}
  \le \frac{\lvertiii \varphi \rvertiii_{\td}}{1-\gamma}  \left( C_4 \psi(\veps) + C_5 \frac{\psi(\veps)}n + \frac{C_6}{n} \right).
 \end{equation*}
\end{enumerate}

\end{theorem}

The complete proof is given in Appendix~\ref{sec:proof-variation-bound-poisson}. To prove part (a)
we utilize an approach similar to that of \cite{GlynnMeyn1996}. The
complication is that because $\varphi$ is $\mc X$-valued rather than $\bb R$-valued, we need to 
prove that the potential $\sum_{k=0}^{\infty} \P^k \varphi$ is a solution to the Poisson
equation; this result is of course well-known for real-valued $\varphi$. 
The inner product structure on $\X$ is used only once, to control the expected $\|\cdot\|_{\mc X}$-norm
of a Martingale. While it is possible to control this term without the inner product structure,
in most applications in statistics and Bayesian inverse problems the functions of interest
are Hilbert-space valued, so the result above is sufficiently general. Let us now return to Example~\ref{SSL-example}
once more to apply Theorems~\ref{thm:approximation-result-in-td} and \ref{thm:VariationBound} to obtain error bounds
on the approximate posteriors and  the posterior C{\'e}saro  average.


\addtocounter{expl}{-1}
\begin{expl}[Continued]
We already verified that RCAR-MH kernel $\P$ and the approximate kernel $\P_\veps$ obtained by discritizing the 
likelihood potential $\Psi$ via projection onto   the first $N$ wavelet bases. Then a direct
application of Theorem~\ref{thm:approximation-result-in-td} yields a bound of the form 
\begin{equation*}
    \td_0( \nu, \nu_\veps) \le C_1 \| I - \Pi_N \|_{H^1(\TT)}^{1/2},
\end{equation*}
where $C_1 >0$ is a  constant independent of $N$ and we recall that we defined $\veps = 1/N$ and $\nu$ denotes the true posterior 
measure.  Let us also consider the function $\varphi: u \mapsto u$ 
and apply Theorem~\ref{thm:VariationBound}(a) to obtain an error bound on the pathwise Ces{\'a}ro averages of the discretized RCAR-MH algorithm:
\begin{equation*}
    \E \left\| \frac{1}{n} \sum_{k=0}^{n-1} U^\veps_k - \nu(\varphi)   \right\|_{H^1(\TT)}
    \le C_2 \left( \| I - \Pi_N \|_{H^1(\TT)}^{1/2} +  \frac{1}{\sqrt{n}}     \right),
\end{equation*}
where $C_2 >0$  is independent of $N, n \ge 1$. To this end, both the error between the invariant measures and 
the Ces{\'a}ro averages of the RCAR-MH algorithm are controlled by the square root of the $H^1(\TT)$-operator norm of $ I - \Pi_N$.
The Ces{\'a}ro average is also controlled by the standard Monte Carlo rate $n^{-1/2}$. 
{}
\end{expl}

\section{Applications}\label{sec:applications}

Here we discuss a number of applications  of our main theoretical results from
Sections~\ref{sec:spectral-gaps} and \ref{sec:Approximations} with a particular focus on approximations
of the RCAR algorithm. We start in Subsection~\ref{sec:nonlinear-regression} by providing a detailed explanation of the numerical 
experiments presented in Figure~\ref{fig:nonlinear-regression-intro} and pertaining to 
Example~\ref{SSL-example}. In Subsection~\ref{sec:pcn-approximation} we consider an application of the pCN 
algorithm where the Karhunen-Lo{\'e}ve modes of the prior are perturbed. Finally, in Subsection~\ref{sec:deconvolution-circle}
we discuss the practicality of Assumption~\ref{assumptions-on-Phi}(c).

\subsection{Details of numerical experiments in Figure~\ref{fig:nonlinear-regression-intro}}
\label{sec:nonlinear-regression}

To generate Figure~\ref{fig:nonlinear-regression-intro} we utilized the RCAR algorithm of \cite{hosseini-RCAR} tailored for gamma random variables
(see in particular Algorithm 6 of that article). The function $u^\dagger$ depicted in Figure~\ref{fig:nonlinear-regression-intro}(a) 
is the function obtained by setting the first, fourth, eighth and sixteenth DB12 wavelet coefficients equal to $2$ 
while the rest of the coefficients are zero. That is, our $u^\dagger$  is sparse in the DB12 wavelet basis 
and its nonzero coefficients are positive so that it is consistent with our choice of the prior $\mu$ which constrains the 
wavelet coefficients to be positive. The Posterior mean in Figure~\ref{fig:nonlinear-regression-intro}(a) was computed by 
truncating the prior at $N=128$ wavelet modes and running the chain for $10^5$ iterations with parameter $\beta = 0.9$ and with a 
burn-in of $5 \times 10^4$. 

Figure~\ref{fig:nonlinear-regression-intro}(b) was generated by varying $N, \beta$ over the indicated ranges and running the RCAR 
algorithm for the same data $y$ shown in Figure~\ref{fig:nonlinear-regression-intro}(a). Each data point on Figure~\ref{fig:nonlinear-regression-intro}(b)
was generated by running the chain for $10^5$ iterations with burn-in of $5 \times 10^4$ over five restarts of the chain with 
random initial conditions from the prior. The acceptance ratios for the five restarts were then averaged to obtain a data point 
in the figure. The restarts were performed to reduce the effect of the initial condition of the chain and other random effects on the reported 
values.

\subsection{pCN with approximate Karhunen-Lo{\'e}ve expansions}
\label{sec:pcn-approximation}

We now consider a perturbation example where the pCN algorithm of \cite{stuart-mcmc} is applied with a perturbed prior covariance.
More precisely, consider the same nonlinear regression problem as Example~\ref{SSL-example} but this time we
wish to recover $u^\dagger \in \H^s(\Omega)$ where $\Omega \subset \mbb R^d$ and $\H^s(\Omega)$ is a Sobolev-type space.
Let $\phi_j$ be the Neumann eigenfunctions (normalized in $L^2(\Omega)$) of the standard Laplacian operator on $\Omega$, i.e., they solve the problems
\begin{equation*}
\begin{aligned}
    - \Delta \phi_j & = c_j  \phi_j, \qquad  && \text{in } \Omega, \\
     \nabla \phi_j \cdot \mbf{n} & = 0, \qquad && \text{on } \partial \Omega,
\end{aligned}
\end{equation*}
where $\mbf{n}$ is the unit outward pointing normal vector on $\partial \Omega$ and the $c_j \ge 0$ are the 
eigenvalues of $\Delta$. Indeed, one can verify that $\Delta$ is positive semi-definite and self-adjoint and so $c_0 = 0$, with the corresponding eigenfunction $\phi_0$ being a constant on $\Omega$,
while the $c_j >0$ for $j \ge 1$.
Now  for integer $s > 0$ consider the spaces $\H^s(\Omega) \subset L^2(\Omega)$ defined as 
\begin{equation*}
    \H^s(\Omega) := \left\{ u \in L^2(\Omega) : \| u \|_{\H^s(\Omega)}^2 :=  \sum_{j=0}^\infty (1 +  c_j)^s \langle u, \phi_j \rangle_{L^2(\Omega)}^2 < +\infty    \right\}.
\end{equation*}
It is known (see \cite[Lem.~7.1]{dunlop2020large}) that for any  $s> 0$ it holds that $\H^s(\Omega) \subset H^s(\Omega)$ where 
$H^s(\Omega)$ denotes the standard Sobolev space of index $s$ on $\Omega$. Then an application of the Sobolev embedding 
theorem \cite{adams} yields $\H^s(\Omega) \subset C(\Omega)$ for $s > d/2$. Thus the nonlinear regression problem  in Example~\ref{SSL-example}
is well-defined for functions $u^\dagger \in \H^s(\Omega)$ for $s > d/2$; which we assume holds henceforth. An identical reasoning to 
Example~\ref{SSL-example} then verifies Assumption~\ref{assumptions-on-Phi}.

Now define the prior measure $\mu$ as 
\begin{equation*}
    \mu = \Law\left\{ \sum_{j =0}^\infty a_j \eta_j \phi_j \right\}, 
\end{equation*}
where $\eta_j \iidsim N(0,1)$ and $a_j = (1 + c_j)^{-k}$ with $k > s$. Let $u \sim \mu$ and write
\begin{equation*}
    \| u\|_{\H^s(\Omega)}^2 = \sum_{j=0}^\infty ( 1 +  c_j)^{-k} \eta_j^2.
\end{equation*}
Observe that $\E (1 + c_j)^{-k} \eta_j^2 = (1 + c_j)^{-k}$ and $\E ( 1+ c_j)^{-2k} \eta_j^4 = 3 (1 + c_j)^{-2k}$.
By Weyl's law $c_j \asymp j^{2/d}$ so that $(1 + c_j)^{-k} \asymp j^{-2k/d}$. Since we assumed that $ k > s > d/2$ we infer 
that $\sum_{j=0}^\infty ( 1 + c_j)^{-k} < +\infty $ and $ \sum_{j=0}^\infty (1 + c_j)^{-2k} <+ \infty$. Kolmogorov's two series 
theorem then yields that $ \| u \|_{\H^s(\Omega)}^2 < +\infty$ a.s. So our prior is supported on $\H^s(\Omega)$ as desired. 

Since $\mu$ is Gaussian our RCAR-MH algorithm reduces to the pCN algorithm, i.e., for a step size $\beta \in (0,1)$ we 
take the kernel $\K(u, \cdot) = \delta_{\beta u}$ and the innovation measure
\begin{equation*}
    \lambda = \Law\left\{ \sum_{j =0}^\infty a_j \xi_j \phi_j   \right\}, \quad \xi_j \iidsim N(0, (1 - \beta^2) ).
\end{equation*}
These choices yield the pCN kernel 
\begin{equation}\label{pcn-lazy-chain-kernel}
\begin{aligned}
    \P(u, \dd v)  = & \alpha(u,v) (\delta_{\beta u} \ast \lambda) (\dd v)  \\
    & + \delta_u \int_{\H^1(\Omega)} \big(1 - \alpha (u,w)\big) \big( \delta_{\beta u} \ast \lambda\big) (\dd w). 
\end{aligned}
\end{equation}
Since pCN is a special case of RCAR we readily verify, by the same calculations presented for Example~\ref{SSL-example}, that 
pCN satisfies the conditions of the weak Harris' theorem 
and so has a $\tilde{d}_0$-spectral gap, where we recall the semimetrics $d_0, \tilde{d}_0$  defined in \eqref{def-d0-td0} with the 
$H^1(\TT)$ norms replaced with $\H^s(\Omega)$ norms and possibly different constants $(\theta, \omega)$.

Let us now consider a perturbation of pCN by replacing the eigenpairs $(c_j, \phi_j)$ with  perturbations $(c_j^\veps, \phi_j^\veps)$
for a parameter $\veps >0$. We have in mind applications where we can only compute $(c_j, \phi_j)$ numerically, using for example 
a finite element method, since the domain $\Omega$ can have complicated geometry. We further assume for brevity  that there 
exists a sufficiently small constant $\veps_0 >0$ so that for all  $\veps \in (0, \veps_0)$ we have 
$c_j^\veps \asymp j^{2/d}$ and 
the $\phi_j^\veps$ are 
normalized in $L^2(\Omega)$ and linearly independent  such that $\text{span} \{ \phi_j^\veps \} \subseteq \H^1(\Omega)$.

Our goal is to obtain an 
error bound between the true posterior $\nu$ and the limit distribution $\nu_\veps$ of the perturbation of pCN 
that utilizes the eigenpairs $( c_j^\veps, \phi_j^\veps)$ rather than the exact pairs $(c_j, \phi_j)$.
 To this end, define the perturbed innovation measure
\begin{equation*}
    \lambda_\veps = \Law\left\{ \sum_{j =0}^\infty a_j^\veps \xi_j \phi_j^\veps   \right\}, \quad \xi_j \iidsim N(0, (1 - \beta^2) ),
\end{equation*}
where $a_j^\veps = (1 + c_j^\veps)^{-k}$
as well as the corresponding perturbed pCN kernel 
\begin{equation}\label{pcn-lazy-chain-kernel-perturbed}
\begin{aligned}
    \P_\veps(u, \dd v)  = & \alpha(u,v) (\delta_{\beta u} \ast \lambda_\veps) (\dd v)  
    \\ & + \delta_u \int_{\H^1(\Omega)} \big(1 - \alpha (u,w)\big) \big( \delta_{\beta u} \ast \lambda_\veps \big) (\dd w).
\end{aligned}
\end{equation}

Repeating the same calculation we did for $\mu$ in the above yields that 
 $\lambda_\veps$ is a Gaussian measure supported on  $\H^1(\Omega)$ for all $\veps \in (0, \veps_0)$ and 
 so  has bounded moments of all orders and so by Theorem~\ref{thm:lyapunov-function-for-RCAR} any function of the 
 form $V(u) = \| u\|_{\H^1(\TT)}^p$ for $p \ge 1$ is a 
Lyapunov function for $\P_\veps$ and so Assumption~\ref{Assumption-on-P-eps}(a) is satisfied. Thus it remains to verify 
Assumption~\ref{Assumption-on-P-eps}(b) before we can apply Theorem~\ref{thm:approximation-result-in-td} to bound 
$\td_0(\nu, \nu_\veps)$. Following this argument, we obtain the following proposition, the proof of which is postponed to 
Appendix~\ref{sec:proof-main-application-results}.

\begin{proposition}\label{pcn-perturbation-intermid-prop}
Consider the above setting with the pCN kernel $\P$ as in \eqref{pcn-lazy-chain-kernel} and the perturbation $\P_\veps$ introduced in \eqref{pcn-lazy-chain-kernel-perturbed}. Suppose that the following conditions hold: 
\begin{enumerate}[label=(\alph*)]
    \item There exists a common Lyapunov function $V$ for $\P, \P_\veps$ so that
    \begin{equation*}
    \P V(u) \le \kappa V(u) + K, \qquad \P_\veps V(u) \le \kappa_\veps V(u) + K_\veps, \qquad \forall u \in \H^1(\Omega),
\end{equation*}
and furthermore 
\begin{equation*}
\kappa \vee \sup_{\veps \in (0, \veps_0)} \kappa_\veps \in (0,1) \qquad \text{and} \qquad  K \vee  \sup_{\veps \in (0, \eps_0)} K_\veps \in [0, + \infty).
\end{equation*}
\item It holds that the sequences 
$ \{a_j \| \phi_j \|_{\H^s(\Omega)}^2\},  \{ a_j \| \phi_j - \phi_j^\veps \|_{\H^s(\Omega)}\}$ and $\{ |a_j - a_j^\veps| \| \phi_j \|_{\H^s(\Omega)}\}$
belong to $\ell^1$.

\end{enumerate}
Then $\forall \veps \in (0, \veps_0)$ and for any $u \in \H^s(\Omega)$ it holds that 
\begin{equation*}
    \td_0(\P \delta_u, \P_\veps \delta_u)^2 \le C \psi(\veps)  \left[ 1 + V(u) \right], \qquad 
    \forall u \in H^1(\TT),
\end{equation*}
where $C> 0$ is a constant independent of $\veps$ and  
\begin{equation*}
    \psi(\veps) =  \left[  \left( \sum_{j=0}^\infty  a_j  \frac{ \| \phi_j - \phi_j^\veps  \|_{\H^s(\Omega)}^2}{ \| \phi_j \|_{\H^s(\Omega)}^2} \right)^{1/2} 
     + \left( \sum_{j=0}^\infty   \frac{ | a_j - a_j^\veps  |^2}{a_j^2} \| \phi_j\|_{\H^s(\Omega)}^2 \right)^{1/2}  \right].
\end{equation*}
\end{proposition}

An application of Theorem~\ref{thm:approximation-result-in-td} then yields  the existence of a constant $C >0$ so that 
\begin{equation*}
    \td_0(\nu, \nu_\veps) \le C \sqrt{ \psi(\veps)}
\end{equation*}
which is the desired result. 

\begin{remark}
The above proposition identifies conditions on approximations 
schemes for the eigenpairs $\{ a_j, \phi_j\}$ in connection with our choice of the prior $\mu$. Most notably, 
the condition that $\{ | a_j - a_j^\veps | \| \phi_j \|_{\H^s(\Omega)} \in \ell^1$ requires the absolute error in 
computing the eigenvalues $c_j$ to decay rapidly since the $\{\| \phi_j \|_{\H^s(\Omega)}\}$ is not summable (higher 
frequency eigenfunctions have larger Sobolev norms). However, one can get around this difficulty simply by 
prescribing a different sequence $a_j$ that can be implemented exactly. For example, by taking $a_j = (1 + b_j)^{-k}$ 
for another sequence of numbers $b_j \asymp j^{-2/d}$. It  can be verified that the resulting prior will still be a Gaussian supported on $\H^s(\Omega)$
but with a different covariance operator.
The conditions $(\{ a_j \| \phi_j \|_{\H^s(\Omega)}^2 \}, \{ a_j \| \phi_j - \phi_j^\veps \|_{\H^s(\Omega)}\} ) \in \ell^1$ can be viewed as guidelines for choosing the index $k > s$ according to the regularity of the $\phi_j$ and accuracy of our numerical scheme for computing 
the $\phi_j^\veps$. Note that we expect the sequence $\| \phi_j - \phi_j^\veps\|_{\H^s(\Omega)}$ to grow since the error of 
standard numerical schemes for computing eigenfunctions grows with their frequency due to their growing $\H^s(\Omega)$ norms. Thus, choosing a larger index $k$ allows us to control this approximation error.
\end{remark}

\begin{remark}
Note that the above bound can also be viewed as an error bound between two posteriors $\nu, \nu_\veps$ that arise from two Gaussian priors $\mu$ and $\mu_\veps$. 
Indeed, our calculations yield a method  
for controlling  the distance between posterior measures in terms 
of prior perturbations, a contemporary topic in the theory of Bayesian inverse problems \cite{sprungk2020local}. 
Admittedly, our method is inefficient as it goes through the construction of a Markov chain that converges to the two 
posteriors.
Regardless, such posterior 
perturbation bounds are often difficult to achieve, 
essentially due to the Feldman-Hajek theorem \cite[Sec.~2.7]{bogachev-gaussian} which implies  that perturbations of Gaussian prior measures can often lead to mutually singular 
priors and in turn mutually singular posteriors. 
Classic stability analyses of Bayesian inverse problems utilize TV or Hellinger distances 
\cite{stuart-bayesian-lecture-notes, hosseini-sparse, hosseini-convex, sullivan} and, due to the singularity of 
the posterior measures, one has no hope of obtaining a useful error bound in those topologies. 
Our calculations  above, and similarly the results of \cite{sprungk2020local}, suggest that transport (semi-)metrics 
hold the key for stability analysis of posterior measures due to prior perturbations. 
\end{remark}

   \subsection{Practicality of Assumption~\ref{assumptions-on-Phi}}\label{sec:deconvolution-circle}
   We dedicate this subsection to a discussion of the relevance of
   Assumption~\ref{assumptions-on-Phi}
     in practical applications. The conditions (a,b) and (d) are standard in the
     theory of Bayesian inverse problems and can be verified for large
     classes of inverse problems such as deconvolution, phase retrieval,
     porous medium flow, etc \cite{stuart-bayesian-lecture-notes, hosseini-sparse, hosseini-convex, stuart-acta-numerica, sullivan}. The condition (c) however is not crucial to ensure the
     existence and uniqueness of the target measure $\nu$, but it is central
     to Theorems~\ref{thm:lyapunov-function-for-RCAR} and \ref{thm:spectral-gap-locally-Lip}.
     We need this condition to make sure that $\Psi$ is uniformly increasing
     when sufficiently far from the origin. Intuitively this  means that if the chain
     is far away then the probability of accepting a proposal that is even
     farther away  decays uniformly.
     While we verified Assumption~\ref{assumptions-on-Phi}(c)
     for Example~\ref{SSL-example},
     it does not hold even in simple linear inverse problems, as we now demonstrate
     with an example in deconvolution \cite{hosseini-RCAR, vogel}.
     
      Let $\mcl H = H^1(\mbb T)$ once more and  consider
      $\mcl G: H^1(\mbb T) \mapsto \mbb R^m$ a bounded linear operator
     of the form
     \begin{equation*}
       (\mcl G(u))_j := (g \ast u)(x_j),
     \end{equation*}
     with $g \in C^\infty(\mbb T)$  a smooth kernel and distinct points $x_j \in \mbb T$.
     Let $u^\dagger \in H^1(\mbb T)$ be the ground truth function giving rise to the data
     $y_j = \mcl G(u^\dagger) +  \eps_j$ where $\eps_j \sim N(0,1)$.  These assumptions
     induce the quadratic likelihood potential
     \begin{equation*}
       \Psi(u;y) : = \frac{1}{2} \| \mcl G(u) - y \|_2^2.
     \end{equation*}
     
     In light of the smoothing effect of  $(g \ast \cdot)$ we can readily see that
     Assumption~\ref{assumptions-on-Phi}(c) cannot be verified: Let $u$ be a point that has
     large $H^1(\mbb T)$ norm and evaluate $\Psi(u)$.
     Then add to $u$ a highly oscillatory function $\delta u$ with small amplitude 
     that will increase the Sobolev norm of $u$ significantly. Since  convolution is linear
     we  have  $g \ast (u + \delta u)
     = g \ast u + g \ast \delta u$ and the perturbation $(g \ast \delta u)(x_j)$ to the
     observed data $y| u$ will be small;
     meaning that $\Psi(u + \delta u)$ is close to $\Psi(u)$. Then
     the probability of accepting a move towards $u + \delta u$ is not guaranteed to 
     decrease uniformly. 

     This  suggests that Assumption~\ref{assumptions-on-Phi}(c) is too restrictive.  But we
     claim that a slight modification of the prior $\mu$ or the likelihood $\Psi$
     can remedy this problem in many applications including deconvolution.
     For a choice of $\tilde b$ and $\tilde \beta$ let $c = \tilde b ( 1- \tilde \beta)$.
    Pick  $R_0 >0$ and 
  define the perturbed likelihood potential
  \begin{equation*}
    \Psi_\veps(u;y) := \Psi(u;y) + \max\{0, \veps\| u\|^2 - R^2_0\}, 
  \end{equation*}
  where $\veps>0$ is a fixed constant satisfying
  \begin{equation*}
    \veps > \frac{ 2c^2}{1 - c^2} \| \mcl G \|^2,
  \end{equation*}
with $\| \mcl G \|$ 
denoting the operator norm of $\mcl G$. Then for any $u \in B_{R_0}(0)^c$
and $v \in B_{c\| u\|}(0)$ we have 
 \begin{align*}
   2\sigma^2(  \Psi_\veps(u;y) - \Psi_\veps(v;y))
   & =  \| \mcl G(u) - y \|_2^2 + \veps \| u\|^2 - \| \mcl G(v) - y \|_2^2 - \veps \| v\|^2\\
   & \ge \veps \| u\|^2 - (2 \| \mcl G\|^2  + \veps ) \| v\|^2 - 2\| y\|_2^2 \\  
   & \ge    \veps \| u\|^2 - c^2 (2 \| \mcl G\|^2  + \veps ) \| u\|^2 - 2\| y\|_2^2 \\
   & =    \Big( \veps  - c^2 (2 \| \mcl G\|^2  + \veps ) \Big) \| u\|^2 - 2\| y\|_2^2. 
 \end{align*}
 The above lower bound is a second order polynomial of $\|u\|$ with a positive leading coefficient
 -- due to the lower bound on $\veps$ -- and so $\Psi_\veps$
 satisfies Assumption~\ref{assumptions-on-Phi}(c) for any choice of $\tilde b, \tilde \beta \in (0,1)$
 and $R_0>0$. 

 It can be verified that this modification of $\Psi$ will result in  a perturbation
 to the posterior $\nu$ that is controlled by the parameter $\veps$, the radius $R_0$, and
 the tails of  $\mu$. 
  Define the perturbed posterior 
 \begin{equation*}
     \frac{\dd \nu_\veps}{\dd \mu}(u) = \frac{1}{Z_\veps(y)} \exp( - \Psi_\veps(u)).
   \end{equation*}
   Using 
   direct computations akin to the proof of \cite[Thm.~5.2]{hosseini-sparse} we can then
   show   $ \exists C  >0$ such that
\begin{equation*}
  d_{TV}( \nu_\veps, \nu) \le C \int_{\{ \|u\|_{H^1} \ge R_0\}} \left( \veps
    \|u\|_{H^1}^2 - R_0 \right) \mu(\dd u), 
\end{equation*}
where $d_{TV}$ denotes the usual TV metric on $\mcl P( H^1(\mbb T))$.
In other words, so long as  $\mu$ has bounded moments of degree at least two
 the TV distance
between  $\nu_\veps$ and $\nu$
can be made arbitrarily small by choosing a large $R_0$. 

This perturbation of the likelihood $\Psi_\veps$
can also  be viewed as a modification of the
prior $\mu$, which results in including the term $\min\{0, \veps \| u\|_{H^1}^2 - R^2_0\}$ in
the MH acceptance ratio. In other words, because we use a proposal
kernel that preserves the original prior, an additional factor (not involving the likelihood 
potential $\Psi$) shows up in the MH acceptance probability.
Regardless of the interpretation, this example illustrates that 
while Assumption~\ref{assumptions-on-Phi}(c) may be difficult to verify in some examples,
often holds for a small perturbation of the problem. Since the term
$\min\{0,\veps\|u\|^2 - R_0^2\}$
is zero near the origin, the dynamics of the Markov chain are entirely unchanged
in a ball around the origin. Since
we can take $R_0$ as large as we like, in practice, this means that the RCAR algorithm corresponding to this
perturbation is virtually identical to 
the original algorithm and the modification is needed only to
control tail behavior necessary to prove exponential
rates of convergence.
These observations may also be taken as a sign that Assumption~\ref{assumptions-on-Phi}(c) is 
an artifact of our method of proof and can be relaxed to a more realistic assumption. 
This would be an interesting direction for future research.

\section{Conclusion}\label{sec:conclusion}

In this article we analyzed the convergence properties of a class of MH algorithms
on infinite-dimensional Banach spaces
that use an RCAR type proposal kernel $\mcl Q$ with a likelihood ratio acceptance probability.
We showed that under very general conditions on the likelihood potential $\Psi$
and the proposal kernel $\mcl Q$
the algorithms have a spectral gap with respect to an appropriate
Wasserstein-type semimetric $\tilde d_q$ which implied exponential convergence to
the target measure $\nu$ in \eqref{nu-definition}. Our results generalize
the dimension-independent spectral gaps of \cite{hairer2014spectral} to a 
larger class of algorithms applicable to non-Gaussian prior measures.

Results showing spectral gaps in infinite dimensions are of particular interest 
in studying the computational complexity of MCMC. Often, 
a spectral gap on the infinite-dimensional space ensures that the variance of
time-averaging estimators for finite-dimensional -- and therefore computationally tractable --
approximations of the Markov kernel is uniformly bounded as a function of dimension. 
Thus the computational complexity of the algorithm is simply
a function of its per-step simulation cost.
The results given here and those of \cite{hairer2014spectral} thus imply that RCAR algorithms
are among the simplest MCMC algorithms whose computational 
complexity depends on dimension only through the per-step computational cost. This is of 
course a special feature of the Ornstein-Uhlenbeck-like proposal, as the random walk
MH algorithm is known to have dimension-dependent spectral gap. It remains 
to be seen whether more sophisticated algorithms can also be designed to have similarly
attractive dimensional scaling properties.

We further developed a general perturbation theory for approximations of MH
algorithms; showing error bounds for computationally tractable approximations
of the algorithm that is arguably more direct than previous works while offering
similar error estimates.
Our main result here was that given an exact MH kernel $\P_0$, an
approximation $\P_\veps$, and an appropriate semimetric $d$,
the distance between the invariant measures of $\P_0$ and $\P_\veps$
can be bounded in terms of the one-step error $d(\P_0 \delta_u, \P_\veps \delta_u)$ --
an error bound that can often  be shown using coupling arguments.
We further applied our
perturbation theory to 
the RCAR algorithm and obtained error bounds for various approximations
including discretization of the likelihood potential $\Psi$ by Galerkin projections as well
as approximation of the prior $\mu$.

Our success in applying the weak Harris' theorem  and perturbation
theory to this large collection of Markov chains suggests the broad utility of this approach to 
studying Markov chains on infinite-dimensional state spaces and computationally
tractable approximations
thereof. 
The tendency of probability measures on infinite-dimensional spaces to be mutually singular
limits the utility of traditional weighted TV norms in these settings. This suggests at the usefulness
of alternative metrics such as the Wasserstein-type semimetrics employed here for studying sequences 
of problems of increasing dimension, which describes many applications of
interest in modern statistics and 
stochastic dynamics.

\section*{Acknowledgement}
Authors are thankful to Prof. Andrew Stuart for interesting conversations regarding
convergence properties of MCMC algorithms, and to Prof. Jonathan Mattingly for pointing out
an error in an early draft of the manuscript. The authors are also grateful to the anonymous reviewers whose 
comments and suggestions helped us improve the article immensely.
\bibliographystyle{abbrv}
\bibliography{ARSDSpecGap_ref}

\appendix

\section{{Proof of convergence results from Section~3}}
\label{sec:proofs-not-included}

\subsection{Proof of Theorem~\ref{thm:lyapunov-function-for-RCAR}}
\label{sec:proof-prop-Lyapunov-fun}

We recall two technical lemmata that are useful in the proof of Theorem~\ref{thm:lyapunov-function-for-RCAR}
as well as the rest of the appendix.

\begin{lemma}\label{large-balls-have-positive-measure}
  For every $\varrho \in P(\mcl{H})$, there exists sufficiently large $R>0$ so that
  $\varrho(B_R(0)) >0$. 
\end{lemma}

\begin{lemma}\label{lp-generalized-triangle-inequality}
  Let $w, v \in \mcl{H}$. Then for $s \ge 0$
  \begin{equation*}
    \| w + v \|^s \le 2^s \left( \| w\|^s + \| v\|^s  \right).
  \end{equation*}
\end{lemma}
\begin{proof}
  When $s \in (0, 1)$ the inequality follows from the identity $(a+ b)^s \le a^s + b^s$
  for positive real numbers $a$ and $b$. The case with $s >1$ follows from
  \cite[Cor.~3.1]{takahasi2010refined}. In fact, the constant $2^s$ is not optimal
  and can be replaced by $1 \vee 2^{s-1}$,
  but it makes for  convenient notation.
\end{proof}

Let us outline a roadmap of the proof that follows the proof  strategy of  \cite{hairer2014spectral}.
Two cases are considered: $u \in B_R(0)$
and the complement of this event. The first case is dispensed with using moment
conditions on $\lambda$ and $\K$ to bound
$\sup_{u \in B_R(0)} (\mcl P V)(u)$.
The 
second case is more difficult. Here we pick an event $A$ such that $\mbb P(A) >0$ uniformly
for all $u \in B_R(0)^c$ and prove the existence of a uniform constant $\tilde \kappa \in (0,1)$ so that
\begin{equation*}
  (\mcl P V)(u) \le \tilde \kappa \mbb P(A) V(u) + \int_{A^c} \left\{ V(u) \vee V(\zeta_u + \xi) \right\}
  \K(u, \dd \zeta_u) \lambda(\dd \xi),
\end{equation*}
and show that the integral term on the right hand side is uniformly bounded as well.
In \cite{hairer2014spectral}, conditional on $u \in B_R(0)^c$, the event
$A = \{ \xi \in B_r(\beta u) : \xi \sim \mu_\beta\}$ is considered. Because
the potential is (eventually) increasing in the tails, the probability of accepting conditional on $A$ can be uniformly bounded
away from $0$. Further, because the pCN proposal is centered at $\beta u$ for
some constant $\beta \in (0,1)$, this event always has positive probability when $\lambda$ 
is Gaussian. This, combined with control of the moments of $\lambda$ and the fact that
when rejection occurs, $V$ does not increase,
is enough to prove that $V$ contracts far from the origin for pCN.
Our proof uses the event
$$A = \{ \|\zeta_u - \beta_0 u \| < b_0 (1-\beta_0) \|u\| \cap \| \xi \| < b_1 \|\zeta\|:
\: \xi \sim \lambda, \zeta_u \sim \K(u, \dd \zeta_u)
\}.$$
The key
difference is that we must now consider the behavior of  $\xi$ and $\zeta_u$ together,
and control  them simultaneously
to ensure that the acceptance probabilities conditional on $A$ can be uniformly bounded
from below when $u$ is far from the origin. 
This introduces complications in the second part of the argument.

\begin{proof}[Proof of Theorem~\ref{thm:lyapunov-function-for-RCAR}]
Fix $R> 0$  and  $u \in B_R(0)$ then, using Assumption~\ref{assumptions-on-K-beta}(a) and
Lemma~\ref{lp-generalized-triangle-inequality} we have
  \begin{align*}
    \sup_{u \in B_R(0)} (\P V)(u)
    &= \sup_{u \in B_R(0)} \int_{\mc H} V(v) \alpha(u,v) \mcl Q(u, \dd v)\\
    & \qquad     + V(u) \left( 1- \int_{\mc H} \alpha(u,v) \mcl Q(u, \dd v) \right) \\
    &\le \sup_{u \in B_R(0)} V(u) + \int_{\mc H} V(v) \mcl  Q(u, \dd v) \\
    &\le R^p + \sup_{u \in B_R(0)} \int_{\mc H \times \mc H} \|\zeta_u + \xi\|^p \K(u,\dd \zeta_u)
     \lambda(\dd \xi) \\
    &\le R^p + 2^p \sup_{u \in B_R(0)} \int_{\mc H \times \mc H} (\|\zeta_u\|^p + \|\xi\|^p)
      \K(u,\dd \zeta_u) \lambda (\dd \xi) \\
    &\le R^p + 2^p  \sup_{u \in B_R(0)}  \|u\|^p + \int_{\mcl H} \|\xi\|^p \lambda(\dd \xi) \\
    &\le R^p + 2^p  R^p + C_0 \equiv K_1. 
  \end{align*}
  Now consider $u \in B_R(0)^c$, $1> \beta_0 > b_0 > 0$, as in Assumption~\ref{assumptions-on-K-beta}(b)
  and 
  define $A$ to be  the  event
   \begin{equation*}
    A : = \left\{ \| \zeta_u - \beta_0 u \| \le b_0 (1 - \beta_0)  \|u\|,
    \| \xi \| < b_1  \| \zeta_u \| \: : \zeta_u \sim \K(u, \dd \zeta_u), \xi \sim \lambda\right\},
\end{equation*}
where $b_1 \in (0,1)$ is a constant to be specified.
Observe that
in the event of $A$ we have that
$\|\zeta_u \| \ge ( \beta_0 - b_0(1- \beta_0)) \| u\| >0$.
  Thus, by the independence of $\zeta_u$ and $\xi$
  we have
  \begin{equation*}
    \begin{split}
      \mbb P[A |u ] & \ge  \mbb P \left[  \| \zeta_u - \beta_0 u \| \le b_0(1- \beta_0)\| u\| \right]
      \mbb P\left[ \| \xi \| < b_1 ( \beta_0 - b_0(1- \beta_0))  R\right].
  \end{split}
\end{equation*}
By Assumption~\ref{assumptions-on-K-beta}(b)  we have
\begin{equation*}
  \mbb P \left[  \| \zeta_u - \beta_0 u \| \le b_0(1 - \beta_0) \| u\| \right] \ge \epsilon_0.
\end{equation*}
On the other hand, 
for fixed  $1>\beta_0>b_0>0$ and $b_1$   it
follows from Lemma~\ref{large-balls-have-positive-measure} that
  if $R$ is sufficiently large then $\mbb P[\| \xi \|
  <  b_1 (\beta_0 - b_0(1- \beta_0)) R] \ge \epsilon_1 >0$.
To this end,
  $\mbb P[A |u ] \ge \epsilon_0 \epsilon_1 >0$. 
  Furthermore, we have that in the event of $A$
  \begin{equation*}
    \begin{split}
    \| \zeta_u + \xi \|^p & \le
    \left[ \big( 1 + b_1 (\beta_0 - b_0( 1-\beta_0) \big)   \| \zeta_u \| \right]^p \\ &
    \le \left[ \Big( 1 + b_1 (\beta_0 - b_0( 1-\beta_0)\Big) \big(\beta_0 + b_0(1- \beta_0) \big)
    \right]^p  \| u\|^p \\
    & \le \kappa_1 \| u\|^p.
  \end{split}
  \end{equation*}
  Now if $b_1$ is sufficiently small then $\kappa_1 <1$.
  In summary  given  $b_0, \beta_0 \in (0,1)$, which depend on the kernel $\K$, 
   we
   choose $b_1$ so that $\kappa_1 <1$ and then we choose
   $R$ large enough so that $\epsilon_0 \epsilon_1 >0$.
  It then follows that in the event of $A$ we have
  $ V(\zeta_u + \xi ) \le \kappa_1 V(u)$.
  Now we have
  \begin{align*}
    (\mcl P V)(u)
    &\le \mbb P(A) [ \mbb P(\text{accept }| A) \kappa_1 V(u) +
      \mbb P(\text{reject } |A) V(u)]  \\
    & \quad + \int_{A^c} \left\{ V(\zeta_u + \xi) \vee V(u) \right\} \mcl K(u, \dd \zeta_u)
      \dd \lambda(\xi) \\
    & = \mbb P(A) [ (1 - (1 - \kappa_1)) \mbb P(\text{accept } |A )] V(u) \\
    &\quad +  \int_{A^c} \left\{ V(\zeta_u + \xi) \vee V(u) \right\} \K(u, \dd \zeta_u)
      \dd \lambda(\xi) \\
    & \le \kappa_2 \mbb{P}(A) V(u) +  \int_{A^c} \left\{ V(\zeta_u + \xi) \vee V(u) \right\}
      \K(u, \dd \zeta_u) \dd \lambda(\xi),
  \end{align*}
  where $\kappa_2 = (1 - (1 - \kappa_1)) \mbb P(\text{accept}|A)$. Since $\Psi$ satisfies
  Assumption~\ref{assumptions-on-Phi}(c), given
  $\tilde{\beta} \in (0,1)$ we can take $R > R_0$ which implies $\mbb P(\text{accept }|A) >0$
  uniformly for all $u \in B_R(0)^c$ and so it follows that $\kappa_2 <1$ uniformly over $B_R(0)^c$.
  It remains to bound the last integral:
  \begin{equation*}
  \begin{aligned}
    \int_{A^c}  \left\{ V(\zeta_u + \xi) \vee V(u) \right\}
    & \K(u, \dd \zeta_u) \dd \lambda(\xi) \\
    &= \int_{A^c} \|\zeta_u +\xi\|^p \vee \|u\|^p   \K(u, \dd \zeta_u) \lambda(\dd \xi) \\
    &\le \int_{A^c} (\|\zeta_u\| +\|\xi\|)^p \vee \|u\|^p   \K(u, \dd \zeta_u) \lambda(\dd \xi) \\
    &\le \int_{A^c} (\|u\| +\|\xi\|)^p   \K(u, \dd \zeta_u) \lambda(\dd \xi) \\
    &= \int_{A^c}  \sum_{k=0}^p {p \choose k} \|u\|^{p-k} \|\xi\|^k \lambda(\dd \xi) \\
    &= \|u\|^p \bb P(A^c) + \sum_{k=1}^p {p \choose k} \|u\|^{p-k} \int_{A^c} \|\xi\|^k  \lambda(\dd \xi) \\
    &\le \|u\|^p \bb P(A^c) + \frac{(1-\kappa_2)\epsilon_0 \epsilon_1}{2} \|u\|^p + K_2,
  \end{aligned}
  \end{equation*}
  where we used Assumption~\ref{assumptions-on-K-beta}(a) to bound $\|\zeta_u\|$ by $\| u \|$, and
  the last step followed because the second term in the
penultimate line is a polynomial in $\|u\|$ of order $p-1$. 
Since $R>1$, this term can be bounded by $c \|u\|^p + K_2$ for any $c>0$, 
where $K_2$ depends on $c$ but not $u$.  
  Substituting the above result back into the bound on $(\mcl P V)(u)$ gives
  \begin{align*}
    (\mcl P V)(u) & \le \left( \kappa_2 \mbb P(A) + \mbb P(A^c) + \frac{(1-\kappa_2)\epsilon_0 \epsilon_1}2 \right) V(u) + K_2 \\
                  & \le \left[ 1 - ( 1 - \kappa_2) \epsilon_0 \epsilon_1 + \frac{(1-\kappa_2)\epsilon_0 \epsilon_1}2 \right] V(u) + K_2 \\
                  & \le \kappa V(u) + K_2,
  \end{align*}
  for  $\kappa = 1- \frac{(1-\kappa_2)\epsilon_0 \epsilon_1}2  \in (0,1)$, which does not depend on $u$. Setting $K = K_1 + K_2$ we
  obtain the desired result
  \begin{equation*}
    (\mcl P V)(u) \le \kappa V(u) + K, \qquad \forall u \in \mcl H.
  \end{equation*}
\end{proof}

\subsection{Proof of Theorem~\ref{thm:spectral-gap-locally-Lip}}
\label{sec:proof-P-is-d-contracting}

The proof of this theorem follows from Theorem~\ref{thm:lyapunov-function-for-RCAR} and 
Propositions~\ref{thm:weak-harris},\ref{P-is-d-contracting-locally-Lipschitz}
and \ref{d-q-small-sets-locally-Lipschitz}, which together establish that the
$n$-step kernel $\mcl{P}^n$ is contracting for $d_q$ and 
the level sets of the Lyapunov functions $V(u) = \sum_{j=0}^p a_j\| u \|^j$ are $d_q$-small.

First, let us define the notation
\begin{equation}\label{def:d-ast}
  d^\ast(u,v):= \frac{ (1 + \eta\| u\| + \eta\| v\|)^q \| u - v \|}{\omega},
  \qquad u,v \in \H,
\end{equation}
for $q, \eta, \omega >0$. Recall that by \eqref{def:distance-like-d-q} we simply have
$d_q(u,v) = 1 \wedge d^\ast(u,v)$.  We then have the following auxiliary lemma
concerning $d^\ast$ and $d_q$.
\begin{lemma}\label{tdq-properties}
  Let $q \ge 0$.
Then $d_q$ and $d^\ast$ satisfy the following properties:
\begin{enumerate}[label=(\alph*)]
\item If $\eta,  d_q(u,v) < 1$ then 
$$
\eta^q( 1 + \| u\| +\| v\|)^q \| u -v \| < \omega \quad
\text{ and } \quad
\frac{\eta^q}{\omega}\| u -v \| \le d_q(u,v).
$$
\item Let $u,v, \zeta_u, \zeta_v \in \mcl H$ such that
  \begin{equation*}
  d_q(u,v) < 1 \quad  \text{and} \quad
  \| \zeta_u \| \le \| u\| \quad  \text{and}  \quad\|\zeta_v \| \le \| v\|.
\end{equation*}
  Define proposals $u^\ast, v^\ast$ as in \eqref{coupled-proposals}. Then
$$
\begin{aligned}
\frac{{d}^\ast(u^\ast, v^\ast)}{{d}^\ast(u,v)} 
\le (1 + 2 \eta \| \xi\|^q) \frac{ \| \zeta_u - \zeta_v\|}{\| u - v \|}.
\end{aligned}
$$
\end{enumerate}
\end{lemma}
\begin{proof}
  
  Statement (a) follows from the fact that $d_q(u,v) = d^\ast(u,v)$ whenever $d_q(u,v) <1$.
  Then assuming $\eta <1$ we have the series of inequalities
  \begin{equation*}
    \begin{split}
      \frac{\eta^q}{\omega} \| u - v\|
      & \le  \frac{\eta^q(1 + \| u\| + \| v\|)^q}{\omega} \| u - v\|\\
     & \le  \frac{(1 + \eta \| u\| + \eta \| v\|)^q}{\omega} \| u - v\|
      = d_q(u,v) < 1,
    \end{split}
  \end{equation*}
  from which the statements follow.
 Now (b) can be proven directly by the following calculation:
  \begin{equation*}
    \begin{split}
       d^\ast(u^\ast, v^\ast)
      & = \frac{1}{\omega}\big( 1 + \eta \| \zeta_u + \xi \| + \eta \|\zeta_v + \xi\| \big)^q  \| \zeta_u - \zeta_v\| \\
      & \le \frac{1}{\omega} \big( 1 + \eta (\| \zeta_u\| + \|\xi \|) + \eta (\|\zeta_v\| + \|\xi\|) \big)^q  \| \zeta_u - \zeta_v\|\\
      & =  \left( \frac{1 + \eta (\| \zeta_u\| + \|\xi \|) + \eta (\|\zeta_v\| + \|\xi\|)}
        {1+ \eta \| u\| + \eta \| v\|  } \right)^q  \frac{ \| \zeta_u - \zeta_v\|}{\| u - v\|} d^\ast(u,v)\\
      & \le (1 + 2 \eta \|\xi\|)^q \frac{\| \zeta_u - \zeta_v\|}{\| u -v \|} d^\ast(u,v).
    \end{split}
  \end{equation*}
  {}
\end{proof}

We are now ready to show  that $\P$ is $d_q$-contracting for appropriate choices of $\omega,\eta$ in
\eqref{def:distance-like-d-q}. 

\begin{proposition}[Contracting for $d_q$]\label{P-is-d-contracting-locally-Lipschitz}
  Suppose conditions of Theorem~\ref{thm:spectral-gap-locally-Lip} are
  satisfied.
  Then
  $\P$
  is contracting for $d_q$ if $\eta, \omega/\eta^q >0$ are
    sufficiently small.
  \end{proposition}

  Our proof strategy is as follows:
  Since the $d_q$ semimetric for measures is defined as the infimum over couplings, naturally our argument
  relies on showing that 
  there exists a coupling for which the desired contraction property holds
  when the two chains start close to each other. Our approach shares some similar features with
\cite{hairer2014spectral}, as well as with earlier work. 
The proof in \cite[Sec.~3.1.2 and 3.2.2]{hairer2014spectral}
uses a ``basic'' or ``same-noise'' coupling of pCN proposals that 
is well-known in coupling of diffusion processes (see, for example \cite{mattingly2003recent}), along with
utilizing the same uniform random variable
to make the accept-reject decision for the two coupled chains. This coupling has also appeared in
the statistics literature, where it is used for convegence diagnosis \cite[pg 164]{johnson1996studying}.
We use a different coupling in our proof. More precisely, we consider two chains
starting at $(u,v)$ and propose
\begin{equation*}
u^* = \zeta_u + \xi, \quad v^* = \zeta_v + \xi,
\end{equation*}
where $(\zeta_u, \zeta_v) \sim \varpi_{u,v}$, the coupling in Assumption \ref{assumptions-on-K-beta} (c).
and $\xi \sim \lambda$. 
we then  utilize the same uniform random number $\varsigma$ to make the accept/reject
decision; recall Algorithm~\ref{generic-RCAR}. The 
existence of a $\|\cdot\|$-contractive coupling $\varpi_{u,v}$ is necessary at this point as without this
condition the $d_q$-contraction condition can easily  fail. 

With the above coupling at hand 
our proof then proceeds by  considering 
three possible outcomes to show the desired contractility results: either both chains accept, both
reject, or one rejects and one accepts. As is the case in proving these properties
for the MH algorithms such as pCN,
the last case is the hardest, since in principle the two components can land far from one another
in $d_q$. In our setting the argument is somewhat lengthy since we need
to control both $\zeta_u$ and $\zeta_v$ as well as the innovation $\xi$ at the same time.

\begin{proof}
  Pick $u,v \in \mcl H$ so that $d_q(u,v) <1$ implying that
  $(1 + \eta \| u\| + \eta \|v \|)^q \| u -v \| < \omega$
  and fix $\beta \in (0,1)$.
  Let 
  $\varpi_{u,v} \in \Upsilon(\K \delta_u, \K \delta_v)$ 
  be the coupling in Assumption~\ref{assumptions-on-K-beta}(c).
  We then define $\pi_0 \in \Upsilon(\mcl P \delta_u, \mcl P \delta_v)$
  the {\it basic coupling} between the $u$ and $v$ chains by the following procedure.
  Draw $(\zeta_u, \zeta_v) \sim \varpi_{u,v}$, $\xi \sim \lambda$, and
  consider proposals
  \begin{equation}\label{coupled-proposals}
    u^\ast = \zeta_u + \xi, \qquad v^\ast = \zeta_v + \xi.
  \end{equation}
  Then draw $\varsigma \sim \mcl U([0,1])$ and accept $u^\ast$ if $ \varsigma \le \alpha(u, u^\ast)$
  and accept $v^\ast$ if $\varsigma \le \alpha(v, v^\ast)$. That is, the two chains use
  the same innovation $\xi$ and
  uniform random variable $\varsigma$ for the accept-reject step.

  Now pick $R>0$ sufficiently large so that
  $R -1 > R_0$ where $R_0$ is as in Assumption~\ref{assumptions-on-Phi}(c).
  We will present the proof for two cases where
  $u,v \in B_R(0)$ and $u,v \in B_{R-1}(0)^c$.
  Note that 
  if $\omega/\eta^q< 1$  we can guarantee

  \begin{equation}\label{guarantee-to-cover-whole-space-globally-Lip-case}
    \begin{split}
      \{ u,v \in \mcl H & : d_q(u,v) < 1\} \\
      &= \{ u,v \in B_R(0) : d_q(u,v) < 1\} \cup \{ u,v \in B_{R-1}(0)^c : d_q(u,v) < 1\}.
  \end{split}
\end{equation}
To see this take $u,v \in \H$ such that $d_q(u,v) <1$ and consider the nontrivial
  case where
   $u,v$ do not belong to the same set $B_R(0)$. Without loss of generality 
  let $u \in B_R(0)$
  and $v \in B_R(0)^c \subset B_{R-1}(0)^c$.
  By Lemma~\ref{tdq-properties}(a) we have $\| u - v\| < \frac{\omega}{\eta^q} <1$ and so
  $u \in B_1(v)$. But $\{ u : v \in B_R(0)^c \text{ and } u \in B_1(v) \} = B_{R-1}(0)^c$
  and so $u,v \in B_{R-1}(0)^c$. 

  Let us proceed with the proof starting with the case where $ u,v \in B_R(0)$. Let $D$ be the event where $\| \xi \| \le R$
  and $\| \zeta_u - \zeta_v\| < \tilde{\beta} \| u -v \|$ where $\tilde \beta \in [\beta_c, 1)$
  and $\beta_c$ is the constant in Assumption~\ref{assumptions-on-K-beta}(c). Due to
  independence of $(\zeta_u, \zeta_v)$ and $\xi$ we have
  \begin{equation*}
    \bb P (D) = \mbb P (\| \xi \| \le R) \mbb P (\| \zeta_u - \zeta_v \| < \tilde \beta \| u -v \|).
  \end{equation*}
  By Lemma~\ref{large-balls-have-positive-measure} $\mbb P (\| \xi \| \le R) >0$ if $R$ is
  sufficiently large. Furthermore, by Markov's inequality
  \begin{equation*}
    \mbb P (\| \zeta_u - \zeta_v \| \ge \tilde \beta \| u -v \|)
    \le \frac{\mbb E \| \zeta_u - \zeta_v \| }{ \tilde \beta \| u -v \| }
    \le \frac{\beta_c}{\tilde \beta} <1.
  \end{equation*}
  Thus, $\mbb P(D) \ge \epsilon_1 >0$ uniformly for all $u,v \in B_R(0)$.
  Recalling that $d_q \le 1$, we have
  \begin{equation*}
    \begin{split}
      d_q(\mcl P \delta_u, \mcl P \delta_v)
      & \le \int_{\H \times \H} d_q(s,t) \pi_0(\dd s,\dd t)\\ 
      & \le  \int_D \Big[ \mbb P( \text{both accept} |\zeta_u, \zeta_v, \xi)
      d_q(u^\ast,v^\ast)\\
      & \qquad  \qquad + \mbb P( \text{both reject} | \zeta_u, \zeta_v, \xi)
      d_q(u,v) \Big] \varpi_{u,v}(\dd \zeta_u, \dd \zeta_v) \lambda(\dd \xi) \\
    & + \int_{D^c}\left[ d_q(u^\ast, v^\ast) \vee  d_q(u,v) \right] \varpi_{u,v}(\dd \zeta_u, \dd \zeta_v)
    \lambda(\dd \xi)\\
                                           & + \mbb P( \text{only one is accepted})\\
   & =: T_1 + T_2 + T_3.
    \end{split}
  \end{equation*}

\textbf{Bound on $T_1$.}  Since $\mbb P(\text{both accept}| \zeta_u, \zeta_v, \xi)
  \le 1 - \mbb P(\text{both reject} | \zeta_u, \zeta_v, \xi)$
  then
  \begin{align*}
    T_1
    &\le   \int_D\Big[ \mbb P( \text{both accept} |\zeta_u, \zeta_v, \xi)
      d_q(u^\ast,v^\ast)\\
    & \qquad  \qquad
      + [1 - \mbb P( \text{both accept} | \zeta_u, \zeta_v, \xi)] d_q(u,v) \Big]
      \varpi_{u,v}(\dd \zeta_u, \dd \zeta_v) \lambda(\dd \xi)\\
  & = \int_D   \mbb P( \text{both accept} |\zeta_u, \zeta_v, \xi)
    \left[ d_q(u^\ast,v^\ast) - d_q(u,v) \right] \varpi_{u,v}(\dd \zeta_u, \dd \zeta_v) \lambda(\dd \xi)\\
    & \qquad  \qquad +  \mbb P(D) d_q(u,v).  
  \end{align*}
  By the definition of the set $D$ and Lemma~\ref{tdq-properties}(b) we can write
  \begin{equation*}
    \begin{split}
      T_1 &- \mbb P(D) d_q(u,v) \le \\
   & \int_D 
    \mbb P( \text{both accept} |\zeta_u, \zeta_v, \xi)
    \left[ (1 + 2\eta R)^q\tilde \beta - 1 \right] d_q(u,v) \varpi_{u,v}(\dd \zeta_u, \dd \zeta_v)
    \lambda(\dd \xi).
  \end{split}
\end{equation*}

By Assumption~\ref{assumptions-on-Phi}(a) and (b) 
$\mbb P( \text{both accept} |\zeta_u, \zeta_v, \xi) \ge \epsilon_2 \ge 0$
uniformly for all $u,v \in B_R(0)$ and so
\begin{equation*}
  \begin{split}
    T_1 & \le \mbb P(D) \big[1 + \epsilon_2 \big(( 1+ 2 \eta R)^q \tilde \beta -1\big)\big] d_q(u,v), \\
    & = \mbb P(D) (1 - \kappa).
  \end{split}
 \end{equation*}

\textbf{Bound on $T_2$}.  Using Lemma~\ref{tdq-properties}(b) and Assumption~\ref{assumptions-on-K-beta}(c)
we can write
 \begin{align*}
  T_2 & = \mbb P(D^c) \mbb E \left( d_q(u^\ast, v^\ast) \vee d_q(u,v)  | D^c \right)\\
       & \le d_q(u,v) \mbb E \left( 1 \vee \frac{d_q(u^\ast, v^\ast)}{d_q(u,v)}
    \Bigg| D^c \right) \\
        & \le d_q(u,v) \mbb E \left( 1 \vee (1 + 2 \eta  \| \xi\| )^q
          \frac{\| \zeta_u - \zeta_v\|}{\|u -v \|}  \Bigg| D^c \right)\\
        & \le d_q(u,v) \mbb E \left( (1  + 2 \eta \| \xi \|)^{\lceil q \rceil} | D^c \right) \\
        & \le d_q(u,v) \mbb E \left( 1 + C (2\eta \| \xi \|)^{\lceil q \rceil} | D^c \right) \\
        & \le d_q(u,v) (\mbb P(A^c) + \eta^{\lceil q \rceil} C_R),
 \end{align*}
where $C> 0$ depends on $\lceil q \rceil$ and is bounded following the binomial expansion
formula. Then $C_R >0$ is a bounded constant  due to the fact that $\lambda$ has
bounded moments of degree $p \ge \lceil q \rceil $.

\textbf{Bound on $T_3$}. Using Lemma~\ref{lp-generalized-triangle-inequality},
Assumption~\ref{assumptions-on-K-beta}(c), and Lemma~\ref{tdq-properties}(a)
we have
\begin{align*}
  T_3 &= \int_{\mcl H} \int_{\mcl H \times \mcl H}   \mbb P[ \text{only one is accepted} | \zeta_u, \zeta_v, \xi] \varpi_{u,v}(\dd \zeta_u, \dd \zeta_v) \lambda(\dd \xi) \\
      &  = \int_{\mcl H} \int_{\mcl H \times \mcl H}   \mbb P[\varsigma
        \text{ between } \alpha(u, u^\ast) \text{ and } \alpha(v, v^\ast)  | \zeta_u, \zeta_v, \xi]\\
  & \hspace{.65\textwidth} \varpi_{u,v}(\dd \zeta_u, \dd \zeta_v) \lambda(\dd \xi) \\
      & \le \int_{\mcl H} \int_{\mcl H \times \mcl H}  | \Psi(u) - \Psi(v)|
        + | \Psi(\zeta_u + \xi) - \Psi(\zeta_v + \xi) | \varpi_{u,v}(\dd \zeta_u, \dd \zeta_v)
         \lambda( \dd \xi) \\
      & \le L ( 1 \vee \| u\|^q \vee \| v\|^q) \| u - v\|  \\
      &\qquad + L \int_{\mcl H} \int_{\mcl H \times \mcl H} (1 \vee \| \zeta_u + \xi\|^q
        \vee \| \zeta_v + \xi \|^q ) \| \zeta_u- \zeta_v\|
        \varpi_{u,v}(\dd \zeta_u, \dd  \zeta_v) \lambda(\dd  \xi) \\
  &\le L  R^q \| u - v\| + \\
     &\qquad  L \int_{\mcl H} \int_{\mcl H \times \mcl H}\left[1 \vee 2^q( \| \zeta_u\|^q + \|\xi\|^q)
       \vee  2^q(\| \zeta_v\|^q + \|\xi \|^q) \right] \| \zeta_u- \zeta_v\| \\
      & \hspace{.65\textwidth}    \varpi_{u,v}(\dd \zeta_u, \dd \zeta_v)  \lambda(\dd \xi) \\
 & \le  L  R^q \| u - v\| +  \\
     & \qquad  L \int_{\mcl H} \int_{\mcl H \times \mcl H} \left(1 + 2^q R^q + 2^q \| \xi \|^q \right) \| \zeta_u- \zeta_v\| \varpi_{u,v}(\dd \zeta_u, \dd \zeta_v) \lambda( \dd \xi)\\ 
      & \le L  R^q \| u - v\| + 
        L \int_{\mcl H}  \left(1 + 2^q R^q + 2^q \| \xi \|^q \right) \beta_c\| u- v\| \lambda(\dd \xi)\\
        & \le L  R^q \| u - v\| + 
          L \beta_c \| u - v\|  \int_{\mcl H}  1 + 2^q R^q + 2^q \| \xi \|^q   \lambda(\dd \xi) \\
  & \le L \frac{\omega}{\eta^q} C_R' d_q(u,v).
\end{align*}
Here $C_R' >0$ is a uniform constant independent of $u,v$ which is bounded 
since $\lambda$ has bounded
moments of degree $p \ge \lceil q \rceil$. Putting together the bounds for
$T_1, T_2$ and $T_3$ we finally have
\begin{align*}
  &d_q(\mcl P \delta_u, \mcl P \delta_v ) \le \\
  &\qquad \left[ \mbb P(D)
  \big[1 + \epsilon_2 \big(( 1+ 2 \eta R)^q \tilde \beta -1\big)\big]+ \mbb P(D^c) + \eta^{\lceil q \rceil} C_R +\eta^{-q} \omega L  C_R' \right] d_q(u,v).
\end{align*}
Since $\epsilon_2 \in (0,1)$, $\tilde \beta <1$, and $R >0$  are uniform constants,
we can choose $\eta$
sufficiently small so that $1 + \epsilon_2 \big(( 1+ 2 \eta R)^q \tilde \beta -1\big) < 1$.
The constants $C_R,C_R'>0$ are also  uniform and so
we can choose $\eta$ and $\omega/\eta^q$
sufficiently small so that the term inside the square brackets is
less than one which gives the desired result

\begin{align*}
  d_q(\mcl P \delta_u, \mcl P \delta_v) \le {\gamma}_1 d_q(u,v),
\end{align*}
for some ${\gamma}_1 \in (0,1)$.

Let us now consider the case where $u,v \in B_{R-1}(0)^c$. The method of proof is
very similar to the first case where $u,v \in B_R(0)$ and so we only highlight the
differences.
Let $\tilde D$ be the event where $\| \xi \| \le R-1$ and $\| \zeta_u - \zeta_v\| \le \tilde \beta
\| u -v \|$ where as before $\tilde \beta \in [\beta_c ,1)$. The same argument
as before yields that $\mbb P(\tilde D) \ge \epsilon_3 >0$. Furthermore, using the
same argument as before we can write
\begin{align*}
  d_q(\mcl P \delta_u, \mcl P \delta_v)
  & \le  \int_{\tilde D} \Big[ \mbb P( \text{both accept} |\zeta_u, \zeta_v, \xi)
    d_q(u^\ast,v^\ast)\\
  & \qquad  \quad + \mbb P( \text{both reject} | \zeta_u, \zeta_v, \xi) d_q(u,v) \Big]
    \varpi_{u,v}(\dd \zeta_u, \dd \zeta_v) \lambda(\dd \xi) \\
  & + \int_{\tilde{D}^c}\left[ d_q(u^\ast, v^\ast) \vee  d_q(u,v) \right]
    \varpi_{u,v}(\dd \zeta_u, \dd \zeta_v)
    \lambda(\dd \xi)\\
                                           & + \mbb P( \text{only one is accepted})\\
   & =: T_1' + T_2' + T_3'.
\end{align*}
By the same argument used to bound $T_1$ we have
\begin{equation*}
   T_1' \le \mbb P(\tilde D) \big[1 + \epsilon_4 \big(( 1+ 2 \eta (R-1))^q \tilde \beta -1\big)\big] d_q(u,v),
 \end{equation*}
 where here $\epsilon_4$ is a constant  so that
 $\mbb P( \text{both accept} | \zeta_u, \zeta_v, \xi) \ge \epsilon_4$ uniformly over $\tilde D$. 
 By Assumption~\ref{assumptions-on-Phi}(c) $\epsilon_4 >0$ uniformly for all $u,v \in B_{R-1}(0)^c$.
Furthermore, we  bound $T_2'$   identically to $T_2$,
 \begin{equation*}
   T_2' \le d_q(u,v) ( \mbb P(\tilde{D}^c) + \eta^{{\lceil q \rceil}} C_{R-1}),
 \end{equation*}
 using Lemma~\ref{tdq-properties}(b) and Assumption~\ref{assumptions-on-K-beta}(b);
 and bound $T'_3$ identically to $T_3$,
 \begin{equation*}
   T_3' \le  \frac{L \omega}{\eta^{q}} C_{R-1}' d_q(u,v),
 \end{equation*}
 using Lemmata~\ref{lp-generalized-triangle-inequality} and \ref{tdq-properties}(a)
 as well as Assumption~\ref{assumptions-on-K-beta}(c).
 In the above bounds $C_{R-1}, C_{R-1}'> 0$ are uniform constants since $\lambda$ has
 bounded moments of degree $p \ge \lceil q \rceil$.
Thus, we have the bound
  \begin{align*}
  d_q(\mcl P \delta_u, \mcl P \delta_v)
  &\le \Big[ \mbb P(\tilde D) \big[1 + \epsilon_4 \big(( 1+ 2 \eta (R-1))^q \tilde \beta -1\big)\big]
  \\ &\qquad  + \mbb P(\tilde{D}^c) + \eta^{\lceil q \rceil} C_{R-1}  + \eta^{-q} \omega L C_{R-1}') \Big] d_q(u,v).
  \end{align*}
Once again choosing $\eta$ and $\omega/\eta^q$ sufficiently small we obtain
\begin{align*}
  d_q(\mcl P \delta_u, \mcl P \delta_v ) \le \gamma_2 d_q(u,v),
\end{align*}
for some constant $\gamma_2 \in (0,1)$. Combining our results for
the two cases of $u,v \in B_R(0)$ and $u,v \in B_{R-1}(0)^c$  we have
the desired bound
\begin{align*}
  d_q( \mcl P \delta_u, \mcl P \delta_v) \le (\gamma_1 \vee \gamma_2)
  d_q(u,v).
\end{align*}

\end{proof}


\begin{proposition}[$d_q$-small $V$ level-sets]\label{d-q-small-sets-locally-Lipschitz} 
  Suppose the conditions of Theorem~\ref{thm:spectral-gap-locally-Lip} are satisfied
  and let 
$S(R) = \{ u \: | \: V(u) \le R \}$ for some $R > 0$.  Then 
there exists an integer $n \ge 1$ and a constant $\tilde{\gamma}_2 \in (0,1)$ so that 
\begin{equation*}
d_q(\mcl{P}^n\delta_u , \mcl{P}^n\delta_v ) \le \tilde{\gamma}_2 \qquad \forall u,v \in S(R).
\end{equation*}
\end{proposition}

Following a similar approach to \cite{hairer2014spectral}, we prove this proposition
using the coupling introduced in the proof of Proposition~\ref{P-is-d-contracting-locally-Lipschitz}
and conditioning on the event that the coupled proposals
are accepted $n$ times
in a row. The probability of this event is uniformly bounded away
from zero on sublevel sets of $V$ following Assumption~\ref{assumptions-on-Phi}(a,b), which is
critical in making the argument. Using the fact that the sublevel sets of $V$ have finite diameter
we then show that if $n$ is sufficiently large then eventually the coupled chains
draw within $d_q$-distance one.

\begin{proof}

  Fix $R>a_0$ and let $R_\ast >0$ be the solution of the
  equation $\sum_{j=0}^pa_j R_\ast^j - R = 0$; note that $R_\ast$ is unique  so long as
  $a_j \ge 0$ and there is at least one coefficient $a_j >0$ as
  it is the root of a monotone polynomial on $(0, +\infty)$. 
  Let $\pi_0$ be the basic coupling used in the proof of Proposition~\ref{P-is-d-contracting-locally-Lipschitz}. 
We use $(u_k, v_k)$ to denote the chain after step $k$ with initial points 
$u_0 = u, v_0 = v \in S$ and denote the innovation at each step
with $\xi_k$. Fix $\tilde \beta \in [\beta_c, 1)$ and consider the
events $D_k$ for $k =1, \dots, n$ where
$\| \zeta_{u_{k-1}} - \zeta_{v_{k-1}} \| < \tilde \beta \| u_{k-1} -v_{k-1} \|$ and
$\| \xi_k \| < r/n$ for some constant $r >0$ to be specified. The events $D_k$
are similar to the event $D$ from the proof of Proposition~\ref{P-is-d-contracting-locally-Lipschitz}.

 Let $E$ 
 be the event that the  the proposals $u^\ast_k = \zeta_{u_{k-1}^\ast} + \xi_k$
 and $v^\ast_k = \zeta_{v_{k-1}^\ast} + \xi_k$ are accepted $n$ times in a row conditioned
 on the intersection of the events $D_k$. Thus, conditional on $E$ we have
\begin{equation}\label{n-step-proof-d-le-diam-S}
  \begin{split}
    d_q( u_n, v_n)
    &  \le \frac{(1 + \eta\| u_n\| +\eta \|v_n\|)^q}{\omega} \| u_n -v_n\| \\
    & \le  \frac{ \tilde{\beta}^n (1 + \eta\| u\|
      +\eta \|v\| + 2 \eta r )^q}{\omega} \| u -v\|\\
    & 
    \le
     (1 + 2 \eta R_\ast  + 2 \eta r )^q
     \frac{\tilde{\beta}^n}{\omega} \text{diam} S,
\end{split}
\end{equation}
where we used $\text{diam} S := \sup_{u, v \in S} \| u -v \|$ to denote the diameter of $S$.
Choosing
\begin{equation*}
  n = \left\lceil \frac{1}{\log \tilde{\beta}}
   \log\left(  \frac{\omega}{ 2 (1 +2  \eta R_\ast + 2\eta r )^q\text{diam} S} \right) \right \rceil
\end{equation*}
conditional on $E$, we have $d(u_n,v_n) < 1/2$ and so
\begin{equation*}
\sup_{u,v \in S(R)} d_q(\mcl{P}^n\delta_u, \mcl{P}^n\delta_v) \le \mbb{P}(E)\frac{1}{2} + (1 - \mbb{P}(E)) < 1. 
\end{equation*}

It remains to show that $\mbb P(E) > 0$.
By Lemma \ref{large-balls-have-positive-measure}
 we can choose  $r$ large enough
 that $ \mbb P(\| \xi_k \| \le r/n)  >0$
 uniformly for all $k$.  Furthermore, using
 an identical argument as in the  proof of Proposition~\ref{P-is-d-contracting-locally-Lipschitz}
 to show $\mbb P(D) >0$, we can use 
 Assumption~\ref{assumptions-on-K-beta}(b) and
 Markov's inequality to show that $\mbb P( \| \zeta_{u_{k-1}} - \zeta_{v_{k-1}}\| \le \tilde{\beta}
 \| u_{k-1} - v_{k-1}\|) >0$
 uniformly for all $k =1, \dots, n$. This follows because
   all pairs $(u_{k}, v_{k})$ 
   are contained within  $B_{R_\ast + r}(0)$.
Thus there exists $\eps> 0$
so that
\begin{equation*}
\inf_{u_0, v_0 \in S(R)} \inf_{k \in \{1, \dots, n\}} \mbb P(D_k| \cap_{j=1}^{k-1} D_j) \ge \eps >0.
\end{equation*}
 Let $I= \cap_{k=1}^n D_k$.
Then by  the law of total probability 
\begin{equation*}
\inf_{u_0, v_0 \in S(R)}  \mbb P(I) \ge \eps^n >0.
\end{equation*}
On the other hand, by Assumption \ref{assumptions-on-Phi}(a) and (b),
$\Psi$ is bounded above and below on bounded sets and so 
\begin{equation*}
\inf_{u_0,v_0 \in S(R)} \bb P(E \mid I ) > 0.
\end{equation*}
Putting together the above lower bounds we obtain the desired result:
\begin{equation*}
\inf_{u_0, v_0 \in S(R) }\bb P(E) = \inf_{u_0, v_0 \in S(R) } \mbb{P}( I ) \mbb{P}(E \mid I) >0.
\end{equation*}

\end{proof}

\begin{proof}[Proof of Theorem~\ref{thm:spectral-gap-locally-Lip}]
  Propositions~\ref{P-is-d-contracting-locally-Lipschitz} and \ref{d-q-small-sets-locally-Lipschitz}
   show that $d_q$ is contracting for $\P$ and that the sublevel sets of
  the $V$ are $d_q$-small; recall Definitions~\ref{d-contracting-metric} and \ref{W-minorization}.
  Furthermore by Theorem~\ref{thm:lyapunov-function-for-RCAR} and Remark~\ref{rem:polynomial-Lyapunov-func}
  we have that the function $V$ as in \eqref{def:polynomial-Lyapunov-func} is a continuous
  Lyapunov function for $\P$.
  An application of Proposition~\ref{thm:weak-harris} then completes the proof.
\end{proof}

\subsection{Proof of Theorem~\ref{thm:existence-uniqueness-of-invariant-measure}}
\label{sec:proof-existence-uniqueness-of-invariant-measure}

  We present a direct proof of the Feller property showing that for any sequence $u_j \to u$
  and any function $\varphi \in C_b(\H)$ we have that $\P \varphi(u_j) \to \P \varphi(u)$.
  The main difficulty in the proof is the fact that the kernel $\K(u, \cdot)$ depends
  on the point $u$ in a non-trivial manner. To make matters more complicated we
  have to deal with integrals of the form $\int_\H \varphi(x) \alpha(u_j, x) \K(u_j, \dd x)$
  that we wish to show converge to $\int_\H \varphi(x) \alpha(u, x) \K(u, \dd x)$; that
  is both the integrand and the measure depend on the sequence $u_j$ and so
  the dominated convergence theorem cannot be applied directly. 
  However, by Assumption~\ref{assumptions-on-K-beta}(c)
  we know that as $u_j \to u$ we can construct a coupling $\varpi_{u_j, u}$
  of the random variables $\zeta_j \sim \K(u_j, \cdot)$
  and $\zeta \sim \K(u, \cdot)$ in such way that $\zeta_j \to \zeta$ a.s. This yields
  the weak convergence of $\varpi_{u_j, u}$ to the trivail coupling
  $(\rm{Id} \times \rm{Id} )_\sharp \K(u, \cdot)$.
  Using this property,
  the boundedness of $\varphi$,  Lipschitz continuity of $\alpha$ due to
  Assumption~\ref{assumptions-on-Phi}(d), and more standard
  applications of dominated convergence theorem we can then prove the
  desired result. 

  \begin{proof}
  Let $\varphi \in C_b(\H)$, our goal is to show that $\P \varphi \in C(\H)$.
  By   \eqref{lazy-chain-kernel} we have that
  \begin{equation*}
    \begin{aligned}
      \P \varphi (u)
      &= \int_\H \int_\H \varphi( \zeta + \xi) \alpha(u, \zeta + \xi)  \K(u, \dd \zeta)
      \lambda(\dd \xi)  \\
      & \quad + \varphi(u) \int_\H (1 - \alpha(u, \zeta + \xi)) \K(u, \dd \zeta) \lambda(\dd \xi) \\
      & \quad =: T_1(u)  +  T_2 (u).
    \end{aligned}
  \end{equation*}
  In order to prove that $\P$ is Feller we need to show that $T_1, T_2$ are
  continuous. 
  We establish this for $T_1$, as it is the more complicated of the two functions, the argument
  for $T_2$ will follow from very similar steps but simpler since the function $\varphi$ appears
  outside of the integral.

  Let
  $\{u_j\}$ be a sequence of points in $\H$ converging to $u$ and let $\varpi_{u_j, u}$ be
  the coupling in Assumption~\ref{assumptions-on-K-beta}(c) between $\K(u_j, \cdot)$ and
  $\K(u, \cdot)$. We then have, for fixed $\xi \in \H$  that 
  \begin{equation}
    \begin{aligned}
      & \left| \int_\H \varphi(\zeta_j + \xi) \alpha(u_j, \zeta_j + \xi) \K(u_j, \dd \zeta_j)
        - \int_\H \varphi(\zeta + \xi) \alpha(u, \zeta + \xi) \K(u, \dd \zeta) \right| \\
      & = \left| \int_{\H \times \H} \varphi(\zeta_j + \xi) \alpha(u_j, \zeta_j + \xi) -
        \varphi(\zeta + \xi) \alpha(u, \zeta + \xi) \varpi_{u_j,u}(\dd \zeta_j, \dd \zeta) \right| \\
      & \le \int_{\H \times \H} \left| \varphi(\zeta_j + \xi) \right|
      \left| \alpha( u_j, \zeta_j + \xi) - \alpha(u, \zeta + \xi) \right| \varpi_{u_j, u}( \dd \zeta_j, \dd \zeta) \\
      & \quad + \left| \int_{\H \times \H}  \left[ \varphi(\zeta_j + \xi) - \varphi(\zeta + \xi) \right] 
      \alpha(u, \zeta + \xi) \varpi_{u_j, u}( \dd \zeta_j, \dd \zeta) \right| \\
      & \le \| \varphi \|_{\infty} \int_{\H \times \H} 
      \left| \alpha( u_j, \zeta_j + \xi) - \alpha(u, \zeta + \xi) \right| \varpi_{u_j, u}( \dd \zeta_j, \dd \zeta) \\
      & \quad + \int_{\H \times \H} \left| \varphi(\zeta_j + \xi) - \varphi(\zeta + \xi) \right|
      \varpi_{u_j, u}( \dd \zeta_j, \dd \zeta) \\
      & =: T'_j(\xi) + T''_j(\xi),
    \end{aligned}
  \end{equation}
  where in the last inequality we used $\| \varphi \|_{\infty} := \sup_{u \in \H} | \varphi(u) | < + \infty$
  since $\varphi \in C_b(\H)$ in the first integral and also the fact that $\alpha$ is positive and bounded by 1 by definition, in the second integral. We now aim to show that 
  $\int_\H T'_j(\xi) +  T''_j(\xi) \lambda(\dd \xi) \to 0$ as $u_j \to u$ implying that $| T_1(u_j) - T_1(u) |
  \to 0$.

  By Assumption~\ref{assumptions-on-Phi}(d) the function $\alpha$ is Lipschitz in both of its arguments.
  In fact,
  \begin{equation*}
    \begin{aligned}
      | \alpha(u, v) - \alpha(w, z) |
      & \le | \alpha(u, v) - \alpha(w, v) | + | \alpha(w, v) - \alpha(w, z) | \\
      & \le | \Psi(u) - \Psi(w) |  + | \Psi(v) - \Psi(z) | \\
      & \le L ( 1 \vee \| u\|^q \vee \| w\|^q \vee \| v\|^q \vee \| z\|^q) ( \| u -w\| + \| v - z\|).
  \end{aligned}
\end{equation*}
Using the above bound together with Assumption~\ref{assumptions-on-K-beta}(a, d)
and Lemma~\ref{lp-generalized-triangle-inequality} we can write
\begin{equation}\label{Feller-proof:bound-on-T-prime-j}
  \begin{aligned}
     \frac{1}{\| \varphi\|_\infty}  T'_j(\xi)
     &\le L \int_{\H \times \H} ( 1 \vee \| u\|^q \vee \| u_j\|^q \vee \| \zeta + \xi\|^q
     \vee \| \zeta_j + \xi \|^q ) \\&  \qquad  \times  \left[ \| u_j - u\|
      + \| \zeta_j - \zeta \| \right] \varpi_{u_j, u}( \dd \zeta_j, \dd \zeta) \\
    &< 2^{q+2} L  ( 1 +  \| u\|^q + \| u_j\|^q +  \| \xi\|^q ) \| u_j - u\|.
  \end{aligned}
\end{equation}
Thus, integrating with respect to $\lambda$ and using the hypothesis that $\lambda$ has
bounded moments of degree $q$ we obtain the bound 
\begin{equation*}
  \begin{aligned}
  \frac{1}{\| \varphi\|_\infty}\int_\H T_j'(\xi) \lambda(\dd \xi)
  &< 2^{q + 2} L \| u_j - u\| \left(1 + \| u\|^q + \| u_j \|^q + \int_\H \| \xi \|^q \lambda(\dd \xi) \right)\\
  &\le 2^{q + 2} L (C + \| u\|^q + \|u_j\|^q) \| u_j - u\|,
\end{aligned}
\end{equation*}
for some constant $C >0 $. From this bound we deduce that 
\begin{equation*}
   \int_\H T_j'(\xi) \lambda(\dd \xi) \to 0, \qquad \text{as} \qquad u_j \to u.
\end{equation*}

Now consider $T''_j(\xi)$. Let $\varpi^\ast_u = (\rm{Id} \times \rm{Id})_\sharp \K(u, \cdot)$ be the trivial coupling obtained by drawing $\zeta' \sim \K(u, \cdot)$, setting $\zeta'' = \zeta'$ and $\varpi_u^\ast
= \Law\{ (\zeta', \zeta'')\}$.  Our first
step is to show that $\varpi_{u_j, u} \weakto \varpi_u^\ast$; converges  in the weak sense.
Let us equip the product space $\H \times \H$ with the norm $\| (u,v) \| := \| u\| \vee \| v\|$ and let 
$\varphi' \in \text{Lip}_1(\H \times \H) \cap C_b(\H \times \H)$. We then have that 
\begin{equation*}
  \begin{aligned}
    & \left|
      \int_{\H \times \H} \varphi'(\zeta_j, \zeta) \varpi_{u_j, u}( \dd \zeta_j, \dd \zeta)
      - \int_{\H \times \H} \varphi'(\zeta', \zeta'') \varpi^\ast_{u}( \dd \zeta', \dd \zeta'')
    \right|  \\
    &  \le 
    \int_{\H \times \H} | \varphi'(\zeta_j, \zeta) - \varphi'(\zeta, \zeta) |
    \varpi_{u_j, u}( \dd \zeta_j, \dd \zeta) \\
    &\quad +  \left| \int_{\H \times \H} \varphi'(\zeta, \zeta)
      \varpi_{u_j, u} (\dd \zeta_j, \dd \zeta)
      - \int_{\H \times \H} \varphi'(\zeta', \zeta'') \varpi^\ast_{u}( \dd \zeta', \dd \zeta'')
    \right|  \\
    & =  \int_{\H \times \H} | \varphi'(\zeta_j, \zeta) - \varphi'(\zeta, \zeta) |
    \varpi_{u_j, u}( \dd \zeta_j, \dd \zeta) \\
    &\quad +  \left| \int_{\H \times \H} \varphi'(\zeta, \zeta)
      \K(u, \dd \zeta)
      - \int_{\H \times \H} \varphi'(\zeta', \zeta') \K( u, \dd \zeta')
    \right|\\
    & = \int_{\H \times \H} | \varphi'(\zeta_j, \zeta) - \varphi'(\zeta, \zeta) |
    \varpi_{u_j, u}( \dd \zeta_j, \dd \zeta).
  \end{aligned}
\end{equation*}
Since $\varphi'$ has Lipschitz constant 1 we further have
\begin{equation*}
  \begin{aligned}
     \int_{\H \times \H} | \varphi'(\zeta_j, \zeta) - \varphi'(\zeta, \zeta) |
    \varpi_{u_j, u}( \dd \zeta_j, \dd \zeta)
    & \le \int_{\H \times \H} \| \zeta_j - \zeta \|
    \varpi_{u_j, u}( \dd \zeta_j, \dd \zeta) \\
    & \le \| u_j - u\|,
  \end{aligned}
\end{equation*}
where the last
inequality follows from Assumption~\ref{assumptions-on-K-beta}(d). 
Since $\varphi'$ was arbitrary an application of  Portmantheau theorem (see for example
\cite[2.2.6]{bogachev2018weak}) yields the weak convergence of $\varpi_{u_j, u}$ to $\varpi_{u}^\ast$. 
Returning to the definition of $T''_j(\xi)$ and recalling that $\varphi\in C_b(\H)$,
we have
for any fixed $\xi \in \H$ that $T''_j(\xi) \to \int_{\H \times \H} | \varphi(\zeta' + \xi) -
\varphi(\zeta + \xi) | \varpi_{u}^\ast(\dd \zeta', \dd \zeta) = 0$ as $u_j \to u$, i.e., the 
$T''_j$ converge to $0$ pointwise. The boundedness of $\varphi$ also yields the
boundedness of $T''_j$. An application of the dominated convergence theorem
then yields $\int_\H T''_j(\xi) \lambda(\dd \xi) \to 0$ as desired. 




 \end{proof}

\section{Proof of perturbation results from Section~4}\label{sec:proof-main-approx-results}

\subsection{Proof of Lemma~\ref{tilde-d-q-generalized-triangle-inequality}}
\label{sec:proof-tilde-d-q-generalized-triangle-inequality}
\begin{proof}
Observe that by Jensen's inequality it is sufficient to show there
exists $G' >0$ so that
\begin{equation*}
  \begin{split}
    d_q(u,v) ( 2 + \theta V(u) + \theta V(v)) 
    &\le
  G' \big( d_q(u,w) ( 2 + \theta V(u) + \theta V(w)) \\
    & +  d_q(w,v) ( 2 + \theta V(w) + \theta V(v)) \big).
\end{split}
\end{equation*}

  Furthermore, by the hypothesis on $V$ we have that
  $d_q(u,v) ( 2 + \theta V(u) + \theta V(v))$ is equivalent to $d_{p +q}(u,v)$.
In fact,
\begin{equation*}
  \begin{split}
  &d_q(u,v) ( 2 + \theta V(u) + \theta V(v)) \\
  & = \frac{1}{\omega} \left( 2+ \theta  \sum_{j=0}^p a_j \left(\|u\|^j + \|v\|^j\right) \right)(1 + \eta \| u\| + \eta \| v\|)^q
  \| u -v\| \\
  & \le \frac{C}{\omega} ( 1 + \eta \| u\| + \eta \| v\|)^{p +q} \| u -v\| = C d_{p +q}(u,v).
\end{split}
\end{equation*}
with $C(\theta, \eta, p, a_j) >0$. 
Conversely, by Lemma~\ref{lp-generalized-triangle-inequality} and the assumption
 that $a_p >0$ we have
\begin{equation*}
  \begin{split}
    d_{p + q}(u,v)
    &= \frac{1}{\omega} ( 1 + \eta \| u\| + \eta \| v\|)^{p +q} \| u -v\| \\
    &\le \frac{2^{2p}}{\omega} (1 + \eta^p \| u\|^p + \eta^p \| v\|^p)
    ( 1 + \eta \| u\| + \eta \| v\|)^{q} \| u -v\|\\
    &\le \frac{2^{2p} C' }{\omega} \left(1 + \theta \sum_{j=0}^p a_j \left(  \| u\|^j + \| v\|^j \right) \right)
    ( 1 + \eta \| u\| + \eta \| v\|)^{q} \| u -v\| \\
    & =  c d_q(u,v)( 2+ \theta V(u) + \theta V(v) ),
\end{split}
\end{equation*}
where once again $c(\theta, \eta, p, a_j) >0$. Thus it suffices if we prove
the generalized triangle inequality for the $d_s(u,v)$ semimetrics with $s \in \mbb N$.

Let $u,v,w \in \H$
and observe that  if either $d_q(u,w)$ or
$d_q(w,v)$ is equal to one then
\begin{equation*}
  d_q(u,v) \le d_q(u,w) + d_q(w, v),
\end{equation*}
since $d_q(u,v)$ is capped at one, so the standard triangle inequality holds.
Now suppose both $d_q(u,w)$ and $d_q(w,v)$ are less  than one,
implying that $\| u -w\| < \omega$ and $\| w - v\| < \omega$.
Then Lemma~\ref{lp-generalized-triangle-inequality} and multiple
applications of the triangle
inequality we can write
\begin{equation*}
  \begin{split}
    \omega d_q(u,v)
    & = ( 1+ \eta \| u\| + \eta \| v\|)^q \| u -v \| \\
    &\le ( 1 + \eta \| u\| + \eta \| w\| + \eta \| v -w\| )^q \| u -w \| \\
    & \quad + (1 + \eta \| u -w \| + \eta \| w\| + \eta \| v\| )^q \| w- v\| \\
    &\le ( 1 + \eta \| u\| + \eta \| w\| + \eta \omega )^q \| u -w \| \\
    & \quad + (1 + \eta \omega + \eta \| w\| + \eta \| v\| )^q \| w- v\| \\
    &\le (1 + \eta \omega)^q
    \Big( ( 1 + \eta \| u\| + \eta \| w\| )^q \| u -w \| \\
    & \qquad \qquad \qquad  + (1  + \eta \| w\| + \eta \| v\| )^q \| w- v\| \Big).
  \end{split}
\end{equation*}
Thus, we have
\begin{equation*}
  d_q(u,v) \le (1 + \eta \omega)^q \left( d_q(u,w) + d_q(w, v) \right).
\end{equation*}
\end{proof}

\subsection{Proof of Theorem~\ref{thm:approximation-result-in-td}}
\label{sec:proof-perturbation-tilde-d}
  Our strategy is to take 
  $k > n$ sufficiently large that   $G \gamma^{\lfloor k/n \rfloor} < 1$, where
  $\gamma$ is the $n$ step spectral gap of $\mcl P_0$ in
  Assumption~\ref{assumptions-on-P-and-td}(b) and $C >0$ is the Lipschitz constant of $\mcl P^{n}$.
 With $k$ as above we then
  take $\vartheta(k) =  C \sum_{j=1}^k G^j (C^\ast \gamma)^{\lfloor j/n  \rfloor} < \infty$ with $C^\ast, C >0$
  constants depending on the Lipschitz constant of $\P_0$ and growth of the
  Lyapunov function $V$ and show that $\tilde d(\mcl P_\veps^k \delta_u, \mcl P_0^k \delta_u)
  \le \vartheta(k) \psi(\veps) ( 1+ \sqrt{V(u)})$. The remainder of the argument then generalizes 
  this bound for point masses $\delta_u$ to  a bound on $\tilde d(\mcl P_\veps^k\mu_1,\mcl P_0^k \mu_2)$
  for general measures $\mu_1, \mu_2$ and in turn for the invariant measures $\nu_0, \nu_\veps$.

  Before presenting the main proof we need an auxiliary lemma stating that $\P_0^k$ is $\td$-Lipschitz
  in the initial condition of the chain.

  \begin{lemma}\label{lem:P-0-is-Lipschitz-in-initial-cond}
    Suppose Assumption~\ref{assumptions-on-P-and-td}(a,b,c) hold. Then
    for any integer $k >0$  there
    exists a constant $C^\ast(k) >0$, so that
    \begin{equation*}
      \td(\P^k_0 \delta_u, \P_0^k \delta_v) \le C^\ast \td(u,v), \qquad \forall u,v \in \H. 
    \end{equation*}
  \end{lemma}

  \begin{proof}
    We consider two cases where $d(u,v) =1$ and $d(u,v) <1$.\\
    Case 1: Suppose  $d(u,v) = 1$. Then, letting $\pi_{u,v} \in \Upsilon( \P_0^k \delta_u, \P_0^k \delta_v)$, we have,
 using the Lyapunov condition and Jensen's inequality
 \begin{align*}
   \tilde d( \P_0^k \delta_u,  \P_0^k \delta_v)^2
   &\le \int_{\H \times \H} \tilde d^2(x,y) \pi_{u,v}(dx,dy) \\
 &= \int_{\H \times \H} (2 + \theta V(x) + \theta V(y)) \pi_{u,v}(dx,dy) \\
 &\le 2 + \theta \kappa^k (V(u) + V(v)) + \frac{\theta K}{1-\kappa}, \\
   \tilde d( \P_0^k \delta_u,  \P_0^k \delta_v)
   & \le \sqrt{d(u,v)} \sqrt{2 + \theta V(u) + \theta V(v)}
                                              + \sqrt{d(u,v)} \sqrt{\frac{\theta K}{1-\kappa}} \\
 &\le C^\ast \tilde d(u,v),
 \end{align*}
 where $C^*(k)$ is a universal constant that does not depend on $u,v$. 
 
 Case 2: If $d(u,v)<1$, then because $\P_0$ is contracting for $d$
 it follows
  that $\P_0^k$ is contracting for  $d$, and so
 \begin{align*}
   \tilde d &( \P_0^k \delta_u,  \P_0^k \delta_v)^2 \\
   &\le \inf_{\pi_{u,v}} \int_{\H \times \H} d(x,y) \pi_{u,v}(dx,dy) \int_{\H \times \H} (2 + \theta V(x) + \theta V(y)) \pi_{u,v}(dx,dy) \\
 &\le \gamma_1 d(u,v) \left[ 2 + \theta \kappa^k (V(x) + V(y)) + \frac{\theta K}{1-\kappa} \right] \\
   \tilde d &( \P_0^k \delta_u, \P_0^k \delta_v) \\
   &\le \sqrt{d(u,v)} \left[ \sqrt{2 + \theta V(u) + \theta V(v)} + \sqrt{\frac{\theta K}{1-\kappa}} \right] \\
 &\le C^* \tilde d(u,v),
 \end{align*}
 where $C^\ast(k) >0$ is the same constant as in Case 1.
  \end{proof}

\begin{proof}[Proof of Theorem~\ref{thm:approximation-result-in-td}]

  Suppose initially that there exists a positive function $\vartheta(k)$ such that for every $k> 0$
 \begin{equation*}
 \tilde d( \P^k_\veps \delta_u,  \P_0^k \delta_u) \le \vartheta(k) \psi(\veps)\left(1+\sqrt{V(u)}\right).
\end{equation*}
 We will  show below that this is implied by the one-step error control
 and Lemma~\ref{lem:P-0-is-Lipschitz-in-initial-cond}.
 For any $k > n$ we have by the weak triangle inequality
 \begin{equation}\label{pointwise-bound-between-P-eps-P-0}
   \tilde d( \P^k_\veps \delta_u,  \P_0^k \delta_v) \le G \gamma^{\lfloor k/n \rfloor}
   \tilde d(u,v) + G \psi(\veps) \vartheta(n) \left(1+\sqrt{V(u)}\right).
 \end{equation}
 Choose $k$ large enough that $\gamma^{\lfloor k/n \rfloor} < G^{-1}$ and put $\gamma^* = \gamma^{\lfloor k/n \rfloor} G<1$. By Remark~\ref{remark:point-mass-to-measure-generalization} the bound in
   \eqref{pointwise-bound-between-P-eps-P-0} can be generalized to 
  any two probability measures $\mu_1,\mu_2 \in P^1(\H; \td)$ and so
 \begin{equation*}
   \tilde d( \P^k_\veps \mu_1, \P_0^k \mu_2) \le \gamma^* \tilde d(\mu_1,\mu_2)
   + G \psi(\veps) \vartheta(n) \left(1+ \int_\H \sqrt{V} \dd \mu_1 \right),
 \end{equation*} 
 for some $\gamma^* < 1$. The integral in the last term appears after integrating
 the right hand side of \eqref{pointwise-bound-between-P-eps-P-0}
 with respect to the optimal coupling of $\mu_1, \mu_2$.
Using the symmetry of $\td$ and 
by putting $\mu_1 = \nu_\veps$ and
 $\mu_2 = \nu_0$ and vice versa we have
 \begin{align*}
   \tilde d(\nu_0,\nu_\veps) &\le \frac{G \psi(\veps) \vartheta(n)}{1-\gamma^*}
                               \left(1+\nu_0\left( \sqrt{V} \right) \wedge \nu_\veps\left( \sqrt{V} \right) \right).
 \end{align*}
 
It remains to show the existence of $\vartheta(n)$. By Lemma~\ref{lem:P-0-is-Lipschitz-in-initial-cond}
and  \eqref{eq:OneStepError} we have
 \begin{align*}
   \tilde d(\P_\veps^k \delta_u , &  \P_0^k \delta_u) \\
   &\le G \left\{ C^\ast \gamma^{\lfloor k/n \rfloor} \tilde d( \P_0^{k-1} \delta_u,  \P_\veps^{k-1} \delta_u) +
   \psi(\veps) \left[ 1+ (\P_\veps^{k-1} \delta_u) \left( \sqrt{V} \right) \right] \right\} \\
   &\le \psi(\veps) \sum_{j=1}^k G^j (C^\ast \gamma)^{\lfloor j/n \rfloor}
\left[ 1+ (\P_\veps^{k-1} \delta_u) \left( \sqrt{V} \right) \right] \\
   &\le \psi(\veps) \sum_{j=1}^k G^j (C^\ast \gamma)^{\lfloor j/n \rfloor}
     \left(1+ \kappa_\veps^{(k-j)/2} \sqrt{V(u)}
     + \frac{\sqrt{K_\veps}}{1-\sqrt{\kappa_\veps}} \right) \\
 &\equiv \psi(\epsilon) \vartheta(k) (1+  \sqrt{V(u)} ),
 \end{align*}
 where $\vartheta(k) < \frac{\sqrt{K_\veps} + 1}{1-\sqrt{\kappa_\veps}}
 \sum_{j=1}^{k} G^j (C^\ast \gamma)^{\lfloor j/n \rfloor}$ 
and $C^\ast$ is the 
 constant in Lemma~\ref{lem:P-0-is-Lipschitz-in-initial-cond}.
 
\end{proof}

\subsection{Proof of Theorem~\ref{thm:VariationBound} }
\label{sec:proof-variation-bound-poisson}
Our strategy employs the Poisson equation and Martingale/potential methods. The argument is complicated by the fact that $\tilde d$ is not a metric, and we seek to prove bounds for $\varphi : \mc H \to \mc X$ for a separable Hilbert space $\mc X$. This requires us to first show that the potential $\sum_{k=0}^{\infty} \P^k \varphi$ solves the Poisson equation, by checking that the potential converges to a well-defined limit and that $\P$ is a bounded linear operator in an appropriate operator norm. We then are able to use the inner product and norm on $\mc H$ to make a Martingale argument reminiscent of that in \cite{GlynnMeyn1996}.

We prove three preparatory Lemmas that are used in the main proof. 
In what follows we let $\X$ be separable Hilbert space with norm $\| \cdot \|_\X$.

\begin{lemma} \label{lem:BanachLip}
Suppose there exists a $C < \infty$ and $k \in \bb N$ such that 
\be
\tilde d(\P^k_0 \delta_u ,  \P^k_0 \delta_v) \le C \tilde d(u,v).
\ee
Then for any $\varphi : \mc H \to \mc X$ with $\lvertiii \varphi \rvertiii_{\td} < \infty$,
\be
\| \P^k_0 \varphi (u) - \P^k_0 \varphi (v)\|_{\X} \le C \lvertiii \varphi \rvertiii_{\td} \tilde d(u,v).
\ee
\end{lemma}
\begin{proof}
Since $\X$ is a Hilbert space,
\be
\|\varphi(u) - \varphi(v)\|_{\X} \ge | \|\varphi(u)\|_{\X} - \|\varphi(v)\|_{\X}|
\ee
and so
\begin{equation} \label{eq:XNormCompare}
\lvertiii \varphi \rvertiii_{\td} 
= \sup_{u \ne v} \frac{\|\varphi(u) - \varphi(v)\|_{\X}}{\tilde d(u,v)} 
\ge \sup_{u \ne v} \frac{|\|\varphi(u)\|_{\X} - \|\varphi(v)\|_{\X}|}{\tilde d(u,v)} = \lvertiii \: \|\varphi\|_{\X} \: \rvertiii_{\td}.
\end{equation}
So then
\be
\big\| \int_\H \varphi(x)  &(\P^k_0 \delta_u  -  \P^k_0 \delta_v)(dx) \big\|_{\X} \\
& \le \int_\H \|\varphi(x)\|_{\X} (\P^k_0 \delta_u  -  \P^k_0 \delta_v)(dx) \\
&\le \lvertiii \varphi \rvertiii_{\td} \inf_{\pi_{u,v} \in \Upsilon(\P^k_0 \delta_u, \P^k_0 \delta_v)} \int_{\H \times \H} \tilde d(x,y) \pi_{u,v}(dx,dy) \\
&\le \lvertiii \varphi \rvertiii_{\td} C \tilde d(u,v),
\ee
where the last inequality follows from the hypothesis of the lemma.
\end{proof}

This implies immediately that $\P_0^k$ is a $\lvertiii \cdot \rvertiii_{\td}$ contraction.
\begin{corollary}
Let $k$ be the smallest integer for which $\tilde d(\P^k_0 \delta_u,\P^k_0 \delta_v) < \gamma \tilde d(\delta_u,\delta_v)$, then
Lemma \ref{lem:BanachLip} immediately implies that
\be
\| \P^{kn}_0 \nu_1 \varphi - \P^{kn}_0 \nu_2 \varphi\|_{\X} \le \gamma^n \lvertiii \varphi \rvertiii_{\td} \tilde d(\nu_1,\nu_2).
\ee
Furthermore, for $j < k$, combining Lemmas \ref{lem:BanachLip} and \ref{lem:P-0-is-Lipschitz-in-initial-cond}, we obtain
\be
\| \P^{kn}_0 \nu_1 \varphi - \P^{kn}_0 \nu_2 \varphi\|_{\X} \le C^* \lvertiii \varphi \rvertiii_{\td} \tilde d(\nu_1,\nu_2).
\ee
\end{corollary}

Finally, we show that for $\lvertiii \cdot \rvertiii_{\td}$-Lipschitz $\varphi$, the potential $\Theta^*: \H \to \mc X$ of $\varphi$ is well-defined,
has bounded $\lvertiii \cdot \rvertiii_{\td}$ seminorm, and is a solution to the Poisson equation for $\varphi$.

\begin{lemma} \label{lem:PotentialLipschitz}
Consider $\varphi : \H \mapsto \X$ which is $\nu_0$-Bochner-measurable and with $\lvertiii \varphi \rvertiii_{\td} < +\infty$.  Define 
$\widetilde \varphi = \varphi - \nu_0( \varphi)$ and the potential function
\be
  \label{eq:Theta-ast-def}
  \Theta^* := \sum_{j=0}^\infty \P_0^j \widetilde \varphi.
\ee
If $\P_0$ satisfies Assumption~\ref{assumptions-on-P-and-td}  it holds true that: 
\begin{enumerate}[label=(\alph*)]

    \item There exists a uniform constant $C_0 >0$ so that
  \begin{equation*}
    \lvertiii \Theta^\ast \rvertiii_{\td} < \frac{C_0 \lvertiii \varphi \rvertiii_{\td}}{1 - \gamma}.
  \end{equation*}
  \item  $\Theta^*$ is a solution to the Poisson equation 
  \be
  (\P_0 - I) \hat \Theta = - \widetilde{\varphi}.
  \ee
  \item $\Theta^* :\H \to \X $ is well-defined pointwise.
  \end{enumerate}
\end{lemma}
\begin{proof}
Let $k$ be the smallest integer such that for all
$u,v \in \mc X$, $\tilde d( \P_0^k \delta_u, \P^k_0 \delta_v) < \gamma \tilde d(u,v)$
for some $\gamma < 1$, which is finite because $\P_0$ satisfies  Assumption~\ref{assumptions-on-P-and-td}(f).
 \begin{equation}\label{eq:lemma-7-display}
   \begin{aligned}
   \sum_{j=0}^{\infty} \P_0^j \tilde \varphi
   &= \sum_{j=0}^{k-1} \P_0^j \tilde \varphi
     + \sum_{j=0}^{k-1} \sum_{i=1}^\infty \P_0^{i k+j} \tilde \varphi, \\
   \lvertiii \sum_{j=0}^{\infty} \P_0^j \tilde \varphi \rvertiii_{\td}
   &\le \sum_{j=0}^{k-1} \lvertiii \P_0^j \tilde \varphi \rvertiii_{\td}
     + \sum_{j=0}^{k-1} \sum_{i=1}^\infty \lvertiii \P_0^{i k+j} \tilde \varphi \rvertiii_{\td} \\
   &\le k C^* \lvertiii \varphi \rvertiii_{\td}
     + \frac{k \lvertiii \varphi \rvertiii_{\td}}{1-\gamma} = k\left( C^* + \frac1{1-\gamma} \right) \lvertiii \varphi \rvertiii_{\td}.
    \end{aligned}
 \end{equation}
 where the last line followed by observing $\lvertiii \tilde \varphi \rvertiii_{\td} = \lvertiii \varphi \rvertiii_{\td}$
 and applying Lemmas \ref{lem:BanachLip} and \ref{lem:P-0-is-Lipschitz-in-initial-cond}. This concludes the proof of (a).
 
 To prove (b)
 consider the space $L_1(\X,\nu_0;\X)$ of  $\nu_0$-Bochner-measurable functions $f : \H \to \mc X$ satisfying $\nu_0(\|f ( \cdot) \|_{\X}) < \infty$, equipped with the norm 
 \begin{equation}
 \| f \|_{L_1(\nu_0)} = \int_\H \|f(u)\|_{\X} \nu_0(\dd u).    
 \end{equation} \label{eq:L1Norm}
 We now show that the series in \eqref{eq:Theta-ast-def} converges in $L_1(\X,\nu_0; \X)$. By \eqref{eq:XNormCompare}, $\lvertiii \varphi \rvertiii_{\td} > \lvertiii \: \|\varphi\|_{\X} \: \rvertiii$. Notice that since $\lvertiii \varphi \rvertiii_{\td} = \lvertiii \varphi - \varphi(0) \rvertiii_{\td}$, it follows that 
 \be
 \|\varphi(u) - \varphi(0) \|_{\X} = \|\varphi(u)\|_{\X} \le \lvertiii \varphi \rvertiii_{\td} \sqrt{2 + \theta V(u)},
 \ee
 since $\tilde d(u,v) < \sqrt{2 + \theta V(u) + \theta V(v)}$ and $V(0) = 0$. Since $\nu_0 (V) < \infty$, we have that for any $\varphi$ with $\lvertiii \varphi \rvertiii_{\td}< \infty$
 \begin{equation} \label{eq:L1BoundLipNorm}
 \| \varphi \|_{L_1(\nu_0)} = \nu_0 \left( \|\varphi (\cdot) \|_{\X} \right) \le \lvertiii \varphi \rvertiii_{\td} (\sqrt{2} + \theta \nu_0 \sqrt{V}).
 \end{equation}
 Since $L_1(\H, \nu_0; \X)$ is a Banach space \cite[pg. 2]{pisier2016martingales}, \eqref{eq:L1BoundLipNorm} means it is enough to show that the sequence of partial sums $\Theta_m = \sum_{j=0}^m \P_0^j \varphi$ is $\lvertiii \cdot \rvertiii_{\td}$-Cauchy, since this also implies it is $L_1(\nu_0)$-Cauchy. Define $\ell = \lfloor m/k \rfloor$, and $\tilde n = \lceil (n-m)/k \rceil$, so for $n>m$
 \be
 \lvertiii \Theta_n - \Theta_m \rvertiii_{\td} &= \lvertiii \sum_{j=m+1}^n \P^j_0 \tilde \varphi \rvertiii_{\td} 
 \le \sum_{j=m+1}^n \lvertiii \P^j_0 \tilde \varphi \rvertiii_{\td} 
 \le k \lvertiii \varphi \rvertiii_{\td} \sum_{j=0}^{\tilde n} \gamma^{\ell + j}
 \le \frac{k \gamma^{\lfloor m/k \rfloor}}{1-\gamma} \lvertiii \varphi \rvertiii_{\td},
 \ee
 Therefore, the series in \eqref{eq:Theta-ast-def} converges in $L_1(\nu_0)$. Since $\nu_0$ is the unique invariant measure of $\P_0$, and  
 \be
 \lvertiii \P_0 \rvertiii_{\td}^\ast &:= \sup_{\lvertiii \varphi \rvertiii_{\td} < 1} \left\| \int_\H (\P_0 \varphi)(u) \nu_0(du) \right\|_{\X} \\
 &\le \sup_{\lvertiii \varphi \rvertiii_{\td} < 1} \int_\H \|(\P_0 \varphi)(u)\|_{\X} \nu_0(du) \\
 &< \int_\H (\P_0 \sqrt{2+\theta V})(u) \nu_0(du) < \infty,
 \ee
 so that $\P_0$ is a bounded linear operator on the space of $\text{Lip}(\td)$ functions from  $\mc H$ to $\mc X$ equipped with 
 the operator norm  $\lvertiii \cdot \rvertiii_{\td}^\ast$.  Thus we conclude that $\Theta^*$ is a solution of the Poisson equation 
 for $\nu_0$-Bochner-measurable functions $\varphi \in \text{Lip}(\td)$.
 
Finally, we prove (c) by showing that $\Theta^*(u)$ converges in $\| \cdot \|_{\X}$. We have
\be
\|\Theta_n(u) - \Theta_m(u)\|_{\X} &= \left\| \sum_{j=m+1}^n \P_0^j \widetilde \varphi(u) \right\|_{\X} 
= \left\| \sum_{j=m+1}^n \P_0^j \varphi(u) - \nu_0 (\varphi)  \right\|_{\X} \\
&\le \sum_{j=m+1}^n \|\P_0^j \varphi(u) - \nu_0 (\varphi)\|_{\X} 
\le \lvertiii \varphi \rvertiii_{\td} \frac{k \gamma^{\lfloor m/k \rfloor}}{1-\gamma} \tilde d(\delta_u,  \nu_0).
\ee
So the sequence is $\|\cdot\|_{\X}$-Cauchy for any $u \in \H$.
\end{proof}

We are now ready to present the complete proof of Theorem~\ref{thm:VariationBound}. We primarily
focus on part (a) as part (b) follows as a corollary of the calculations in the proof of (a).

\begin{proof}[Proof of Theorem~\ref{thm:VariationBound}]
   (a)
Define $\Theta = \Theta^* - \Theta^*(0)$ for any $\varphi: \mc H \to \mc X$ with
$\lvertiii \varphi \rvertiii_{\td} < \infty$. By Lemma \ref{lem:PotentialLipschitz}, $\lvertiii \Theta^* \rvertiii_{\td} < \frac{C_0 \lvertiii \varphi \rvertiii_{\td}}{1-\gamma}$ for some $C_0 <+ \infty$, and $\Theta(u)$ is a well-defined element of $\mc X$ for any $u \in \mc H$. So with $C = C_0 \lvertiii \varphi \rvertiii_{\td} (1-\gamma)^{-1}$, we have
\begin{equation} \label{eq:PoissonBound}
  \begin{aligned}
     \|\Theta^*(u) - \Theta^*(v)\|_{\X} &\le C \sqrt{2 + \theta V(u) + \theta V(v)}, \\
     \|\Theta^*(u) - \Theta^*(0)\|_{\X} &=  \|\Theta(u)\|_{\X} \le C \sqrt{2 +
      \theta V(u)},
  \end{aligned}
\end{equation}
and furthermore $\lvertiii \Theta \rvertiii_{\td} = \lvertiii \Theta^* \rvertiii_{\td}$.
 Note that
\begin{equation} \label{eq:Poisson}
(\P_0 - I)\Theta^*(u) = (\P_0-I)\Theta(u) = -\tilde \varphi(u),
\end{equation}
thus,
\begin{align*}
  \Theta^*(U_n^\veps) -\Theta^*(U_0^\veps)&= \sum_{k=0}^{n-1} \Theta^*(U_{k+1}^\veps) -\Theta^*(U_k^\veps) = \sum_{k=0}^{n-1} \Theta(U_{k+1}^\veps) -\Theta(U_k^\veps) \\
  &= \sum_{k=0}^{n-1}  [\Theta(U_{k+1}^\veps) -\mcl{P}_\epsilon \Theta(U_k^\veps) ]
+ \sum_{k=0}^{n-1} (\mcl{P}_\veps -I) \Theta(U_k^\veps), \\
  & = \sum_{k=0}^{n-1}  [\Theta(U_{k+1}^\veps) -\mcl{P}_\veps \Theta(U_k^\veps) ]\\
&\qquad + \sum_{k=0}^{n-1} (\P_0 -I) \Theta(U_k^\veps) + \sum_{k=0}^{n-1} (\mcl{P}_\veps -\P_0)
\Theta(U_k^\veps).
\end{align*}
Using \eqref{eq:Poisson} and that $\Theta^*(U_n^\veps) - \Theta^*(U_0^\veps) = \Theta(U_n^\veps) - \Theta(U_0^\veps)$
and defining  the Martingale increments
$m_{k+1}^\veps =  \Theta(U_{k+1}^\veps) - \mcl{P}_\veps \Theta(U_k^\veps) $ 
and the Martingale $M_n^\veps=\sum_{k=1}^n m_{k}^\veps$,
we have
\begin{equation} \label{eq:PoissonEquation}
  \begin{aligned}
\frac{1}{n} \sum_{k=0}^{n-1} \varphi(U_k^\veps) - \nu_0 (\varphi)  & =
 \frac{\Theta(U_0^\veps) -\Theta(U_n^\veps)}{n}  + \frac{1}{n} M_n^\veps 
+   \frac{1}{n} \sum_{k=0}^{n-1} (\mcl{P}_\veps -\P_0) \Theta(U_k^\veps) 
\\ &=: T_1 + T_2 + T_3.
\end{aligned}
\end{equation}
Note that the quantity we now care about is
\be
\bb E \left\| \frac{1}{n} \sum_{k=0}^{n-1} \varphi(U_k^\veps) - \nu_0 (\varphi) \right\|_{\X} \le \bb E \left[ \|T_1\|_{\X} + \|T_2\|_{\X} + \|T_3\|_{\X}\right].
\ee
Let $\mcl F_k$ be the filtration indexed by time $k$. We have with $\langle \cdot, \cdot \rangle_\X$ the $\X$-inner product,
\be
\|M_n^\veps\|_{\X}^2 &= \langle M_n^\veps,M_n^\veps \rangle_\X = \sum_{k=0}^{n-1} \sum_{j=0}^{n-1} \langle m_k^\veps, m_j^\veps \rangle_\X \\
\bb E \|M_n^\veps\|_{\X}^2 &= \sum_{k=0}^{n-1} \sum_{j=0}^{n-1} \bb E \left[ \bb E\left[\langle m_k^\veps, m_j^\veps \rangle_\X \mid \mc F_{j \wedge k} \right]\right] = \sum_{k=0}^{n-1} \bb E\left[\bb E \left[\langle m_k^\veps, m_k^\veps \rangle_\X \mid \mc F_k \right] \right] \\
&= \sum_{k=0}^{n-1} \bb E \left[ \bb E \left[\|m_k^\veps\|^2_{\X} \mid \mc F_k \right] \right] 
\le \sum_{k=0}^{n-1} \bb E \left[ \bb E \left[C^2(2+\theta V(U^\veps_{k+1})) \mid \mc F_k \right] \right] \\
&\le C^2 \left( 2n + \theta \sum_{k=0}^{n-1} \bb E \left[\kappa_{\veps} V(U_k) + K_{\veps} \right] \right) \\ 
&\le C^2 \left( 2n + \theta \sum_{k=0}^{n-1} \kappa^k_{\veps} V(u_0) + \frac{K_{\veps}}{1-\kappa_{\veps}} \right) \le C^2\left( 2n + \theta \frac{V(u_0) + n K_{\veps}}{1-\kappa_{\veps}} \right),
\ee
which in turn implies that
\be
\bb E \|T_2\|_{\X} = n^{-1} \bb E \|M_n^\veps\|_{\X} &\le n^{-1} \left(\bb E \|M_n^\veps\|^2_{\X} \right)^{1/2} \\
&\le n^{-1} \sqrt{n} C \left(2 + \theta \frac{V(u_0)/n + K_{\veps}}{1-\kappa_{\veps}}\right)^{1/2} \\
&=  \frac{C}{\sqrt{n}} \left(2 + \theta \frac{V(u_0)/n + K_{\veps}}{1-\kappa_{\veps}}\right)^{1/2}.
\ee
Next we have
\begin{equation}
\begin{aligned}
\left\| \sum_{k=0}^{n-1} (\mcl{P}_\veps -\P_0) \Theta(U_k^\veps)  \right\|_{\X} 
&\le  \sum_{k=0}^{n-1} \|(\mcl{P}_\veps -\P_0) \Theta(U_k^\veps)\|_{\X} \\ 
& \le \sum_{k=0}^{n-1} C \psi(\veps) (1 + \sqrt{V(U_k^\veps)}), \\
\bb E \left\| \sum_{k=0}^{n-1} (\mcl{P}_\veps -\P_0) \Theta(U_k^\veps)  \right\|_{\X} &\le \sum_{k=0}^{n-1} C \psi(\veps) (1 + \bb E \sqrt{V(U_k^\veps)}) \\
&\le \sum_{k=0}^{n-1} C \psi(\veps) \left(1 + \kappa_\veps^{k/2} V(u_0) + \frac{\sqrt{K_\veps}}{1-\sqrt{\kappa_\veps}} \right) \\
&= C \psi(\veps) n\left( 1 + \frac{\sqrt{\kappa_\veps}/n V(u_0) + \sqrt{K_\veps}}{1-\sqrt{\kappa_\veps}} \right),  \\
\bb E\|T_3\|_{\X} &\le C \psi(\veps) \left( 1 + \frac{\sqrt{\kappa_\veps}/n V(u_0) + \sqrt{K_\veps}}{1-\sqrt{\kappa_\veps}} \right).  \label{eq:T3Bound}
\end{aligned}
\end{equation}
Finally we have
\begin{equation}
    \begin{aligned}
&\|\Theta(U_0^\veps) -\Theta(U_n^\veps)\|_{\X} \le C \sqrt{2 + \theta V(U_0^\veps) + \theta V(U_n^\veps)} \\
& \qquad \le C (\sqrt{2} + \sqrt{\theta} (\sqrt{V(U_0^\veps)} + \sqrt{V(U_n^\veps)}), \\
&\bb E \|T_1\|_{\X} \le \frac{C}{n} \left[ \sqrt{2} + \sqrt{\theta} \left(\sqrt{V(U_0^\veps)} + \kappa_\veps^{n/2} \sqrt{V(U_0^\veps)} + \frac{\sqrt{K_\veps}}{1-\sqrt{\kappa_\veps}} \right) \right] \\
&\qquad  \le \frac{C}{n} \left[ \sqrt{2} + \sqrt{\theta} \left(\sqrt{V(U_0^\veps)}(1+ \kappa_\veps^{n/2}) + \frac{\sqrt{K_\veps}}{1-\sqrt{\kappa_\veps}} \right) \right]. \label{eq:T1Bound}
    \end{aligned}
\end{equation}
Putting together the bounds for $\bb E |T_1|_{\mc X}, \bb E|T_2|_{\mc X}, \bb E|T_3|_{\mc X}$ we arrive at
\be
\bb E & \left\| \frac{1}{n} \sum_{k=0}^{n-1} \varphi(U_k^\veps) - \nu_0 (\varphi) \right\|_{\X} \\
&\le \frac{C}{n} \left[ \sqrt{2} + \sqrt{\theta} \left(\sqrt{V(U_0^\veps)}(1+ \kappa_\veps^{n/2}) + \frac{\sqrt{K_\veps}}{1-\sqrt{\kappa_\veps}} \right) \right] \\
&+ \frac{C}{\sqrt{n}} \left(2 + \theta \frac{V(u_0)/n + K_{\veps}}{1-\kappa_{\veps}}\right)^{1/2} + C \psi(\veps) \left( 1 + \frac{\sqrt{\kappa_\veps}/n V(u_0) + \sqrt{K_\veps}}{1-\sqrt{\kappa_\veps}} \right) \\
&= \frac{C_0 \lvertiii \varphi \rvertiii_{\td}}{1-\gamma} \left( C_1 \psi(\veps) \left(1 + \frac1n \right) + \frac{C_2}{\sqrt{n}} + \frac{C_3}{n}  \right),
\ee
completing the proof of part (a).

We now consider statement (b). We have from \eqref{eq:PoissonEquation}
 \be
   \left\| \E \frac{1}{n} \sum_{k=0}^{n-1} \varphi(U_k^\veps) - \nu_0 (\varphi) \right\|_{\X}
   &=\bigg\| \E \left( \frac{\Theta(U_0^\veps) -\Theta(U_n^\veps)}{n} \right) \\
   & \qquad  + \frac{1}{n}  \E M_n^\veps 
+   \frac{1}{n} \sum_{k=0}^{n-1} \E (\mcl{P}_\veps -\mcl{P}) \Theta(U_k^\veps) \bigg\|_{\X} \\
   & \hspace{-10ex} \le \frac1n \E \|\Theta(U_0^\veps) -\Theta(U_n^\veps) \|_{\X} + \frac1n \sum_{k=0}^{n-1}
     \E \|(\mcl{P}_\veps -\mcl{P}) \Theta(U_k^\veps)\|_{\X}.
\ee
Now we can bound the first term using \eqref{eq:T1Bound}
\be
  \frac1n \E\|\Theta(U_0^\veps) -\Theta(U_n^\veps)\|_\X &\le \frac{C}{n} \left[ \sqrt{2} + \sqrt{\theta} \left(\sqrt{V(U_0^\veps)}(1+ \kappa_\veps^{n/2}) + \frac{\sqrt{K_\veps}}{1-\sqrt{\kappa_\veps}} \right) \right]  \\
&\equiv \frac1n \frac{\lvertiii \varphi \rvertiii_{\td}}{1-\gamma} C_4,
\ee
where once again $C =  \frac{C_0}{1 - \gamma} \lvertiii \varphi \rvertiii_{\td}$, and now using \eqref{eq:T3Bound}
\be
  \frac1n \sum_{k=0}^{n-1} \E \|(\mcl{P}_\veps -\P_0) \Theta(U_k^\veps)\|_{\mc X}
  &\le C \psi(\veps) \left( 1 + \frac{\sqrt{\kappa_\veps}/n V(u_0) + \sqrt{K_\veps}}{1-\sqrt{\kappa_\veps}} \right) \\
  &\equiv \frac{\lvertiii \varphi \rvertiii_{\td}}{1-\gamma} \psi(\veps) \left(C_5 + \frac{C_6}n \right).
\ee
\end{proof}

{}


\section{Proof of results in Section~5}\label{sec:proof-main-application-results}

\begin{proof}[Proof of Proposition~\ref{pcn-perturbation-intermid-prop}]
Repeating the same calculation as in \eqref{coupling-argument-example-1}
 we have for any $u \in \H^1(\Omega)$,
\begin{equation*}
        \td_0(\P \delta_u, \P_\veps \delta_u)^2  \le  d_0(\P \delta_u, \P_\veps \delta_u)  \left[ 2 + \theta \big( (\kappa + \kappa_\veps) V(u) + (K + K_\veps) \big) \right].
\end{equation*}
Now let $\pi_u \in \Upsilon( \P \delta_u, \P_\veps \delta_u)$ obtained as follows: draw $\xi_j \iidsim N(0, (1 - \beta^2) )$ 
and set $v = \sum_{j=0}^\infty a_j \xi_j \phi_j$ and $v_\veps = \sum_{j=0}^\infty a_j^\veps \xi_j \phi_j^\veps$ and 
propose 
\begin{equation*}
    u^\ast = \beta u + v, \qquad u^\ast_\veps = \beta u + v_\veps.
\end{equation*}
Next draw a uniform random variable $\varsigma$, then the first (exact) chain accepts the proposal $u^\ast$ if $\varsigma < \alpha(u, u^\ast)$ 
while the second (perturbed) chain accepts $u^\ast_\veps$ if $\varsigma < \alpha(u, u^\ast_\veps)$. Since this coupling is 
not necessarily optimal we have that 

\begin{equation}\label{RCAR-approx-karhunen-loeve-intermediate-calc}
\begin{aligned}
d_0( \P \delta_u, \P_\veps \delta_u) 
 \le & \E \left[  d_0(u^\ast, u^\ast_\veps) \mathbb P (\text{both chains accept})  \right] \\    
&  + \E \left[  d_0(u^\ast, u) \mathbb P (\text{only first chain accepts})  \right] \\ 
&  + \E \left[  d_0(u, u^\ast_\veps) \mathbb P (\text{only second chain accepts})  \right] \\  
\le & \E \left[  d_0(u^\ast, u^\ast_\veps) \mathbb P (\text{both chains accept})  \right] \\    
&  + \E \left[   \mathbb P (\text{only one chain accepts})  \right].
\end{aligned}
\end{equation}
Since $\Psi$ is globally Lipschitz and $1 \wedge \exp$ is also $1$-Lipschitz  it follows that 
\begin{equation*}
    \begin{aligned}
    \mathbb P & (\text{only one chain accepts}) 
     \le | \Psi(u^\ast) - \Psi(u^\ast_\veps) | \\
     &\le L \| u^\ast - u^\ast_\veps \|_{\H^1(\Omega)}  =  L \| v - v_\veps \|_{\H^s(\Omega)} \\
     & \le L \left( \left\| \sum_{j =0}^\infty a_j \xi_j \left( \phi_j - \phi_j^\veps \right) \right\|_{\H^s(\Omega)} 
     +  \left\| \sum_{j=0}^\infty (a_j - a_j^\veps) \xi_j  \phi_j \right\|_{\H^s(\Omega)} \right) \\
     & \le L \sum_{j=0}^\infty a_j |\xi_j| \| \phi_j - \phi_j^\veps \|_{\H^s(\Omega)} + |a_j - a_j^\veps| |\xi_j| \| \phi_j \|_{\H^s(\Omega)}.
    \end{aligned}
\end{equation*}
Now  by the hypothesis that the sequences  
$\{ a_j \| \phi_j  - \phi_j^\veps\|_{\H^s(\Omega)}\}$ and 
$\{ |a_j - a_j^\veps| \| \phi_j \|_{\H^s(\Omega)}\}$ belong to $\ell^1$,  
Kolmogorov's two series theorem yields  that the above sum converges a.s. Applying Cauchy-Schwarz we can write 
\begin{equation*}
\begin{aligned}
     \mathbb P  & (\text{only one chain  accepts}) \\
     &\le L \left( \sum_{j=0}^\infty a_j | \xi_j|^2 \| \phi_j \|_{\H^s(\Omega)}^2 \right)^{1/2} 
     \left( \sum_{j=0}^\infty  a_j  \frac{ \| \phi_j - \phi_j^\veps  \|_{\H^s(\Omega)}^2}{\| \phi_j \|_{\H^s(\Omega)}^2} \right)^{1/2} \\
     & +  L \left( \sum_{j=0}^\infty a_j^2 | \xi_j|^2   \right)^{1/2} 
     \left( \sum_{j=0}^\infty    \frac{| a_j - a_j^\veps  |^2}{a_j^2} \| \phi_j\|^2_{\H^s(\Omega)} \right)^{1/2},
\end{aligned}
\end{equation*}
from which it follows that 
\begin{equation*}
\begin{aligned}
        \E    \mathbb P & (\text{only one chain accepts}) \\ 
        &\le L \left( \sum_{j=0}^\infty a_j \| \phi_j \|_{\H^s(\Omega)}^2 \E | \xi_j|^2  \right)^{1/2} 
     \left( \sum_{j=0}^\infty  a_j  \frac{ \| \phi_j - \phi_j^\veps  \|_{\H^s(\Omega)}^2}{\| \phi_j\|^2_{\H^s(\Omega)}} \right)^{1/2} \\
     & +  L \left( \sum_{j=0}^\infty a_j^2 \E | \xi_j|^2  \right)^{1/2} 
     \left( \sum_{j=0}^\infty   \frac{ | a_j - a_j^\veps  |^2}{a_j^2} \| \phi_j\|_{\H^s(\Omega)}^2 \right)^{1/2}\\ 
     & \le C_1 \left[  \left( \sum_{j=0}^\infty  a_j  \frac{ \| \phi_j - \phi_j^\veps  \|_{\H^s(\Omega)}^2}{ \| \phi_j \|_{\H^s(\Omega)}^2} \right)^{1/2} 
     + \left( \sum_{j=0}^\infty   \frac{ | a_j - a_j^\veps  |^2}{a_j^2} \| \phi_j\|_{\H^s(\Omega)}^2 \right)^{1/2}  \right].
\end{aligned}
\end{equation*}
Since $\{ a_j \} \in \ell^2$ and $\{ a_j \| \phi_j\|_{\H^s(\Omega)}^2 \} \in \ell^1$. We further have, by a similar calculation as above, that 
\begin{equation*}
\begin{aligned}
    d_0( u^\ast, u^\ast_\veps) 
    & = 1 \wedge \frac{\| u^\ast - u^\ast_\veps \|_{\H^1(\Omega)} }{\omega} = 1 \wedge \frac{\| v^\ast - v^\ast_\veps \|_{\H^1(\Omega)} }{\omega} \\
    & \le 1 \wedge \frac{1}{\omega}  \left( \sum_{j=0}^\infty a_j | \xi_j |^2 \| \phi_j \|_{\H^s(\Omega)}^2  \right)^{1/2} 
     \left( \sum_{j = 0}^\infty  a_j  \frac{\| \phi_j - \phi_j^\veps  \|_{\H^s(\Omega)}^2}{\| \phi_j \|_{\H^s(\Omega)^2}} \right)^{1/2}  \\
    & \quad \quad \: \: +  \frac{1}{\omega} \left( \sum_{j=0}^\infty a_j^2 | \xi_j|^2  \right)^{1/2} 
    \left( \sum_{j=0}^\infty \frac{|a_j - a_j^\veps|^2}{a_j^2} \| \phi_j \|_{\H^s(\Omega)}^2 \right)^{1/2}\\
    & =: 1 \wedge T_1 + T_2.
\end{aligned}
\end{equation*} 
Thus it follows by Markov's inequality that 
\begin{equation*}
    \begin{aligned}
    \E d_0 & (u^\ast, u^\ast_\veps) 
         \le \mathbb{P} \left( T_1 + T_2 \ge 1    \right) + \E T_1 + T_2 \\ 
        & \le 2 \E (T_1 + T_2) \\
        & \le \frac{2C_2}{\omega} \left[  \left( \sum_{j=0}^\infty  a_j  \frac{ \| \phi_j - \phi_j^\veps  \|_{\H^s(\Omega)}^2}{ \| \phi_j \|_{\H^s(\Omega)}^2} \right)^{1/2} 
     + \left( \sum_{j=0}^\infty   \frac{ | a_j - a_j^\veps  |^2}{a_j^2} \| \phi_j\|_{\H^s(\Omega)}^2 \right)^{1/2}  \right].
    \end{aligned}
\end{equation*}
 Substituting the above bounds back into \eqref{RCAR-approx-karhunen-loeve-intermediate-calc}
we obtain 
\begin{equation*}
\begin{aligned}
      d_0(\P \delta_u,&  \P_\veps \delta_u) \\ 
      & \le C \left[  \left( \sum_{j=0}^\infty  a_j  \frac{ \| \phi_j - \phi_j^\veps  \|_{\H^s(\Omega)}^2}{ \| \phi_j \|_{\H^s(\Omega)}^2} \right)^{1/2} 
     + \left( \sum_{j=0}^\infty   \frac{ | a_j - a_j^\veps  |^2}{a_j^2} \| \phi_j\|_{\H^s(\Omega)}^2 \right)^{1/2}  \right].
\end{aligned}
\end{equation*}
\end{proof}

\end{document}